\documentclass[11pt,reqno]{amsart}

\usepackage{amsmath,amssymb,amsthm}
\usepackage{graphicx,psfrag,epsf}
\usepackage{enumerate}
\usepackage{natbib}
\usepackage{url}

\newtheorem{theorem}{Theorem}
\newtheorem{corollary}{Corollary}
\newtheorem{lemma}{Lemma}

\newtheorem{assumption}{Assumption}
\theoremstyle{definition}

\newtheorem{remark}{Remark}
\newtheorem{example}{Example}

\newcommand{\R}{\mathbb{R}}
\newcommand{\mF}{\mathcal{F}}
\newcommand{\mG}{\mathcal{G}}

\newcommand{\Ep}{\mathrm{E}}
\renewcommand{\Pr}{\mathrm{P}}
\renewcommand{\tilde}{\widetilde}
\renewcommand{\hat}{\widehat}

\DeclareMathOperator{\Var}{Var}
\DeclareMathOperator{\Cov}{Cov}

\DeclareMathOperator{\E}{E}

\begin{document}
\title[Uniform confidence bands in deconvolution]{Uniform confidence bands in deconvolution with unknown error distribution}
\thanks{K. Kato is supported by Grant-in-Aid for Scientific Research (C) (15K03392) from the JSPS. We would like to thank seminar participants at Bank of Canada and UIUC, and conference participants at AMES 2017, CMES 2017, IAAE 2017, and New York Camp Econometrics XII for useful comments}
\author[K. Kato]{Kengo Kato}
\author[Y. Sasaki]{Yuya Sasaki}

\date{First arXiv version: August 4, 2016. This version: \today}

\address[K. Kato]{
Graduate School of Economics, University of Tokyo\\
7-3-1 Hongo, Bunkyo-ku, Tokyo 113-0033, Japan.
}
\email{kkato@e.u-tokyo.ac.jp}

\address[Y. Sasaki]{
Department of Economics, Vanderbilt University\\
VU Station B \#351819, 2301 Vanderbilt Place, Nashville, TN 37235-1819, U.S.A.
}
\email{yuya.sasaki@vanderbilt.edu}

\begin{abstract}
This paper develops a method to construct uniform confidence bands in deconvolution when the error distribution is unknown. 
We mainly focus on the baseline setting where an auxiliary sample from the error distribution is available and the error density is ordinary smooth. 
The auxiliary sample may directly come from validation data, or can be constructed from panel data with a symmetric error distribution.
We also present extensions of the results on confidence bands to the case of super-smooth error densities. 
Simulation studies demonstrate the performance of the multiplier bootstrap confidence band in the finite sample. 
We apply our method to the Outer Continental Shelf (OCS) Auction Data and draw confidence bands for the density of common values of mineral rights on oil and gas tracts.
Finally, we present an application of our main theoretical result specifically to additive fixed-effect panel data models.
As an empirical illustration of the panel data analysis, we draw confidence bands for the density of the total factor productivity in a manufacturing industry in Chile.
\end{abstract}


\maketitle

\section{Introduction}
In this paper, we propose a method of uniform inference on the density function of a latent signal $X$ in the measurement error model 
\begin{equation}
Y = X + \varepsilon, \label{eq: model}
\end{equation}
where $X$ and $\varepsilon$ are independent real-valued random variables with unknown densities $f_{X}$ and $f_{\varepsilon}$, respectively. 
An econometrician observes $Y$ in data, but does not observe $X$ or $\varepsilon$. 
The variable $\varepsilon$ represents a measurement error. 
In this model, the density $f_Y$ of $Y$ can be written by the convolution of $f_{X}$ and $f_{\varepsilon}$:
\begin{equation}
f_{Y}(y) = (f_{X} * f_{\varepsilon}) (y) = \int_{\R} f_{X}(x) f_{\varepsilon}(y-x) dx. \label{eq: convolution}
\end{equation}
Deconvolution refers to solving the convolution integral equation (\ref{eq: convolution}) for $f_X$, and the deconvolution problem in econometrics and statistics has concerned with identifying, estimating and making inference on $f_{X}$ from available data. 

The goal of this paper is to develop a multiplier-bootstrap method to construct uniform confidence bands for $f_{X}$ when the error density $f_{\varepsilon}$ is unknown. 
\cite{BiDuHoMu07} provide a condition under which the nonparametric bootstrap method to construct confidence bands is valid when the error density is known.
In light of this result, \cite{BoSa11} ``conjecture that the bootstrap remains consistent when the error distribution needs to be estimated,'' while they are ``not aware of a formal proof of this result.''
In this paper, we do provide a formal proof that the multiplier bootstrap is consistent when the error distribution needs to be estimated.
Furthermore, we do so under (much) milder conditions due to the construction based upon the ``intermediate'' Gaussian approximation and the Gaussian multiplier bootstrap.

Our data requirement is as follows.
We observe a sample $Y_{1},\dots,Y_{n}$ from $f_Y$.
In addition, we assume to observe an auxiliary independent sample $\eta_{1},\dots,\eta_{m}$ from $f_{\varepsilon}$, where $m=m_{n} \to \infty$ as $n \to \infty$. 
This assumption is satisfied in various ways depending on an application of interest.
One case is when administrative data provide only measurement errors but do not disclose a sample of $X$.
Another case is when we have panel data or repeated measurements $(Y^{(1)}, Y^{(2)})$ for a common signal $X$ with errors $(\varepsilon^{(1)}, \varepsilon^{(2)})$, such that the conditional distribution of the one given the other is symmetric -- see Example \ref{ex: canonical} ahead.
The latter case is similar to the model of \cite{HoMa96}.
Under these two representative situations with unknown $f_\varepsilon$, we develop an estimator and confidence bands for $f_X$. 
Our method is based on the deconvolution kernel density estimator \citep{CaHa88,StCa90,Fa91a,Fa91b}; except that we replace the error characteristic function by the empirical characteristic function with the auxiliary sample $\eta_{1},\dots,\eta_{m}$. 

Asymptotic properties of the deconvolution kernel density estimator critically depend on the smoothness of the distributions of $X$ and $\varepsilon$, where two categories of smoothness, ordinary-smooth and super-smooth distributions, are often employed \citep[cf.][]{Fa91a}.\footnote{
Specifically, the difficulty of estimating $f_{X}$ depends on how fast the modulus of the error characteristic function $\varphi_{\varepsilon}(t) = \Ep[e^{it \varepsilon}]$ with $i=\sqrt{-1}$ decays as $|t| \to \infty$, in addition to the smoothness of $f_{X}$.
The faster $|\varphi_{\varepsilon}(t)|$ decays as $|t| \to \infty$, the more difficult estimation of $f_{X}$ will be. 
The error density $f_{\varepsilon}$ is said to be \textit{ordinary-smooth} if $|\varphi_{\varepsilon}(t)|$ decays at most polynomially fast as $|t| \to \infty$, while $f_{\varepsilon}$ is said to be \textit{super-smooth} if $|\varphi_{\varepsilon}(t)|$ decays exponentially fast as $|t| \to \infty$. See \cite{Fa91a}.}
We first consider the case of ordinary-smooth error densities and prove asymptotic validity of the multiplier bootstrap confidence band under mild regularity conditions. 
In this ordinary-smooth case, the auxiliary sample size $m$ need not be large in comparison with $n$. 
Furthermore, we extend the results on confidence bands to the case of super-smooth error densities.
In the super-smooth case, however, we require relatively more auxiliary data ($m/n \to \infty$) for a technical reason.
It is worth pointing out that the multiplier bootstrap confidence band proposed in the present paper is robustly valid for both cases where the error density is ordinary- and super-smooth (although in the latter case we require $m/n \to \infty$), despite the fact that the limit distributions of the supremum deviation of the deconvolution kernel density estimator in general differ between those two cases.

We conduct simulation studies to demonstrate the performance of the multiplier bootstrap confidence band in finite samples. 
The simulation studies show that the simulated coverage probabilities are very close to nominal coverage probabilities even with sample sizes as small as 250 and 500, suggesting practical benefits of our confidence band. 
Following \cite{LiPeVu00}, we apply our method to the Outer Continental Shelf (OCS) Auction Data \citep{HePoBo87} and draw confidence bands for the density of \textit{ex post} values of mineral rights on oil and gas tracts in the Gulf of Mexico.
In the empirical auction literature, obtaining confidence intervals/bands for a deconvolution density with unknown error distribution is of interest \citep[e.g.,][]{Kr11}, and practitioners have implemented nonparametric bootstrap without a theoretical support for its validity.
We draw valid confidence bands for a deconvolution density, and provide statistical support for some qualitative features of the common value density that \cite{LiPeVu00} find visually in their estimate.
Finally, we discuss an application of our methods to additive fixed-effect panel data models.
As an empirical illustration of the panel analysis, we draw confidence bands for the density of the total factor productivity in the food manufacturing industry in Chile using the data set of \cite{LePe03}.

The rest of the paper is organized as follows. 
Section \ref{sec: literature} reviews the related literature.
Section \ref{sec: methodology} presents our methodology of constructing confidence bands for $f_{X}$. 
Section \ref{sec: main} presents the main theoretical results of this paper where we consider the ordinary smooth case. 
Section \ref{sec: numerical simulations} presents the numerical simulations. 
Section \ref{sec: application} presents an empirical application to auction data.
Section \ref{sec: panel} presents an application to panel data and its empirical illustration.
Section \ref{sec: supersmooth} presents extensions of the results on confidence bands to the case of super-smooth error densities. 
Section \ref{sec: conclusion} concludes. 

\section{Relation to the Literature}\label{sec: literature}

The literature related to this paper is broad. 
We refer to books by \cite{Fu87}, \cite{CaRuStCr06}, \cite{Me09} and \citet[][Chapter 5]{Ho09} and surveys by \cite{ChHoNe11} and \cite{Sc16} for general references on measurement error models and deconvolution methods.
Our method builds upon the deconvolution kernel density estimation method, which is pioneered by \cite{CaHa88,StCa90,Fa91a,Fa91b}.
These earlier studies focus on the case where the error density $f_\varepsilon$ is assumed to be known.

The deconvolution problem with unknown $f_\varepsilon$ is studied by \cite{DiHa93,Ef97,Ne97,Jo09,CoLa11,DaReTr16}. Similarly to our paper, these papers assume the availability of auxiliary measurements from the error distribution. 
\cite{HoMa96} and \cite{LiVu98} consider to estimate a deconvolution density from repeated measurements (panel data) of $Y$, instead of assuming measurements from the error distribution \textit{per se}; see also \cite{Ne07}, \cite{DeHaMe08}, and \cite{BoRo10} for further developments. 
Our framework also covers the case of using repeated measurements (panel data) with a symmetric error distribution similarly to that of \cite{HoMa96}.
A recent work by \cite{DeHa16} relaxes the requirement of repeated measurements under the assumption of a symmetric error distribution.
Despite the richness of this literature, however, uniform confidence bands for $f_X$, which we develop in this paper, have not been developed in any of these preceding papers allowing for unknown $f_\varepsilon$.

The deconvolution problem is a statistical ill-posed inverse problem \citep[see, e.g.][]{Ho09}, and developing formal theories for inference in ill-posed inverse problems tends to be challenging.\footnote{See  \citet{HoLe12,ChCh15,Ba16} for uniform confidence bands in the context of nonparametric instrumental variables (NPIV) models, one of the popular classes of econometric models with ill-posedness.}
Existing studies on uniform confidence bands in deconvolution focus on the case where $f_\varepsilon$ is known. 
To the best of our knowledge, \cite{BiDuHoMu07} is the first paper that formally studies uniform confidence bands in deconvolution.    
They assume that $f_\varepsilon$ is known and ordinary smooth, and prove a Smirnov-Bickel-Rosenblatt type limit theorem \citep[cf.][]{Sm50,BiRo73} for the deconvolution kernel density estimator under a number of technical conditions based on the Koml\'{o}s-Major-Tusn\'{a}dy (KMT) strong approximation \citep{KoMaTu75} and  extreme value theory \citep[cf.][]{LeLiRo83}. 
They prove that the supremum deviation of the deconvolution kernel density estimator, suitably normalized, converges in distribution to a Gumbel distribution. 
They also prove consistency of the nonparametric bootstrap. 
See also \cite{BiHo08}.
For super-smooth error densities, \cite{EsGu08} show that the limit distribution of the supremum deviation of the deconvolution kernel density estimator in general differs from Gumbel distributions. 
We also refer to \cite{LoNi11,ScMuDu13,DeHaJa15}.  
Importantly, none of these papers formally studies the case where the error density $f_{\varepsilon}$ is unknown.\footnote{In developing uniform confidence bands for the cumulative distribution function, as opposed to the density function, \cite{AdOtWh16} consider the case of unknown $f_X$. \cite{AdOtWh16} appeared after the present paper was uploaded on arXiv.}
Indeed \citet[][Section 4.2]{DeHaJa15} discuss how to possibly accommodate the case of unknown error density, but a theory to support this argument is not provided.\footnote{The focus in \cite{DeHaJa15} is on pointwise confidence intervals for nonparametric regression functions, and differs from our objective to conduct uniform inference on nonparametric density functions. 
}
While the effect of pre-estimating the unknown error characteristic function for the purpose of estimating $f_X$ is modest, its effect on the validity of inference on $f_X$ is not ignorable.
We contribute to this literature by formally establishing a method to construct uniform confidence bands for $f_X$ where the error density $f_{\varepsilon}$ is unknown. 


From a technical point of view, the present paper builds upon non-trivial applications of the ``intermediate'' Gaussian and multiplier bootstrap approximation theorems developed in \cite{ChChKa14a,ChChKa14b,ChChKa16}. 
These approximation theorems are applicable to the general empirical process under weaker regularity conditions than those for the KMT and Gumbel approximations. 
However, we stress that those theorems are not directly applicable to our problems and substantial work is need to derive our results. 
This is because: 1) the ``deconvolution'' kernel $K_{n}$ (see Section \ref{sec: methodology} ahead) is implicitly defined via the Fourier inversion, and verifying conditions in those approximation theorems with the deconvolution kernel $K_{n}$ is involved; 2) the error density $f_{\varepsilon}$ is unknown and we have to work with the estimated deconvolution  kernel $\hat{K}_{n}$, and so the estimation error has to be taken into account, which requires delicate cares.

\section{Methodology}
\label{sec: methodology}

In this section, we informally present our methodology to construct confidence bands for $f_{X}$.
Formal analysis of our methodology will be carried out in the following sections. 

We first fix basic notations.
For $a,b \in \R$, let $a \wedge b = \min \{ a,b \}$. For $a \in \R$ and $b > 0$, we use the notation $[a \pm b] = [a-b,a+b]$. 
For a non-empty set $T$ and a (complex-valued) function $f$ on $T$, we use the notation $\| f \|_{T} = \sup_{t \in T} |f(t)|$. Let $\ell^{\infty}(T)$ denote the Banach space of all bounded real-valued functions on $T$ with norm $\| \cdot \|_{T}$. 
The Fourier transform of an integrable  function $f$ on $\R$ is defined by $\varphi_{f}(t) = \int_{\R} e^{itx} f(x) dx$ for $t \in \R$,
where $i=\sqrt{-1}$ denotes the imaginary unit throughout the paper. 
We refer to \cite{Fo99} as a basic reference on the Fourier analysis. 

Let $\varphi_{Y}, \varphi_{X}$, and $\varphi_{\varepsilon}$ denote the Fourier transforms (the characteristic functions) of $f_{Y},f_{X}$, and $f_{\varepsilon}$, respectively. The model (\ref{eq: model}) implies that these characteristic functions satisfy the relation
\begin{equation}
\varphi_{Y} (t) = \varphi_{X}(t) \varphi_{\varepsilon}(t) \quad \text{for all} \ t \in \R.
\label{eq: subindependence}
\end{equation}
If $\varphi_{\varepsilon}$ does not vanish on $\R$ and $| \varphi_{X} |$ is integrable on $\R$, then the Fourier inversion formula yields
\begin{equation}
f_{X} (x) = \frac{1}{2\pi} \int_{\R} e^{-itx} \frac{\varphi_{Y}(t)}{\varphi_{\varepsilon}(t)} dt. \label{eq: deconvolution}
\end{equation}
Suppose that independent copies $Y_{1},\dots,Y_{n}$ of $Y$ are observed. 
For convenience of presentation, assume just for the moment that the distribution of $\varepsilon$ were known. 
The standard deconvolution kernel density estimator of $f_X$ under this tentative assumption is given by 
\[
\hat{f}_{X}^{*} (x) =\frac{1}{2\pi} \int_{\R} e^{-itx} \hat{\varphi}_{Y}(t) \frac{\varphi_{K}(th_{n})}{\varphi_{\varepsilon}(t)} dt = \frac{1}{nh_{n}} \sum_{j=1}^{n} K_{n}  ((x-Y_{j})/h_{n}), 
\]
where 
\[
\hat{\varphi}_{Y}(t) = \frac{1}{n} \sum_{j=1}^{n} e^{itY_{j}}  \quad \text{and} \quad K_{n}(x) = \frac{1}{2\pi}  \int_{\R} e^{-itx} \frac{\varphi_{K}(t)}{\varphi_{\varepsilon}(t/h_{n})} dt.
\]
Here, the kernel function $K: \R \to \R$ is real-valued, is integrable, and integrates to one, such that its Fourier transform $\varphi_{K}$ is supported in $[-1,1]$ (i.e., $\varphi_{K}(t)=0$ for all $|t| > 1$), and the sequence of positive numbers (bandwidth) $h_{n}$ satisfies $h_{n} \to 0$ \citep[cf.][]{CaHa88, StCa90}. 
Note that the function $K_{n}$ is real-valued, and bounded due to the compactness of the support of $\varphi_{K}$. 
The function $K_{n}$ is called a \textit{deconvolution kernel}. 

We are now interested in constructing a confidence band for $f_{X}$ on a compact interval $I \subset \R$. 
A confidence band $\mathcal{C}_{n}$ at level $(1-\tau)$ for a given $\tau \in (0,1)$ is a family of random intervals $\mathcal{C}_{n} = \{ \mathcal{C}_{n}(x) = [c_{L}(x),c_{U}(x)]: x \in I \}$ such that 
\[
\Pr \{ f_{X}(x) \in [c_{L}(x),c_{U}(x)] \ \forall x \in I \} \geq 1-\tau.
\]
Such a confidence band can be constructed by approximating the distribution of the supremum in absolute value of the following stochastic process:
\[
Z_{n}^{*}(x) = \frac{\hat{f}_{X}^{*} (x)- \Ep[\hat{f}_{X}^{*}(x)]}{\sqrt{\Var (\hat{f}_{X}^{*}(x))}} = \frac{\sqrt{n}h_{n}\{ \hat{f}_{X}^{*} (x)- \Ep[\hat{f}_{X}^{*}(x)] \}}{\sigma_{n}(x)}, \ x \in I,
\]
where we assume that $\sigma_{n}^{2}(x) = \Var ( K_{n}((x-Y)/h_{n})) > 0$ for all $x \in I$, and $\sigma_{n}(x)$ is defined by $\sigma_{n}(x) = \sqrt{\sigma_{n}^{2}(x)}$. 
Recall the notation $\| Z_{n}^{*} \|_{I} = \sup_{x \in I} |Z^{*}_{n}(x)|$. 
Set
\[
c_{n}^{*} (1-\tau) = \text{$(1-\tau)$-quantile of $\| Z_{n}^{*} \|_{I}$}.
\]
Then, the band of the form 
\begin{equation}
\mathcal{C}_{n}^{*}(x) = \left [ \hat{f}_{X}^{*}(x) \pm  \frac{\sigma_{n}(x)}{\sqrt{n}h_{n}} c_{n}^{*}(1-\tau) \right ], \ x \in I \label{eq: ideal band}
\end{equation}
contains $\{ \Ep[ \hat{f}_{X}^{*}(x)]  : x \in I \}$ with probability at least $(1-\tau)$, as 
\[
\Pr \{ \Ep[ \hat{f}_{X}^{*}(x)] \in \mathcal{C}_{n}^{*}(x) \ \forall x \in I \} = \Pr \{ \| Z_{n}^{*} \|_{I} \leq c_{n}^{*}(1-\tau) \} \geq 1-\tau.
\]
If the bias $\| f_{X} (\cdot)- \Ep[ \hat{f}_{X}^{*}(\cdot) ]\|_{I}$ is made sufficiently small (e.g., by choosing  undersmoothing bandwidths), then the band of the form (\ref{eq: ideal band}) serves as a valid confidence band for $f_{X}$ on $I$ at level approximately $(1-\tau)$. 

Constructing a band of the form (\ref{eq: ideal band}) is, however, infeasible because both the distribution of $\| Z_{n}^{*} \|_{I}$ and the variance function $\sigma^{2}_{n}(\cdot)$ are unknown. 
More fundamentally, in most (if not all) economic applications, the error distribution is unknown, and so the deconvolution kernel estimator $\hat{f}_{X}^{*}$ is infeasible. 
In this paper, we allow $f_{\varepsilon}$ to be unknown, but assume the availability of an independent sample $\eta_{1},\dots,\eta_{m}$ from $f_{\varepsilon}$ where $m=m_{n} \to \infty$ as $n \to \infty$.  
One such case is where a validation data set provides $\eta_{1},\dots,\eta_{m}$. 
A more realistic example where such observations are available is the case where we observe repeated measurements on $X$ with errors such that the conditional distribution of one given the other is symmetric.
The following example illustrates the case in point.

\begin{example}[\cite{CaRuStCr06}, p.298]
\label{ex: canonical}
Suppose that we observe repeated measurements on $X$ with errors:
\[
\begin{cases}
Y^{(1)} = X + \varepsilon^{(1)}, \\
Y^{(2)} = X + \varepsilon^{(2)},
\end{cases}
\]
where $X$ and $(\varepsilon^{(1)},\varepsilon^{(2)})$ are independent.
$\varepsilon^{(1)}$ and $\varepsilon^{(2)}$ need not to be independent, nor do they have common distribution. 
Then, we have 
\[
\underbrace{(Y^{(1)}+Y^{(2)})/2}_{=:Y} = X + \underbrace{(\varepsilon^{(1)}+\varepsilon^{(2)})/2}_{=:\varepsilon}
\]
and $(\varepsilon^{(1)} - \varepsilon^{(2)})/2 =  (Y^{(1)}-Y^{(2)})/2$.
Hence, if $\varepsilon^{(1)}+\varepsilon^{(2)}$ has the same distribution as $\varepsilon^{(1)} - \varepsilon^{(2)}$, which is true if the conditional distribution of $\varepsilon^{(2)}$ given $\varepsilon^{(1)}$ is symmetric, then $\eta:=(Y^{(1)}-Y^{(2)})/2$ has the same distribution as $\varepsilon = (\varepsilon^{(1)}+\varepsilon^{(2)})/2$.
In this example, $m=n$.
$\triangle$ 
\end{example}

In any of these cases, a natural estimator of $\varphi_{\varepsilon}$ is the empirical characteristic function based on $\eta_{1}.\dots,\eta_{m}$:
\[
\hat{\varphi}_{\varepsilon}(t)  = \frac{1}{m} \sum_{j=1}^{m} e^{i t \eta_{j}}.
\]
Suppose that $\inf_{|t| \leq h_{n}^{-1}}| \hat{\varphi}_{\varepsilon}(t) | > 0$ with probability approaching one, which is indeed guaranteed under the assumptions to be formally stated below.
Then, we can estimate the deconvolution kernel $K_{n}$ by
\begin{equation}
\hat{K}_{n}(x) =\frac{1}{2\pi}  \int_{\R} e^{-itx} \frac{\varphi_{K}(t)}{\hat{\varphi}_{\varepsilon}(t/h_{n})} dt.
\label{eq: estimated kernel}
\end{equation}
Now, define the feasible version of $\hat{f}_{X}^{*}$ as 
\[
\hat{f}_{X} (x)= \frac{1}{nh_{n}} \sum_{j=1}^{n} \hat{K}_{n}((x-Y_{j})/h_{n}).
\]
This estimator was first considered by \cite{DiHa93}. 
In addition, we may estimate the variance function $\sigma_{n}^{2}(x)$ by 
\[
\hat{\sigma}_{n}^{2}(x) = \frac{1}{n} \sum_{j=1}^{n} \hat{K}_{n}^{2}((x-Y_{j})/h_{n}) - \left ( \frac{1}{n} \sum_{j=1}^{n} \hat{K}_{n}((x-Y_{j})/h_{n}) \right )^{2}. 
\]
Consider the stochastic process
\[
\hat{Z}_{n}(x) = \frac{\sqrt{n}h_{n} \{ \hat{f}_{X}(x) - \Ep[ \hat{f}_{X}^{*}(x) ] \}}{\hat{\sigma}_{n}(x)}, \ x \in I,
\]
where $\hat{\sigma}_{n}(x) = \sqrt{\hat{\sigma}_{n}^{2}(x)}$.
Let 
\[
c_{n}(1-\tau) = \text{$(1-\tau)$-quantile of $\| \hat{Z}_{n} \|_{I}$}.
\]
Then the band of the form 
\[
\left [ \hat{f}_{X}(x) \pm \frac{\hat{\sigma}_{n}(x)}{\sqrt{n}h_{n}}c_{n}(1-\tau) \right ], \ x \in I
\]
will be a valid confidence band for $f_{X}$ on $I$ at level approximately $(1-\tau)$, provided that the bias $\| f_{X} (\cdot)- \Ep[ \hat{f}_{X}^{*}(\cdot) ]\|_{I}$ is made sufficiently small. 

The quantiles of $\| \hat{Z}_{n} \|_{I}$ are still unknown, but it will be shown below that, under suitable regularity conditions, the distribution of $\| \hat{Z}_{n} \|_{I}$ can be approximated by that of the supremum in absolute value of a tight Gaussian random variable $Z_{n}^{G}$ in $\ell^{\infty}(I)$ with mean zero and the same covariance function as $Z_{n}^{*}$. 
As such, we propose to estimate the quantiles of $\| \hat{Z}_{n} \|_{I}$ via the \textit{Gaussian multiplier bootstrap} as in \cite{ChChKa14b} in the following manner. 

Generate independent standard normal random variables $\xi_{1},\dots,\xi_{n} \sim N(0,1)$, independently of the data $\mathcal{D}_{n} = \{ Y_{1},\dots,Y_{n},\eta_{1},\dots,\eta_{m} \}$, and consider the multiplier process
\[
\hat{Z}_{n}^{\xi} (x) = \frac{1}{\hat{\sigma}_{n}(x)\sqrt{n}} \sum_{j=1}^{n} \xi_{j} \left \{  \hat{K}_{n} ((x-Y_{j})/h_{n}) - n^{-1} {\textstyle \sum}_{j'=1}^{n} \hat{K}_{n}((x-Y_{j'})/h_{n}) \right \}
\]
for $x \in I$. 
Conditionally on the data $\mathcal{D}_{n}$, $\hat{Z}_{n}^{\xi}(x), x \in I$ is a Gaussian process with mean zero and the covariance function ``close'' to that of $Z_{n}^{G}$. 
Hence we propose to estimate the quantile $c_{n}(1-\tau)$ by 
\[
\hat{c}_{n}(1-\tau) = \text{conditional $(1-\tau)$-quantile of $\| \hat{Z}_{n}^{\xi} \|_{I}$ given $\mathcal{D}_{n}$},
\]
and the resulting confidence band is 
\begin{equation}
\hat{\mathcal{C}}_{n}(x) = \left [ \hat{f}_{X}(x) \pm \frac{\hat{\sigma}_{n}(x)}{\sqrt{n}h_{n}} \hat{c}_{n}(1-\tau) \right ], \ x \in I. \label{eq: multiplier bootstrap confidence band}
\end{equation}

A few remarks are in order.
\begin{remark}
 How to choose the bandwidth in practice is an important yet difficult problem in any nonparamtric inference. 
Practical choice of the bandwidth will be discussed in Section \ref{subsec: choice of the bandwidth}. 
\end{remark}
\begin{remark}
Our construction (and the formal analysis below) covers the case where $I$ is singleton, i.e., $I=\{ x_{0} \}$.
In this case, the above confidence band gives a confidence interval for $f_{X}(x_{0})$. 
\end{remark}
\begin{remark}
The presence of $\hat{\varphi}_{\varepsilon}$ in the denominator in the integrand in (\ref{eq: estimated kernel}) could make the estimate $\hat{f}_{X}$ numerically unstable in practice.
A solution to this problem is to restrict the integral in (\ref{eq: estimated kernel}) to the set $\{ t \in \R: | \hat{\varphi}_{\varepsilon}(t)| \geq m^{-1/2} \}$ \citep[cf.][]{Ne97}. 
Likewise, replacing $\hat{\sigma}_{n}(x)$ by $\max \{ \hat{\sigma}_{n}(x), \sqrt{h}_{n} \}$ in the definition of $\hat{Z}_{n}$ and $\hat{Z}_{n}^{\xi}$ would make resulting confidence bands numerically more stable in practice. 
These modifications do not alter the asymptotic results presented below, and we will work with the original definitions of $\hat{K}_{n}$ and $\hat{\sigma}_{n}$. 
\end{remark}
\begin{remark}
In the present paper, we work with the classical measurement error setting, namely, we assume that $X$ and $\varepsilon$ are independent. 
However, for our theoretical results to hold, the full independence between $X$ and $\varepsilon$ is not necessary. 
Instead, we only require condition (\ref{eq: subindependence}), which may hold even when $X$ and $\varepsilon$ are not independent. \cite{Sc13} argues that condition (\ref{eq: subindependence}) is ``as weak as a conditional mean assumption.''
\end{remark}


\section{Main results}
\label{sec: main}

In this section, we present theorems that provide the asymptotic validity of the proposed confidence bands.
We first consider the case where the error density $f_{\varepsilon}$ is ordinary smooth.
%
We begin with stating and discussing the assumptions.

\begin{assumption}
\label{as: standard}
The function $| \varphi_{X} |$ is integrable on $\R$.
 \end{assumption}

\begin{assumption}
\label{as: kernel}
Let $K: \R \to \R$ be an integrable function (kernel) such that  $\int_{\R} K(x) dx = 1$, and its Fourier transform $\varphi_{K}$ is continuously differentiable and supported in $[-1,1]$. 
\end{assumption}

Both of these assumptions are standard in the literature on deconvolution. Note that Assumption \ref{as: standard} implies that $f_{X}$ is bounded and continuous, which in turn implies that $f_{Y} = f_{X} * f_{\varepsilon}$ is bounded and continuous. Recall that if $f: \R \to \R$ is integrable and $g: \R \to \R$ is bounded, then their convolution $f*g$ is bounded and continuous \citep[cf.][Proposition 8.8]{Fo99}. 
The kernel function $K$ does not necessarily have to be non-negative under Assumption \ref{as: kernel}.

The next assumption is concerned with the tail behavior of the error characteristic function $\varphi_{\varepsilon}$, which is a source of ``ill-posedness'' of the deconvolution problem and an important factor that determines the difficulty of estimating $f_{X}$; see, e.g., \cite{Ho09}. 
(Another factor is the smoothness of $f_{X}$.) 
We assume here that the error density $f_{\varepsilon}$ is ordinary smooth, i.e., $| \varphi_{\varepsilon} (t) |$ decays at most polynomially fast as $|t| \to \infty$, as formally stated below. 

\begin{assumption}
\label{as: ill-posedness}
The error characteristic function $\varphi_{\varepsilon}$ is continuously differentiable and does not vanish on $\R$, and 
there exist constants  $\alpha > 0$ and $C_{1} > 1$ such that $
C_{1}^{-1} |t|^{-\alpha} \leq | \varphi_{\varepsilon} (t) | \leq C_{1} |t|^{-\alpha}$ and $| \varphi_{\varepsilon}'(t) | \leq C_{1}|t|^{-\alpha-1}$ for all $|t| \geq 1$
\end{assumption}

Concrete examples of distributions that satisfy Assumption \ref{as: ill-posedness} are Laplace and Gamma distributions together with their convolutions, but apparently many other distributions satisfy Assumption \ref{as: ill-posedness}. It is not difficult to see that Assumption \ref{as: ill-posedness} implies that 
\[
\inf_{|t| \leq h_{n}^{-1}} |\varphi_{\varepsilon} (t)| \geq C_{1}^{-1} (1 - o(1)) h_{n}^{\alpha}
\]
as $n \to \infty$. The value of $\alpha$ quantifies the degrees of ``ill-posedness'' of the deconvolution problem, and the larger the value of $\alpha$ is, the more difficult the estimation of $f_{X}$ will be. 

We draw confidence bands for $f_X$ on a set $I$ given in the following assumption.
\begin{assumption}
\label{as: variance}
Let $I \subset \R$ be a compact interval such that $f_{Y}(y) > 0$ for all $y \in I$. 
\end{assumption}
Now, recall that $\sigma_{n}^{2}(x) = \Var (K_{n}  ((x-Y)/h_{n}))$.
In developing our theory, we will need $\sigma_{n}^{2}(x)/h_{n}^{-2\alpha+1}$ to be bounded away from zero on $I$. 
It will be shown in Lemma \ref{lem: variance} that Assumptions \ref{as: standard}--\ref{as: variance} guarantee that $\sigma_{n}^{2}(x)/h_{n}^{-2\alpha+1}$ is bounded away from zero on $I$ for sufficiently large $n$.

The next assumption is a mild moment condition on the error distribution, which is used in establishing  uniform convergence rates of the empirical characteristic function (see Lemma \ref{lem: rate ecf}). 

\begin{assumption}
\label{as: moment condition}
$\Ep[| \varepsilon |^{p}] < \infty$ for some $p > 0$.
\end{assumption}

The next assumption mildly restricts the bandwidth $h_{n}$ and the sample size $m = m_{n}$ for $f_{\varepsilon}$. 
\begin{assumption}
\label{as: rate}
(a) $\frac{(\log h_{n}^{-1})^{2}}{nh_{n}^{2}} \to 0$. 
(b) $\frac{nh_{n} \log h_{n}^{-1}}{m} \bigvee \frac{(\log h_{n}^{-1})^{2}}{mh_{n}^{2\alpha+2}} \to 0$. 
\end{assumption}

\begin{remark}
For an illustrative purpose, consider the canonical case where $m=n$. 
Then Assumption \ref{as: rate} reduces to the following simple condition: 
\begin{equation}
\frac{(\log h_{n}^{-1})^{2}}{nh_{n}^{2\alpha+2}} \to 0. \label{eq: simplified condition}
\end{equation}
The conventional ``optimal'' bandwidth that minimizes the MISE of the kernel estimator (when $f_{\varepsilon}$ is known) is proportional to $n^{-1/(2\alpha+2\beta+1)}$ where $\beta$ is the ``smoothness'' of $f_{X}$ \citep[cf.][]{Fa91a}, and so condition (\ref{eq: simplified condition}) is satisfied with this bandwidth if $\beta >1/2$. 
See also Corollary \ref{cor: uniform convergence rate} below.
\end{remark}

The following theorem establishes that the distribution of the supremum in absolute value of the stochastic process $\hat{Z}_{n}(x), x \in I$ can be approximated by that of a tight Gaussian random variable $Z_{n}^{G}$ in $\ell^{\infty}(I)$ with mean zero and the same covariance function as $Z_{n}^{*}$. 
This theorem is a building block for establishing the validity of the Gaussian multiplier bootstrap described in the previous section. 
Recall that a Gaussian process $Z = \{ Z(x) :  x \in I \}$ indexed by $I$ is a tight random variable  in $\ell^{\infty}(I)$ if and only if $I$ is totally bounded for the intrinsic pseudo-metric $\rho_{2}(x,y) = \sqrt{\Ep[ \{ Z(x) - Z(y) \}^{2}]}$ for $x,y \in I$, and $Z$ has sample paths almost surely uniformly $\rho_{2}$-continuous; see \citet[][p.41]{vaWe96}.  
In that case, we say that $Z$ is a tight Gaussian random variable in $\ell^{\infty}(I)$ \citep[cf.][Lemma 3.9.8]{vaWe96}. 

\begin{theorem}
\label{thm: Gaussian approximation}
Suppose htat Assumptions \ref{as: standard}--\ref{as: rate} are satisfied. For each sufficiently large $n$, there exists a tight Gaussian random variable $Z_{n}^{G}$ in $\ell^{\infty}(I)$ with mean zero and the same covariance function as $Z_{n}^{*}$, such that 
as $n \to \infty$, 
\[
\sup_{z \in \R} | \Pr \{ \| \hat{Z}_{n} \|_{I} \leq z \} - \Pr \{ \| Z_{n}^{G} \|_{I} \leq z \} | \to 0.
\]
\end{theorem}

In the case where $I$ is not a singleton, it is possible to further show that $\| Z_{n}^{G} \|_{I}$ (and hence $\| \hat{Z}_{n} \|_{I}$) properly normalized converges in distribution to a Gumbel distribution (i.e., a Smirnov-Bickel-Rosenblatt type limit theorem) under additional substantial conditions, as in \cite{BiDuHoMu07}.
However, we intentionally stop at the ``intermediate'' Gaussian approximation instead of deriving the Gumbel approximation, because of the following two reasons. 1) The Gumbel approximation is poor, and the coverage error of the resulting confidence band is of order $1/\log n$ \citep{Ha91}. 2) Deriving the Gumbel approximation requires additional substantial conditions.
Because of the slow rate of the Gumbel approximation, it is often preferred to use versions of bootstraps to construct confidence bands for nonparametric density and regression functions \citep[see, e.g.,][]{Clva03, BiDuHoMu07}, but the Gumbel approximation was used as  a building block for showing validity of the bootstraps. 
It was, however, pointed out in \cite{ChChKa14b} that the intermediate Gaussian approximation (such as that in Theorem \ref{thm: Gaussian approximation}) is in fact sufficient for showing the validity of bootstraps. 
We defer the discussion on the regularity conditions to the end of this section. 

Another technicality in the proof of Theorem \ref{thm: Gaussian approximation} concerns about bounding the effect of the estimation error of $\hat{\varphi}_{\varepsilon}$. 
\citet[][p.172]{DaReTr16} derive a bound on $\| \hat{f}_{X} - \hat{f}_{X}^{*} \|_{\R}$ that is of order $O_{\Pr} \{ h_{n}^{-\alpha} (mh_{n})^{-1/2} \}$, but this rate is not sufficient for our purpose and in particular excludes the case with $m=n$ in Theorem \ref{thm: Gaussian approximation}; see Step 2 in the proof of Theorem \ref{thm: Gaussian approximation}. Hence, to bound the effect of the estimation error of $\hat{\varphi}_{\varepsilon}$, we require a novel idea beyond \cite{DaReTr16}; see Step 2 in the proof of Theorem \ref{thm: Gaussian approximation}. 

As a byproduct of the techniques used to prove Theorem \ref{thm: Gaussian approximation}, we can derive uniform convergence rates of $\hat{f}_{X}$ on $\R$. 
In the next corollary, Assumption \ref{as: variance} is not needed.

\begin{corollary}
\label{thm: uniform convergence rate}
Suppose that Assumptions \ref{as: standard}--\ref{as: ill-posedness}, \ref{as: moment condition}, and \ref{as: rate} are satisfied. Then, $\| \hat{f}_{X} (\cdot) - \Ep [ \hat{f}_{X}^{*}(\cdot) ] \|_{\R} = O_{\Pr} \{ h_{n}^{-\alpha} (nh_{n})^{-1/2}  \sqrt{\log h_{n}^{-1}} \}$ as $n \to \infty$.
\end{corollary}

Corollary \ref{thm: uniform convergence rate} does not take into account the bias $\| \Ep [\hat{f}^{*}_{X}(\cdot)] - f_{X}(\cdot) \|_{\R}$, but the above rate is the correct one for the ``variance part'' (or the ``stochastic part'') when $f_{\varepsilon}$ is known. 
To decide uniform convergence  rates for $\hat{f}_{X}$, we have to make an assumption on the smoothness of $f_{X}$. 
In the following, for $\beta > 0$ and $B > 0$, let $\Sigma (\beta,B)$ denote a H\"{o}lder ball of functions on $\R$ with smoothness $\beta$ and radius $B$, namely,
\begin{align*}
\Sigma (\beta, B) = \{ f: \R \to \R &: \text{$f$ is $k$-times differentiable and} \\
&\quad | f^{(k)} (x) - f^{(k)} (y) | \leq B|x-y|^{\beta-k} \ \forall x,y \in \R \},
\end{align*}
where $k$ is the integer such that $k < \beta \leq k+1$ ($k=0$ if $\beta \in (0,1]$). Further, we will assume that the kernel function $K$ is such that 
\begin{equation}
\label{eq: higher order kernel}
\int_{\R} |x|^{k+1} |K(x)| dx < \infty, \ \text{and} \ \int_{\R} x^{\ell} K(x) dx, \ \ell=1,\dots,k,
\end{equation}
i.e., $K$ is a $(k+1)$-th order kernel. 
For any positive sequences $a_{n}, b_{n}$, we write $a_{n} \ll b_{n}$ if $a_{n}/b_{n} \to 0$ as $n \to \infty$. 

\begin{corollary}
\label{cor: uniform convergence rate}
Suppose that Assumptions \ref{as: standard}--\ref{as: ill-posedness} and \ref{as: moment condition} are satisfied. Further, suppose that $f_{X} \in \Sigma (\beta,B)$ for some $\beta > 1/2, B > 0$, and that Condition (\ref{eq: higher order kernel}) is satisfied for the kernel function $K$,  where $k$ is the integer such that $k < \beta \leq k+1$. Take $h_{n} = C (n/\log n)^{-1/(2\alpha + 2\beta+1)}$ for any constant $C>0$; then 
\[
\| \hat{f}_{X} - f_{X} \|_{\R} = O_{\Pr}  \{ (n/\log n)^{-\beta/(2\alpha + 2\beta+1)}  \}
\]
 provided that $m \gg n^{\frac{2\alpha+2\beta}{2\alpha+2\beta+1}} (\log n)^{1+\frac{1}{2\alpha+2\beta+1}}$. 
\end{corollary}

\begin{remark}[On Condition (\ref{eq: higher order kernel})]
Condition (\ref{eq: higher order kernel}) on the kernel function $K$ can be verified through its Fourier transform $\varphi_{K}$. In fact, it is not difficult to see that, if $\varphi_{K}$ is $(k+3)$-times continuously differentiable and $\varphi_{K}^{(\ell)}(0) = 0$ for $\ell=1,\dots,k$, then Condition (\ref{eq: higher order kernel}) is satisfied. 
\end{remark}

\begin{remark}
Informally, for a given error density $f_{\varepsilon}$ such that $| \varphi_{\varepsilon}(t) |$ decays like $|t|^{-\alpha}$ as $|t| \to \infty$, $(n/\log n)^{-\beta/(2\alpha + 2\beta+1)}$ is the minimax rate of convergence for estimating $f_{X}$ under the sup-norm loss when $f_{X} \in \Sigma(\beta,B)$ and there is no additional sample from the error distribution. 
See Theorem 1 in \cite{LoNi11} for the precise formulation. In fact, the proof of Theorem 1 in \cite{LoNi11} continues to hold even when there is a sample $(\eta_{1},\dots,\eta_{m})$ from the error distribution that is independent from $(Y_{1},\dots,Y_{n})$ -- in their proof, modify $P_{k}^{n}$ to be the distribution admitting the joint density $\prod_{j=1}^{n} (f_{k}*f_{\varepsilon})(y_{j}) \prod_{k=1}^{m} f_{\varepsilon}(\eta_{k})$. 
Hence Corollary \ref{cor: uniform convergence rate} shows that $\hat{f}_{X}$ attains the minimax rate under the sup-norm loss for $\beta > 1/2$, provided that other technical conditions are satisfied.
\end{remark}

\begin{remark}
The literature on uniform convergence rates in deconvolution is limited. 
\cite{LoNi11} and \citet[][Section 5.3.2]{GiNi16} derive uniform convergence rates for deconvolution wavelet and kernel density estimators on the entire real line assuming that the error density is known; \citet[][Proposition 2.6]{DaReTr16} derive uniform convergence rates for the deconvolution kernel density estimator with the estimated error characteristic function, but on a bounded interval. So their results do not cover the above corollaries.
\end{remark}

Now, we present the validity of the proposed multiplier bootstrap confidence bands.

\begin{theorem}
\label{thm: multiplier bootstrap}
Suppose that Assumptions \ref{as: standard}--\ref{as: rate} are satisfied.
As $n \to \infty$, 
\begin{equation}
\sup_{z \in \R}  | \Pr \{ \| \hat{Z}_{n}^{\xi} \|_{I} \leq z \mid \mathcal{D}_{n} \} - \Pr \{ \| Z_{n}^{G} \|_{I} \leq z \} | \stackrel{\Pr}{\to} 0, \label{eq: bootstrap}
\end{equation}
where $\mathcal{D}_{n}= \{ Y_{1},\dots,Y_{n},\eta_{1},\dots,\eta_{m} \}$, and $Z_{n}^{G}$ is the Gaussian random variable in $\ell^{\infty}(I)$ given in Theorem \ref{thm: Gaussian approximation}. 
Therefore, letting $\hat{c}_{n}(1-\tau)$ denote the $(1-\tau)$-quantile of the conditional distribution of $\| \hat{Z}_{n}^{\xi} \|_{I}$ given $\mathcal{D}_{n}$, we have that
\begin{equation}
\Pr \{ \| \hat{Z}_{n} \|_{I} \leq \hat{c}_{n}(1-\tau) \} \to 1-\tau \label{eq: validity of MB}
\end{equation}
as $n \to \infty$. Finally, the supremum width of the band $\hat{\mathcal{C}}_{n}$ is $O_{\Pr} \{ h_{n}^{-\alpha}(nh_{n})^{-1/2}\sqrt{\log h_{n}^{-1}} \}$. 
\end{theorem}

Theorem \ref{thm: multiplier bootstrap} shows that the multiplier bootstrap confidence band $\hat{\mathcal{C}}_{n}$ defined in (\ref{eq: multiplier bootstrap confidence band}) contains the surrogate function $\Ep[\hat{f}_{X}^{*}(\cdot)]$ on $I$ with probability  $1-\tau+o(1)$ as $n \to \infty$. 
If $f_{X}$ belongs to a H\"{o}lder ball $\Sigma (\beta,B)$, then $\hat{\mathcal{C}}_{n}$ will be a valid confidence band for $f_{X}$ provided that $h_{n}$ is chosen in such a way that $h_{n}^{\alpha+\beta}\sqrt{n h_{n} \log h_{n}^{-1}} \to 0$, which corresponds to choosing undersmoothing bandwidths. 

\begin{corollary}
\label{cor: multiplier bootstrap}
Suppose that Assumptions \ref{as: standard}--\ref{as: rate} are satisfied. 
Furthermore, suppose that $f_{X} \in \Sigma (\beta,B)$ for some $\beta > 0, B > 0$, and that Condition (\ref{eq: higher order kernel}) is satisfied for the kernel function $K$ where $k$ is the integer such that $k < \beta \leq k+1$. Consider the multiplier bootstrap confidence band $\hat{\mathcal{C}}_{n}$ defined in (\ref{eq: multiplier bootstrap confidence band}). 
Then, as $n \to \infty$, $\Pr \{ f_{X}(x) \in \hat{\mathcal{C}}_{n}(x) \ \forall x \in I \} \to 1-\tau$,
provided that
\begin{equation}
h_{n}^{\alpha+\beta}  \sqrt{n h_{n} \log h_{n}^{-1}} \to 0. \label{eq: undersmoothing}
\end{equation}
\end{corollary}

Consider the canonical case where $m=n$. Then the conditions on the bandwidth $h_{n}$ in Corollary \ref{cor: multiplier bootstrap} reduce to 
\[
\frac{(\log h_{n}^{-1})^{2}}{nh_{n}^{2\alpha+2}}  \bigvee h_{n}^{\alpha+\beta}  \sqrt{n h_{n} \log h_{n}^{-1}} \to 0,
\]
and so we need $\beta > 1/2$ in order to ensure the existence of bandwidths satisfying these conditions. For example, if $\beta > 1/2$, choosing $h_{n} = v_{n} (n/\log n)^{-1/(2\alpha+2\beta+1)}$ for $v_{n} \sim (\log n)^{-1}$ satisfies the above restriction and yields that the supremum width of the band $\hat{\mathcal{C}}_{n}$ is  
\[
O_{\Pr} \{ (n/\log n)^{-\beta/(2\alpha+2\beta+1)}(\log n)^{\alpha+1/2} \},
\]
which is close to the optimal rate up to $\log n$ factors. 

\begin{remark}[On undersmoothing]
In the present paper, we assume undersmoothing bandwidths so that the deterministic  bias is asymptotically negligible relative to the ``variance'' or ``stochastic'' term. An alternative approach is to estimate the bias at each point, and construct a bias correct confidence band; see \cite{EuSp93} and \cite{Xi98} for bias corrected confidence bands in a regression context. See also \cite{HaHo13}, \cite{ChChKa14b}, \cite{Sc15}, \cite{ArKo14}, and \cite{CaCaFa15} for recent discussions including yet alternative approaches on selection of bandwidths for confidence intervals or bands. 
These papers do not formally cover the case of deconvolution, and formally adapting such approaches to deconvolution is beyond the scope of this paper. For practical choice of the bandwidth, see Section \ref{subsec: choice of the bandwidth}.
\end{remark}

\begin{remark}[Comparisons with \cite{BiDuHoMu07} and \cite{ScMuDu13}]
\label{rem: comparison}
\cite{BiDuHoMu07} is an important pioneering work on confidence bands in deconvolution. They assume that the error density is known and ordinary smooth, and show that 
\[
\sqrt{2\log h_{n}^{-1}} (\| \sqrt{n}h_{n}^{\alpha+1/2} (\hat{f}_{X}^{*}(\cdot) - \Ep[\hat{f}_{X}^{*}(\cdot)])/\sqrt{f_{Y}(\cdot)} \|_{[0,1]}/C_{K,1}^{1/2} - d_{n} )
\]
converges in distribution to a Gumbel distribution, where $d_{n}=\sqrt{2\log h_{n}^{-1}} + \frac{\log (C_{K,2}^{1/2} /2\pi)}{\sqrt{2\log(1/h_{n})}}$,
and $C_{K,1},C_{K,2}$ are numerical constants that depend only on $K$; see \cite{BiDuHoMu07} for their explicit values. 
Furthermore, they show the validity of the nonparametric bootstrap for approximating the distribution of $\| \sqrt{n}h_{n}^{\alpha+1/2}(\hat{f}_{X}^{*} (\cdot) - \Ep[\hat{f}_{X}^{*}(\cdot)])/\sqrt{f_{Y}(\cdot)} \|_{[0,1]}$; see their Theorem 2. 

Since we work with a different setting from that of \cite{BiDuHoMu07} in the sense that we allow $f_{\varepsilon}$ to be unknown and an auxiliary sample from $f_{\varepsilon}$ is available, the regularity conditions in the present paper are not directly comparable to those of \cite{BiDuHoMu07}. 
However, it is worthwhile pointing out that conditions on the error characteristic function are significantly relaxed in the present paper. Indeed, their Assumption 2 is substantially more restrictive than our Assumption \ref{as: ill-posedness}. 
The reasons that they require their Assumption 2 are that: 1) they use the KMT strong approximation \citep{KoMaTu75} to the  empirical process $[0,1] \ni x \mapsto \sqrt{n}h_{n}^{\alpha+1/2} (\hat{f}_{X}^{*}(x) - \Ep[\hat{f}_{X}^{*}(x)])/\sqrt{f_{Y}(x)}$, for which a bound on the total variation of  $K_{n}$ is needed, and their Assumption 2 (a) plays that role; and 2) their analysis relies on the Gumbel approximation, for which they require further approximations based on the extreme value theory \citep[cf][]{LeLiRo83} beyond the KMT approximation, and consequently require some extra assumptions, namely, their Assumption 2 (b). 

In the present paper, we build upon the intermediate Gaussian and multiplier bootstrap approximation theorems  developed in \cite{ChChKa14a, ChChKa14b, ChChKa16}, and regularity conditions needed to apply those techniques are typically much weaker than those for the KMT and Gumbel approximations. In particular, we do not need a bound on the total variation of $K_{n}$; instead, we need that the class of functions $\{ y \mapsto K_{n}((x-y)/h_{n}) : x \in I \}$ is of Vapnik-Chervonenkis type, and to that end, thanks to Lemma 1 in \cite{GiNi09}, it is enough to prove that $K_{n}$ has a bounded \textit{quadratic variation} of order $h_{n}^{-2\alpha}$, which is ensured by our Assumptions \ref{as: kernel} and \ref{as: ill-posedness} (see Lemmas \ref{lem: VC type} and \ref{lem: quadratic variation} ahead). 
In addition, in contrast to \cite{BiDuHoMu07}, we do not need that $\sigma_{n}^{2}(x)/h_{n}^{-2\alpha+1}$ has a fixed limit; we only need that $\sigma_{n}^{2}(x)/h_{n}^{-2\alpha+1}$ is bounded away from zero uniformly in $x \in I$. 

Furthermore, the intermediate Gaussian and multiplier bootstrap approximations apply not only to the ordinary smooth case, but also to the super-smooth case, as discussed in Section \ref{sec: supersmooth}, and so they enable us to study confidence bands for $f_{X}$ in a unified way (although in the super-smooth case we require $m/n \to \infty$). On the other hand, as shown in \cite{EsGu08}, the Gumbel approximation does not hold for the super-smooth case in general (see also Remark \ref{rem: van Es} ahead). 

\cite{ScMuDu13}, assuming that the error density is known and ordinary smooth, develop methods to make inference on shape constraints for $f_{X}$, which also cover a construction of confidence bands (although their main interest is not in confidence bands). 
They use an intermediate Gaussian approximation different from ours based on the KMT approximation, and are able to relax assumptions in \cite{BiDuHoMu07}. Still, our conditions on the error characteristic function are weaker than theirs, since they further require that $\varphi_{\varepsilon}$ is twice differentiable and $|\varphi_{\varepsilon}''(t)|$ decays like $|t|^{-\alpha-2}$. 
Importantly, the crucial point of their approach is that the distribution of the approximating Gaussian process is known, which is the case when the distribution of $\varepsilon$ is known but not the case otherwise. Hence their methodology is not directly applicable to our case.
We also note that, in their methodology, $f_{Y}$ appears as a scaling constant, and so that we need to estimate  $f_{Y}$ separately and thus to choose an appropriate bandwidth for $f_{Y}$ separately. 
On the other hand, we are using a different scaling, and a separate estimation of $f_{Y}$ is not needed. 
\end{remark}

\section{Simulation studies}
\label{sec: numerical simulations}

\subsection{Simulation framework}

In this section, we present simulation studies to evaluate finite-sample performance of the  inference method developed in the previous two sections.
We generate data from the model introduced in Example \ref{ex: canonical}.
For distributions of the primitive latent variables $(X,\varepsilon^{(1)},\varepsilon^{(2)})$, we consider two alternative models described below.

In the first model, $X$ is drawn from the centered normal distribution $N(0,\sigma_X^2)$, and
$\varepsilon^{(1)}$ and $\varepsilon^{(2)}$ are drawn from the Laplace distribution with $(0,1)$ as the location and scale parameters.
This Laplace distribution is symmetric around zero, and therefore the premise of Example \ref{ex: canonical} regarding the error variables is satisfied.
The distribution of $\varepsilon = (\varepsilon^{(1)} + \varepsilon^{(2)})/2$ has its characteristic function not vanishing on $\R$ and is ordinary smooth with $\alpha=4$.
This setting conveniently yields the signal-to-noise ratio given by 
\[
\sqrt{\frac{\Var (X)}{\Var (\varepsilon)}} = \sqrt{\frac{\sigma_X^2}{\Var (\varepsilon^{(1)})/4 + \Var (\varepsilon^{(2)})/4}} = \sigma_X.
\]

In the second model, $X$ is drawn from the chi-squared distribution $\chi^2(df)$, and
$\varepsilon^{(1)}$ and $\varepsilon^{(2)}$ are drawn from the Laplace distribution with $(0,\sqrt{2})$ as the location and scale parameters.
The distribution of $\varepsilon = (\varepsilon^{(1)} + \varepsilon^{(2)})/2$ has its characteristic function not vanishing on $\R$ and is ordinary smooth with $\alpha=4$.
In this setting, the signal-to-noise ratio is given by 
\[
\sqrt{\frac{\Var (X)}{\Var (\varepsilon)}} = \sqrt{\frac{2df}{\Var (\varepsilon^{(1)})/4 + \Var (\varepsilon^{(2)})/4}} = \sqrt{df}.
\]
Table \ref{tab:summary_two_models} summarizes the two models and their relevant properties.

{\small
\begin{table}
	\centering
		\begin{tabular}{lccc}
		\hline\hline
		                                                          && Model 1 & Model 2\\
		\hline
			Distribution of $X$                                     && $N(0,\sigma_X^2)$ & $\chi^2(df)$ \\
			Distribution of $(\varepsilon^{(1)},\varepsilon^{(2)})$ && Laplace\,$(0,1)$ & Laplace\,$(0,\sqrt{2})$\\
			Smoothness of $X$                                       && Super & Ordinary \\
			Smoothness of $\varepsilon$                             && Ordinary & Ordinary \\
			Signal-to-noise ratio                                   && $\sigma_X$ & $\sqrt{df}$\\
                   Interval $I$                                                && $[-2\sigma_X, 2\sigma_X]$ & $[\mu_X/2, \mu_X + 2\sigma_X]$\\
		\hline\hline
		\end{tabular}
     \medskip
	\caption{A summary of the two models considered for simulation studies.}
	\label{tab:summary_two_models}
\end{table}
}

The observed portion of data, $\mathcal{D}_n = \left\{Y_1,\ldots,Y_n,\eta_1,\ldots,\eta_n\right\}$, is constructed by $Y_j = (Y^{(1)}_j + Y^{(2)}_j)/2$ and $\eta_j=(Y^{(1)}_j- Y^{(2)}_j)/2$, where $Y^{(1)}_j= X_j+ \varepsilon^{(1)}_j$ and $Y^{(2)}_j = X_j + \varepsilon^{(2)}_j$, for each $j=1,\dots,n$.
The three primitive latent variables, $X$, $\varepsilon^{(1)}$, and $\varepsilon^{(2)}$ are independently generated.
We use Monte Carlo simulations to compute the coverage probabilities of our multiplier bootstrap confidence bands for $f_X$ on the interval $I = [-2\sigma_X, 2\sigma_X]$ for Model 1 and on the interval $I = [\mu_X/2, \mu_X + 2\sigma_X]$ where $(\mu_X,\sigma_X)=(df,\sqrt{2df})$ for Model 2.
We use the kernel function $K$ defined by its Fourier transform $\varphi_{K}$ as follows: 
\[
\varphi_{K} (t) =
\begin{cases}
1 &\text{if } | t | \leq c\\
\exp\left\{ \frac{-b \exp(-b/(| t | - c)^2)}{(| t | - 1)^2} \right\} &\text{if } c < | t | < 1\\
0 &\text{if } |t| \geq 1
\end{cases}
,
\]
where  $b=1$ and $c=0.05$ \citep[cf.][]{McPo04, BiDuHoMu07}. 
Note that $\varphi_K$ is infinitely differentiable with support $[-1,1]$, 
and its inverse Fourier transform $K$ is real-valued and integrable with $\int_{\R} K(x) dx = 1$.
For the bandwidth selection, we follow a data-driven rule discussed in the next subsection, inspired by \cite{BiDuHoMu07}.

\subsection{Bandwidth selection}
\label{subsec: choice of the bandwidth}

Our theory prescribes admissible asymptotic rates for the bandwidth $h_n$ that require undersmoothing.
The literature provides data-driven approaches to bandwidth selection, which are usually based on minimizing the MISE. 
These data-driven approaches tend to yield non-under-smoothing bandwidths, and do not conform with our requirements.
We adopt the two-step selection method developed in \citet[][Section 5.2]{BiDuHoMu07} that aims to select undersmoothing bandwidths.
The first step selects a pilot bandwidth $h_n^P$ based on a data-driven approach.
We simply use a normal reference bandwidth \citep[][Section 3.1]{DeGi04} for $h_n^P$.
Once the pilot bandwidth $h_n^P$ is obtained, we next make a list of candidate bandwidths $h_{n,j} = (j/J) h_n^P$ for $j = 1,\ldots, J$.
The deconvolution estimate based on the $j$-th candidate bandwidth is denoted by $\hat f_{X,j}$.
The second step in the two step approach chooses the largest bandwidth $h_{n,j}$ such that the adjacent uniform distance $\| \hat f_{X,j-1} - \hat f_{X,j} \|_{I}$ is larger than $\rho \| \hat f_{X,J-1} - \hat f_{X,J} \|_{I}$ in the pilot case for some $\rho > 1$.
Similarly to the values recommended by \cite{BiDuHoMu07}, we find that $J \approx 20$ and $\rho \approx 3$ work well in our simulation studies.

\subsection{Simulation results}

Simulated uniform coverage probabilities are computed for each of the three nominal coverage probabilities, 80\%, 90\%, and 95\%, based on 2,000 Monte Carlo iterations.
In each run of the simulation, we generate 2,500 multiplier bootstrap replications given the observed data $\mathcal{D}_n$ to compute the estimated critical values, $\hat c_n(1-\tau)$.

\begin{table}[t]
	\centering
	\scalebox{.92}{
		\begin{tabular}{cccccccccc}
		\hline\hline
		  &&& \multicolumn{3}{c}{(A) Model 1} && \multicolumn{3}{c}{(B) Model 2} \\
			\cline{4-6}\cline{8-10}
			Nominal Coverage & Sample && \multicolumn{3}{c}{Signal-to-Noise Ratio} && \multicolumn{3}{c}{Signal-to-Noise Ratio}\\
			Probability ($1-\tau$) & Size ($n$) && 2.0 & 4.0 & 8.0 && 2.0 & 4.0 & 8.0 \\
		\hline
		  0.800            &   250  && 0.762 & 0.755 & 0.728  && 0.636 & 0.740 & 0.666\\
			                 &   500  && 0.786 & 0.746 & 0.763  && 0.698 & 0.751 & 0.685\\
											 & 1,000  && 0.784 & 0.750 & 0.754  && 0.732 & 0.749 & 0.702\\
		\hline
			0.900            &   250  && 0.870 & 0.866 & 0.843  && 0.760 & 0.853 & 0.799\\
			                 &   500  && 0.897 & 0.862 & 0.863  && 0.822 & 0.867 & 0.822\\
											 & 1,000  && 0.897 & 0.862 & 0.863  && 0.848 & 0.878 & 0.836\\
		\hline
			0.950            &   250  && 0.930 & 0.936 & 0.907  && 0.834 & 0.927 & 0.871\\
			                 &   500  && 0.951 & 0.929 & 0.923  && 0.834 & 0.928 & 0.897\\
											 & 1,000  && 0.942 & 0.933 & 0.933  && 0.912 & 0.940 & 0.906\\
		\hline\hline
		\end{tabular}
	}
     \medskip
	\caption{{\small Simulated uniform coverage probabilities of $f_X$ by estimated confidence bands in $I=[-2\sigma_X,2\sigma_X]$ for Model 1 and $I = [\mu_X/2, \mu_X + 2\sigma_X]$ for Model 2. The simulated probabilities are computed for each of the three nominal coverage probabilities, 80\%, 90\%, and 95\%, based on 2,000 Monte Carlo iterations.}}
	\label{tab:simulation_results}
\end{table}

Results under Models 1 and 2 are summarized in column groups (A) and (B), respectively, of Table \ref{tab:simulation_results} for each of the three different cases of the signal-to-noise ratio: $\sigma_X \in \{2.0, 4.0, 8.0\}$, and for each of the three sample sizes $n=m \in \{250, 500, 1,000\}$.
Observe that the simulated probabilities are close to the respective nominal probabilities.
Not surprisingly, the size tends to be more accurate for the results based on larger sample sizes.
The simulated probabilities are closer to the nominal probabilities in (A) than (B).

In addition to the size, we also analyze the power of uniform specification tests based on our uniform confidence band.
We now consider a list of alternative specifications of $f_X$ given by
$f_{X,\mu_X}(x) = (2\pi)^{-1/2} e^{-(x-\mu_{X})^{2}/2}$ for $\mu_X \in \{0.0,0.1,0.2,0.3,0.4,0.5\}$,
and likewise consider a list of alternative specifications given by
$f_{X,\sigma_X}(x) = (2\pi \sigma_{X}^{2})^{-1/2} e^{-x^{2}/(2\sigma_{X}^{2})}$ for $\sigma_X \in \{1.0,1.1,1.2,1.3,1.4,1.5\}$.
For the errors, we again consider the independent Laplace random vector $(\varepsilon^{(1)}, \varepsilon^{(2)})$ as in Model 1.
Figure \ref{fig:alternative_mu_coverage90} plots simulated coverage probabilities for the list of the alternative specifications of $f_{X,\mu_X}$ (top) and for the list of the alternative specifications of $f_{X,\sigma_X}$ (bottom) for the nominal coverage probability of $(1-\tau) = 0.90$.
The three curves are drawn for each of the three sample sizes $n \in \{250, 500, 1,000\}$.
Observe that, under the true specification (i.e., $\mu_X = 0.0$ in the top graph and $\sigma_X = 1.0$ in the bottom graph), the simulated coverage probabilities are close to the nominal coverage probability of $0.90$, with the case of $n=1,000$ being the closest and the case of $n=250$ being the farthest.
On the other hand, as the specification deviates away from the truth (i.e., as $\mu_X$ or $\sigma_X$ increases), the nominal coverage probabilities decrease, with the case of $n=1,000$ being the fastest and the case of $=250$ being the slowest.
These results evidence the power as well as the size of the uniform specification tests.

\begin{figure}
	\centering
		\includegraphics[width=0.75\textwidth]{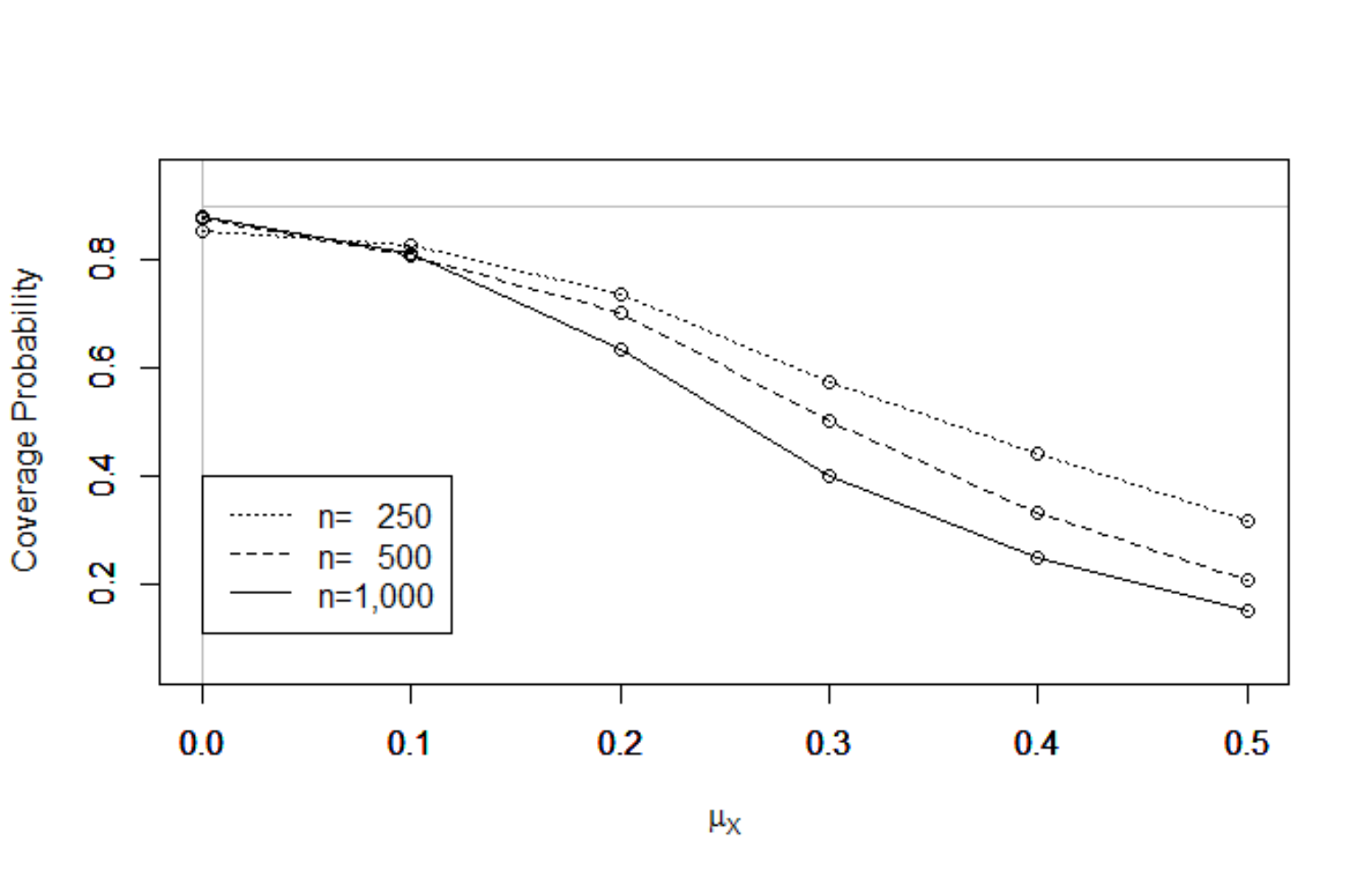}\\
		\includegraphics[width=0.75\textwidth]{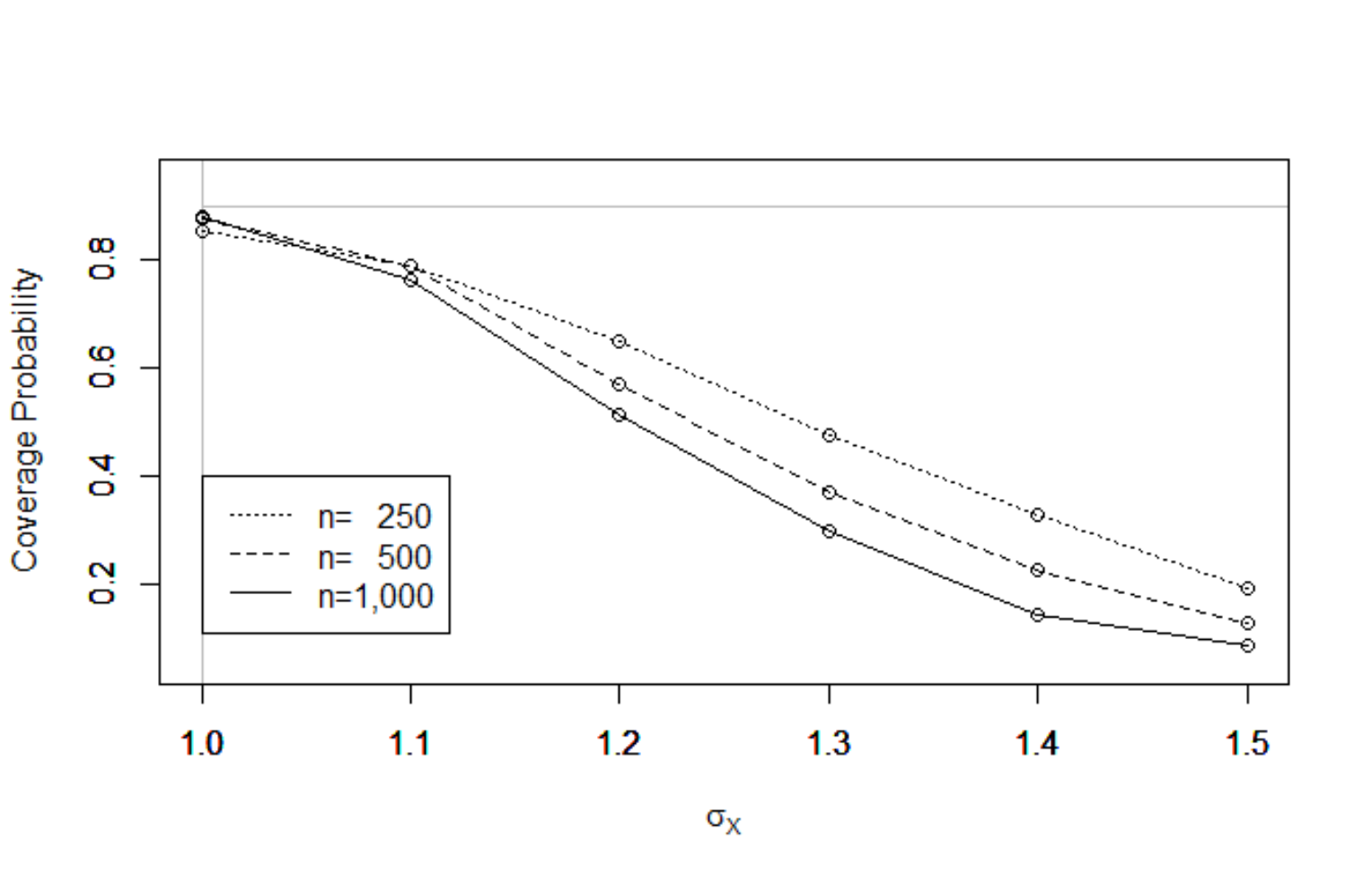}
	\caption{{\small Simulated uniform coverage probabilities for alternative specifications of $f_X$  by estimated confidence bands in $I=[-2,2]$ under the standard normal random variable $X$ and an independent Laplace random vector $(\varepsilon^{(1)}, \varepsilon^{(2)})$. For the top graph, the list of alternative specifications are given by $f_{X,\mu_X}(x) = (2\pi)^{-1/2} e^{-(x-\mu_{X})^{2}/2}$ for $\mu_X \in \{0.0,0.1,0.2,0.3,0.4,0.5\}$. 
For the bottom graph, the list of alternative specifications are given by $f_{X,\sigma_X}(x) = (2\pi \sigma_{X}^{2})^{-1/2} e^{-x^{2}/(2\sigma_{X}^{2})}$ for $\sigma_X \in \{1.0,1.1,1.2,1.3,1.4,1.5\}$. The simulated probabilities are computed for the nominal coverage probability of 90\% based on 2,000 Monte Carlo iterations.}}
	\label{fig:alternative_mu_coverage90}
\end{figure}

\section{Application to OCS Wildcat auctions}
\label{sec: application}

In this section, we apply our method to the Outer Continental Shelf (OCS) Auction Data (see \cite{HePoBo87} for details), and construct a confidence band for the density of mineral rights on oil and gas on offshore lands off the coasts of Texas and Louisiana in the gulf of Mexico.
We focus on ``wildcat sales,'' referring to sales of those oil and gas tracts whose geological or seismic characteristics are unknown to participating firms.
The sales rule follows the first-price sealed-bid auction mechanism, where participating firms simultaneously submit sealed bids, and the highest bidder pays the price they submitted to receive the right for the tract.
Firms who are willing to participate in sales can carry out a seismic investigation before the sales date in order to estimate the value of mineral rights.
The ex ante value $Y^{(1)}$ (in the logarithm of US dollars per acre) obtained by firm 1 through its investigation of the tract is treated as a measure of the \text{ex post} value $X$ (also known as the common component, in the logarithm of US dollars per acre) with an assessment error $\varepsilon^{(1)}$ (also known as the private component), i.e., $Y^{(1)} = X + \varepsilon^{(1)}$.
Collecting the ex ante values $(Y^{(1)},Y^{(2)})$ for pairs of firms across various wildcat auctions, we can obtain data necessary to construct a uniform confidence band for the density $f_X$ of ex post mineral right values under our assumptions.

For this setup and for this data set, \cite{LiPeVu00} apply the method of \cite{LiVu98} to nonparametrically estimate $f_X$, but they do not obtain a confidence band.
In their analysis, firms' ex ante values $(Y^{(1)},Y^{(2)})$ are first recovered from bid data through a widely used method in economics which is based on an equilibrium restriction (Bayesian Nash equilibrium) for the first-price sealed-bid auction mechanism -- see our supplementary material.
In this paper, we directly take these ex ante values $(Y^{(1)},Y^{(2)})$ as the data to be used as an input for our analysis.
The sample consists of 169 tracts with 2 firms in each tract.
We next construct our auxiliary data $(Y,\eta)$ following Example \ref{ex: canonical}.

We continue to use the same kernel function and the same bandwidth selection rule as those ones used for simulation studies in Section \ref{sec: numerical simulations}.
Confidence bands for $f_X$ are constructed using 25,000 multiplier bootstrap replications.
Figure \ref{fig:application_band} (A) shows the obtained 90\% and 95\% confidence bands in light gray and dark gray, respectively.
The black curve draws our nonparametric estimate of $f_X$, and runs at the center of the bands.
Our nonparametric estimate, our confidence bands, and the nonparametric estimate obtained by \cite{LiPeVu00} -- shown in their Figure 4 -- are all similar to each other, and share the same qualitative characteristics.
First, unlike the bimodal density of the private values $(Y^{(1)},Y^{(2)})$ (see Figure \ref{fig:application_data} in the supplementary material), the density of the common value $X$ is suggested to be single-peaked in all the results of our analysis and that of \cite{LiPeVu00}.
Second, as \cite{LiPeVu00} also emphasized, the small bump in the estimated density $f_X$ around $x=6$ is common in all the results.
Furthermore, our confidence bands do not include zero at this locality, $x=6$.
This result can be viewed as a statistical evidence in support of a significant presence of such a bump pointed out by \cite{LiPeVu00}.
A more global look into the graph suggests that the 95\% confidence band is bounded away from zero on the interval $[3.90,6.32]$.
In other words, we can conclude that the ex post value of mineral rights in the US dollars per acre is supported on a superset of the interval $[\exp(3.90),\exp(6.32)] \approx [49,555]$, if we consider a set of density functions contained in the 95\% confidence band.

\begin{figure}
	\centering
	\begin{tabular}{cc}
		(A) Unknown Error Distribution & (B) Laplace Error Distribution\\
		\includegraphics[width=0.5\textwidth]{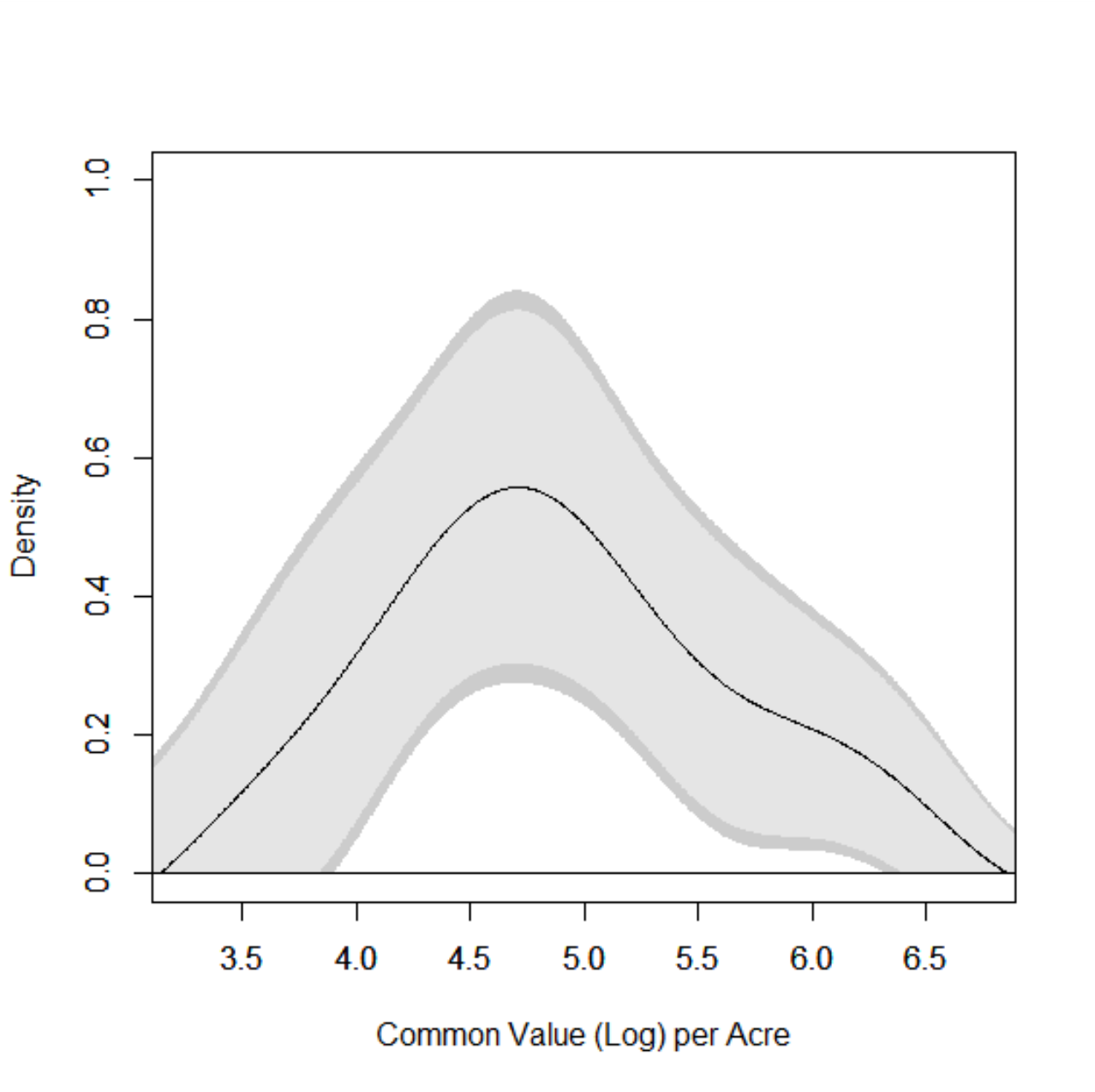} &
		\includegraphics[width=0.5\textwidth]{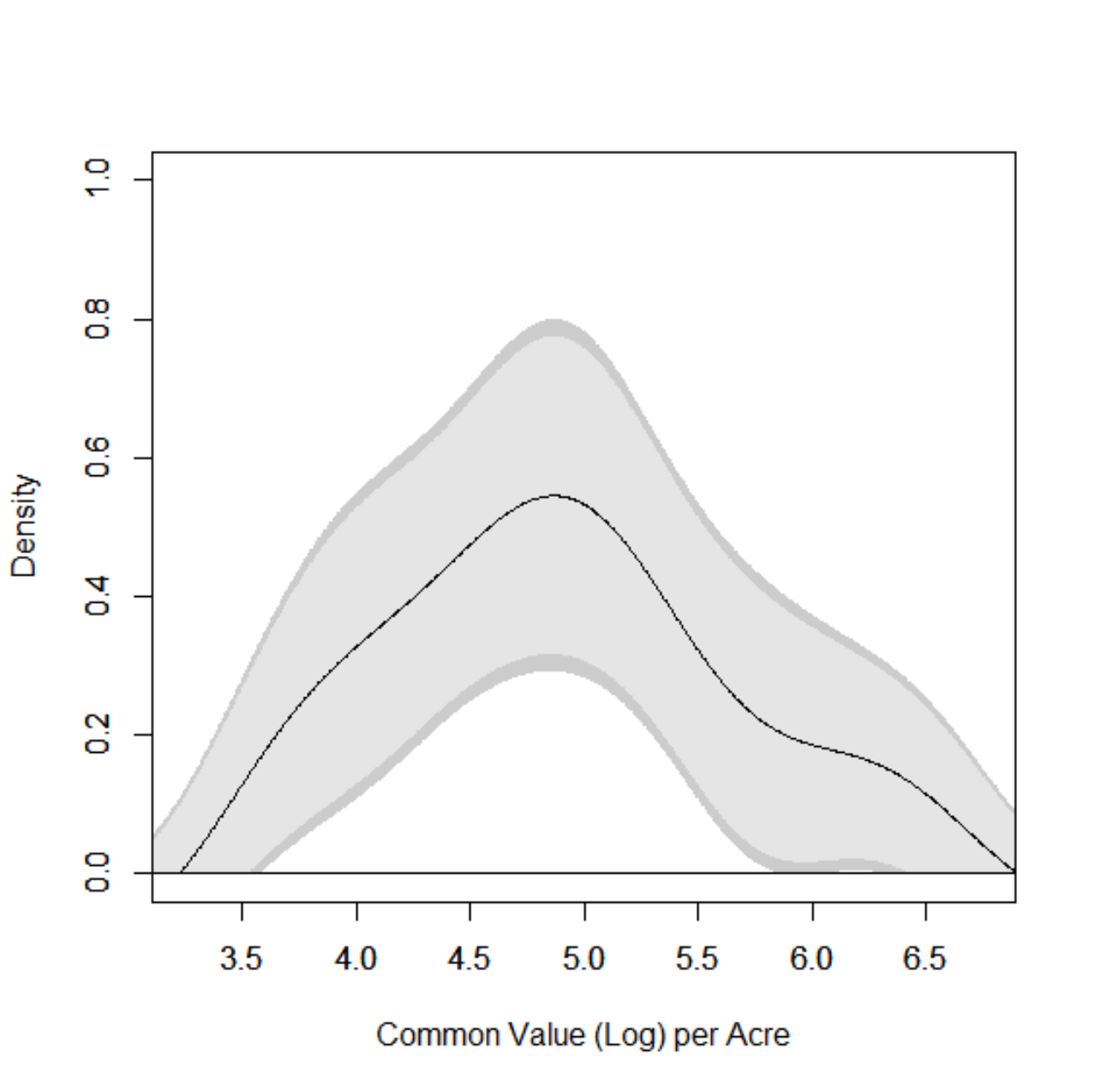}
	\end{tabular}
	\caption{{\small 90\% and 95\% confidence bands for $f_X$ marked in light gray and dark gray, respectively. The black curve in the middle indicates our  deconvolution kernel density estimate of $f_X$.}}
	\label{fig:application_band}
\end{figure}

Finally, 
we show an empirical result that we would obtain under the assumption that the distribution of $\varepsilon$ were known as in the previous literature.
Specifically, for this exercise, we assume that $\varepsilon$ follows $\text{Laplace}\,(0,\sqrt{\sigma_\varepsilon^2/2})$, where $\sigma_\varepsilon^2 = 0.168$ -- this number is the sample variance of $\eta_j$.
Figure \ref{fig:application_band} (B) shows the result.
There are some notable differences from Figure \ref{fig:application_band} (A).
Most importantly, the 95\% confidence band now contains zero around the locality, $x=6$, of the aforementioned bump in the estimated density of $X$.
The contrast between Figure \ref{fig:application_band} (A) and (B) shows that there can be non-trivial differences in statistical implications of confidence bands between the case of assuming known error distribution and the case of assuming unknown error distribution.

\section{Application to panel data}
\label{sec: panel}

Example \ref{ex: canonical} demonstrates that the availability of repeated measurements with a symmetric error distribution satisfy our data requirement.
A particular example of this case is the additive panel data model with fixed effects as in \citet{HoMa96}:
\[
Y_{j,t} = W_{j,t}'\theta_{0} + U_{j} + V_{j,t}, \quad j=1,\dots,n; t=1,2,
\]
where $Y_{j,t}$ is a scalar outcome variable, $W_{j,t}$ is a $d$-dimensional vector of regressors, $\theta_{0} \in \R^{d}$ is the slope parameter, $U_{j}$ is an unobservable individual-specific effect, and $V_{j,t}$ is an error term. 
Inference on the density function of unobservables in this sort of panel data models is of interest in empirical research \citep[e.g.,][]{BoRo10,BoSa11}.
\citet{HoMa96} or other related panel data papers do not provide asymptotic distribution results for the density function of unobservables to our knowledge.
Our method applies to inference on the density of $U_j$.

\subsection{Methodology for panel data}

We assume that 
\begin{equation}
\label{eq: panel data condition}
\begin{cases}
\text{1) $\{ (Y_{j,1},Y_{j,2},W_{j,1},W_{j,2},U_{j},V_{j,1},V_{j,2}) \}_{j=1}^{\infty}$ are i.i.d.}; \\
\text{2) $(V_{j,1},V_{j,2})$ is independent of $(W_{j,1},W_{j,2},U_{j})$}; \\
\text{3) the conditional distribution of $V_{j,2}$ given $V_{j,1}$ is symmetric}. 
\end{cases}
\end{equation}
Independence between $(V_{j,1},V_{j,2})$ and $(W_{j,1},W_{j,2})$ can be removed at the expense of more complicated regularity conditions, but we assume the independence assumption for the simplicity of exposition; see Remark \ref{rem: independence} ahead. \cite{HoMa96} assume that $V_{j,1}$ and $V_{j,2}$ are i.i.d. and the common distribution is symmetric (see their Condition A.1); in contrast, we do not assume that $V_{j,1}$ and $V_{j,2}$ are i.i.d., nor did we assume symmetry of both distributions of $V_{j,1}$ and $V_{j,2}$. 
Now, consider the following transformations:
\[
\begin{cases}
Y_{j}^{\dagger} = (Y_{j,1}+Y_{j,2})/2 - \theta_{0}'(W_{j,1}+W_{j,2})/2, \\
 \varepsilon_{j} = V_{j,1}+V_{j,2})/2, \\
 \eta_{j}  = (V_{j,1}-V_{j,2})/2= (Y_{j,1}-Y_{j,2})/2 - \theta_{0}'(W_{j,1}-W_{j,2})/2.
\end{cases}
\]
Observe that 
\[
Y_{j}^{\dagger} = U_{j} + \varepsilon_{j}, \quad \eta_{j} \stackrel{d}{=} \varepsilon_{j}.
\]
We assume that the densities of $U_{j}$ and $\varepsilon_{j}$ exist and are denoted by $f_{U}$ and $f_{\varepsilon}$, respectively. 
The density of $Y_{j}^{\dagger}$ is given by 
\[
f_{Y^{\dagger}} (u) =(f_{U}*f_{\varepsilon})(u) =  \int_{\R} f_{U}(u-y) f_{\varepsilon}(y) dy.
\]
Hence, estimation of the density $f_{U}$ reduces to a deconvolution problem. 
The difference from the original setup is that $\theta_{0}$ is unknown and has to be estimated. 
We assume, as in \citet{HoMa96}, that there is an estimator $\hat{\theta}$ of $\theta_{0}$ such that $\hat{\theta} - \theta_{0} = O_{\Pr}(n^{-1/2})$, and let 
\[
\begin{cases}
\hat{Y}_{j}^{\dagger} = (Y_{j,1}+Y_{j,2})/2 - \hat{\theta}'(W_{j,1}+W_{j,2})/2, \\
\hat{\eta}_{j} =   (Y_{j,1}-Y_{j,2})/2 - \hat{\theta}'(W_{j,1}-W_{j,2})/2.
\end{cases}
\]
Define
\begin{align*}
&\varphi_{Y^{\dagger}} (t) = \Ep[ e^{itY_{j}^{\dagger}}], \ \varphi_{U}(t)  = \Ep[ e^{itU_{j}} ], \ \varphi_{\varepsilon}(t) = \Ep[ e^{it\varepsilon_{j}}], \\
&\hat{\varphi}_{Y^{\dagger}}^{*}(t) = \frac{1}{n}\sum_{j=1}^{n} e^{itY_{j}^{\dagger}}, \ \hat{\varphi}_{Y^{\dagger}} (t) = \frac{1}{n}\sum_{j=1}^{n} e^{it\hat{Y}_{j}^{\dagger}}, \ \hat{\varphi}_{\varepsilon}^{*}(t) = \frac{1}{n}\sum_{j=1}^{n} e^{it\eta_{j}}, \ \hat{\varphi}_{\varepsilon}(t) = \frac{1}{n} \sum_{j=1}^{n} e^{it\hat{\eta}_{j}}.
\end{align*}
Let $K: \R \to \R$ be a kernel function such that its Fourier transform $\varphi_{K}$ is supported in $[-1,1]$.
The deconvolution kernel density estimator of $f_{U}$ is given by 
\[
\hat{f}_{U}(u) = \frac{1}{2\pi} \int_{\R} e^{-itu} \hat{\varphi}_{Y^{\dagger}} (t) \frac{\varphi_{K}(th_{n})}{\hat{\varphi}_{\varepsilon}(t)} dt = \frac{1}{nh_{n}} \sum_{j=1}^{n} \hat{K}_{n}((u-\hat{Y}_{j}^{\dagger})/h_{n}),
\]
where  $h_{n}$ is a sequence of bandwidths tending to $0$ as $n \to \infty$, and 
\[
\hat{K}_{n}(u) = \frac{1}{2\pi} \int_{\R} e^{-itu} \frac{\varphi_{K}(t)}{\hat{\varphi}_{\varepsilon}(t/h_{n})} dt. 
\]
The rest of the procedure is the same as in Section \ref{sec: methodology}. Let $\xi_{1},\dots,\xi_{n}$ be independent standard normal variables independent of the data $\mathcal{D}_{n}= \{ (Y_{j,1},Y_{j,2},W_{j,1},W_{j,2}) \}_{j=1}^{n}$, and consider the multiplier process
\[
\hat{Z}_{n}^{\xi} (u) = \frac{1}{\hat{\sigma}_{n}(u)\sqrt{n}} \sum_{j=1}^{n} \xi_{j} \left \{  \hat{K}_{n} ((u-\hat{Y}^{\dagger}_{j})/h_{n}) - n^{-1} {\textstyle \sum}_{j'=1}^{n} \hat{K}_{n}((u-\hat{Y}^{\dagger}_{j'})/h_{n}) \right \}, \ u \in I,
\]
where $I \subset \R$ is a compact interval on which we would like to make inference on $f_{U}$, and 
\[
\hat{\sigma}_{n}^{2}(u) = \frac{1}{n} \sum_{j=1}^{n} \hat{K}_{n}^{2}((u-\hat{Y}_{j}^{\dagger})/h_{n}) - \left ( \frac{1}{n} \sum_{j=1}^{n} \hat{K}_{n}((u-\hat{Y}_{j}^{\dagger})/h_{n}) \right )^{2}, \ u \in I. 
\]
Now, for a given $\tau \in (0,1)$, let
\[
\hat{c}_{n}(1-\tau) = \text{conditional $(1-\tau)$-quantile of $\| \hat{Z}_{n}^{\xi} \|_{I}$ given $\mathcal{D}_{n}$},
\]
and consider the confidence band 
\begin{equation}
\hat{\mathcal{C}}_{n} (u)= \left [ \hat{f}_{U}(u) \pm \frac{\hat{\sigma}_{n}(u)}{\sqrt{n}h_{n}} \hat{c}_{n}(1-\tau)\right ], \ u \in I. 
\label{eq: MB confidence band panel}
\end{equation}

We make the following assumption for the validity of the confidence band (\ref{eq: MB confidence band panel}). 

\begin{assumption}
\label{as: panel}
In addition to the baseline condition (\ref{eq: panel data condition}), we assume the following conditions. (i) The function $|\varphi_{U}|$ is integrable on $\R$. (ii) $f_{U} \in \Sigma (\beta,B)$ for some $\beta > 1/2$ and $B>0$. Let $k$ denote the integer such that $k < \beta \leq \beta + 1$. (iii) Let $K: \R \to \R$ be a kernel function such that its Fourier transform $\varphi_{K}$ is supported in $[-1,1]$ and Condition (\ref{eq: higher order kernel}) is satisfied. (iv) The error characteristic function  $\varphi_{\varepsilon}$ is continuously differentiable and does not vanish on $\R$, and 
there exist constants $\alpha > 1/2$ and $C_{1} > 1$ such that $
C_{1}^{-1} |t|^{-\alpha} \leq | \varphi_{\varepsilon} (t) | \leq C_{1} |t|^{-\alpha}$ and $| \varphi_{\varepsilon}'(t) | \leq C_{1}|t|^{-\alpha-1}$ for all $|t| \geq 1$. (v) For some $p > 0$, $\Ep[ |\varepsilon_{1}|^{p} ] < \infty$ and $\Ep[ (1+|Y_{1}^{\dagger}|^{p}) (\| W_{1,1} \|^{2} + \| W_{1,2} \|^{2})] < \infty$. 
(vi) Let $\hat{\theta}$ be an estimator for $\theta_{0}$ such that $\hat{\theta} - \theta_{0} = O_{\Pr}(n^{-1/2})$. 
(vii) Let $I$ be a compact interval in $\R$ such that $f_{Y^{\dagger}}(u) > 0$ for all $u \in I$. (viii) 
\begin{equation}
\frac{\log h_{n}^{-1}}{nh_{n}^{5}} \bigvee \frac{(\log h_{n}^{-1})^{2}}{nh_{n}^{2\alpha+2}} \bigvee h_{n}^{\alpha+\beta}  \sqrt{n h_{n} \log h_{n}^{-1}} \to 0. \label{eq: bandwidth panel}
\end{equation}
Furthermore, for $W_{1,t} = (W_{1,t,1},\dots,W_{1,t,d})'$, 
\begin{equation}
\int_{-h_{n}^{-1}}^{h_{n}^{-1}}  | s \Ep[ e^{isU_{1}} W_{1,t,\ell}] | ds = o \{ h_{n}^{-\alpha-1/2} (\log h_{n}^{-1})^{-1/2} \}, \ t=1,2; \ell=1,\dots,d. \label{eq: additional condition}
\end{equation}
\end{assumption}

These conditions ensure the asymptotic validity of the confidence band (\ref{eq: MB confidence band panel}). 

\begin{theorem}
\label{thm: validity of MB panel}
Under Assumption \ref{as: panel}, we have that $\Pr \{ f_{U}(u) \in \hat{\mathcal{C}}_{n} (u) \ \forall u \in I \} \to 1-\tau$ as $n \to \infty$. Furthermore, the supremum width of the band $\hat{\mathcal{C}}_{n}$ is $O_{\Pr}\{ h_{n}^{-\alpha} (nh_{n})^{-1/2} \sqrt{\log h_{n}^{-1}}\}$. 
\end{theorem}

\begin{remark}[Discussions on Assumption \ref{as: panel}]
These conditions are mostly adapted from the conditions given in Section \ref{sec: main} with $m = n$. 
We assume here that $\alpha > 1/2$ for a technical reason to bound the impact of the estimation error in $\hat{\theta}$ on $\hat{f}_{U}$. Condition (\ref{eq: bandwidth panel}) restricts $(\alpha,\beta)$ so that $\alpha+\beta > 2$, which we believe is a mild restriction. Condition (\ref{eq: additional condition}) is another technical condition to deal with the impact of the estimation error in $\hat{\theta}$. Condition (\ref{eq: additional condition}) is not restrictive; in general, $|\Ep[ e^{isU_{1}} W_{1,t,\ell}]| \leq \Ep[ | W_{1,t,\ell} |] < \infty$, so that the left hand side on (\ref{eq: additional condition}) is at most $O(h_{n}^{-2})$ and hence Condition (\ref{eq: additional condition}) is satisfied as long as $\alpha > 3/2$. However, Condition (\ref{eq: additional condition}) can be satisfied without such restrictions on $\alpha$. In many cases, $|\Ep [ e^{isU_{1}}W_{1,t,\ell}]|$ decays to zero as $|s| \to \infty$, which is the case if, e.g., $U_{1}$ and $W_{1,t,\ell}$ have a joint density, or $W_{1,t,\ell}$ is finitely discrete and the conditional distribution of $U_{1}$ given $W_{1,t,\ell}$ is absolutely continuous. So, if $|\Ep[ e^{isU_{1}} W_{1,t,\ell}]| = O(|s|^{-\gamma})$ as $|s| \to \infty$ for some $\gamma > 0$, then the left hand side on (\ref{eq: additional condition}) is 
\[
\begin{cases}
O(h_{n}^{\gamma-2}) & \text{if} \ \gamma < 2 \\
O(\log h_{n}^{-1}) & \text{if} \ \gamma=2 \\
O(1) & \text{if} \ \gamma > 2
\end{cases}
\]
and hence Condition (\ref{eq: additional condition}) is satisfied as long as $\gamma > 3/2-\alpha$. 
\end{remark}

\begin{remark}[Independence between $(V_{j,1},V_{j,2})$ and $(W_{j,1},W_{j,2})$]
\label{rem: independence}
Condition (\ref{eq: panel data condition}) assumes that $(V_{j,1},V_{j,2})$ and $(U_{j},W_{j,1},W_{j,2})$ are independent. Independence between $(V_{j,1},V_{j,2})$ and $(W_{j,1},W_{j,2})$ can be removed at the cost of more complicated regularity conditions. In the proof of Theorem \ref{thm: validity of MB panel}, this independence assumption is used to deduce that
\begin{equation}
\label{eq: implication of independence}
\Ep[ e^{itY_{1}^{\dagger}} W_{1,\ell}^{+} ] = \varphi_{\varepsilon} (t)\Ep[ e^{itU_{1}} W_{1,\ell}^{+}], \ \Ep[ e^{it\eta_{1}} W_{1,\ell}^{-}] =  \varphi_{\varepsilon}(t) \Ep[ W_{1,\ell}^{-}],
\end{equation}
where $W_{1,\ell}^{+} = (W_{1,1,\ell}+W_{1,2,\ell})/2$ and $W_{1,\ell}^{-} = (W_{1,1,\ell}-W_{1,2,\ell})/2$ for $\ell=1,\dots,d$. Now, without requiring independence between $(V_{j,1},V_{j,2})$ and $(W_{j,1},W_{j,2})$ (thereby (\ref{eq: implication of independence}) need not hold), the conclusion of Theorem \ref{thm: validity of MB panel} remains true if $\left | \frac{\Ep[ e^{it\eta_{1}} W_{1,\ell}^{-}]}{\varphi_{\varepsilon}(t)} \right |$
is bounded in $t \in \R$, and instead of (\ref{eq: additional condition}), 
\[
\int_{-h_{n}^{-1}}^{h_{n}^{-1}} \left | \frac{t\Ep[ e^{itY_{1}^{\dagger}} W_{1,\ell}^{+} ]}{\varphi_{\varepsilon}(t)} \right | dt = o \{ h_{n}^{-\alpha-1/2} (\log h_{n}^{-1})^{-1/2} \}
\]
holds for $\ell=1,\dots,d$.
\end{remark}

\subsection{Empirical illustration with panel data}

\citet{LePe03} analyze Chilean industries for the period of 1979--1986.
Following up with their studies, we analyze the distribution of the total factor productivity across firms by applying the method introduced in the previous subsection to the data set of \citeauthor{LePe03}.
We focus on the food industry, which is the largest industry in Chile among those studied by \citeauthor{LePe03}.
The data set is an unbalanced panel of $t=8$ years.
The sample sizes are
$n=709$ firms between 1979--1980,
$n=751$ firms between 1980--1981,
$n=721$ firms between 1981--1982,
$n=691$ firms between 1982--1983,
$n=639$ firms between 1984--1985, and
$n=618$ firms between 1985--1986.

Let $Y_{j,t}$ denote the logarithm of output produced by firm $j$ in year $t$.
The output is produced by using 
unskilled labor inputs denoted in logarithm by $W^{l^u}_{j,t}$,
skilled labor inputs denoted in logarithm by $W^{l^s}_{j,t}$,
capital inputs denoted in logarithm by $W^{k}_{j,t}$,
material inputs denoted in logarithm by $W^{m}_{j,t}$,
electricity inputs denoted in logarithm by $W^{e}_{j,t}$, and
fuel inputs denoted in logarithm by $W^{u}_{j,t}$.
In addition, we include two time-period dummies,
$W^{d^1}_{j,t}$ and $W^{d^2}_{j,t}$ for 1982--1983 and 1984--1986, respectively, following the time periods defined by \cite{LePe03}.
Gross-output production function in logs is written as
$$
Y_{j,t} = W_{j,t}' \theta_{0} + \omega_{j,t} + \eta_{j,t}
$$
where
$W_{j,t} = (W^{l^u}_{j,t},W^{l^s}_{j,t},W^{k}_{j,t},W^{m}_{j,t},W^{e}_{j,t},W^{u}_{j,t},W^{d^1}_{j,t},W^{d^2}_{j,t})$,
$\theta_{0} = (\theta_{lu},\theta_{ls},\theta_k,\theta_m,\theta_e,\theta_u,\theta_{d^1},\theta_{d^2})'$,
$\omega_{j,t}$ denotes the productivity, and
$\eta_{j,t}$ denotes an idiosyncratic shock.

Rational firms accumulate state variables and make static input choices endogenously in response to the current and past productivity levels, and thus $W_{j,t}$ is not statistically independent of $\omega_{j,t}$.
Under the presence of this endogeneity, various approaches \citep{OlPa96,LePe03,AcCaFr06,Wo09} are developed for identification and consistent estimation of the production function parameters $\theta_{0}$.
We use the GMM criterion of \cite{Wo09} to estimate these parameters with the third degree polynomial control of $\omega_{j,t}$ and with $W^{m}_{j,t}$ as a proxy following \citeauthor{LePe03} by pooling all the observations in the data.
Let the estimate be denoted by $\hat{\theta}$.

Focusing on any pair of two adjacent time periods, we can rewrite the gross-output production function as a panel data model with fixed effects as follows.
$$
Y_{j,t} = W_{j,t}'\theta_{0} + U_j + V_{j,t}, \qquad j = 1,...,n; t = 1,2,
$$
where $U_j = \omega_{j,1}$, $V_{j,1} = \eta_{j,1}$ and $V_{j,2} = \omega_{j,2}-\omega_{j,1}+\eta_{j,2}$.
Following the method proposed in the previous subsection, we construct the auxiliary variables
\begin{align*}
Y_j^\dagger &= (Y_{j,1} + Y_{j,2})/2 - \theta_{0}'(W_{j,1} + W_{j,2})/2 = U_j + (V_{j,1} + V_{j,2})/2
\\
\eta_j &= (Y_{j,1} - Y_{j,2})/2 - \theta_{0}'(W_{j,1} - W_{j,2})/2 = (V_{j,1} - V_{j,2})/2
\end{align*}
The independence condition between $U_{j}$ and $(V_{j,1},V_{j,2})$ is satisfied if
\begin{enumerate}[(i)]
	\item the productivity $\omega_{j,1}$ and the idiosyncratic shocks $(\eta_{j,1},\eta_{j,2})$ are independent; and
	\item the productivity $\omega_{j,1}$ and the productivity innovation $\omega_{j,2} - \omega_{j,1}$ are independet.
\end{enumerate}
These conditions are assumed in the aforementioned production function papers, and we thus maintain this primitive assumption in order to satisfy our high-level independence condition.
We substitute the above estimate $\widehat\theta$ for $\theta_{0}$ and use
\begin{align*}
\widehat Y_j^\dagger &= (Y_{j,1} + Y_{j,2})/2 - \hat{\theta}'(W_{j,1} + W_{j,2})/2
\\
\widehat \eta_j &= (Y_{j,1} - Y_{j,2})/2 - \hat{\theta}'(W_{j,1} - W_{j,2})/2
\end{align*}
to apply our method.

We continue to use the same kernel function and the same bandwidth selection rule as those
ones used for simulation studies in Section 5.
Confidence bands for $f_{\omega_{j,t}}$ are constructed for each $t=$1979--1985 using
25,000 multiplier bootstrap replications.
Figure \ref{fig:productivity1} shows the 90\% and 95\% confidence bands as well as the estimate for $f_{\omega_{j,t}}$ for years $t=$1979--1982.
Figure \ref{fig:productivity2} shows the 90\% and 95\% confidence bands as well as the estimate for $f_{\omega_{j,t}}$ for years $t=$1982--1985.
Observe that the distribution of the productivities is shifting to the right as time progresses.
Furthermore, the constructed confidence bands informatively indicate the possible densities accounting for uncertainties in data sampling.

\begin{figure}
	\centering
		\includegraphics[width=0.67\textwidth]{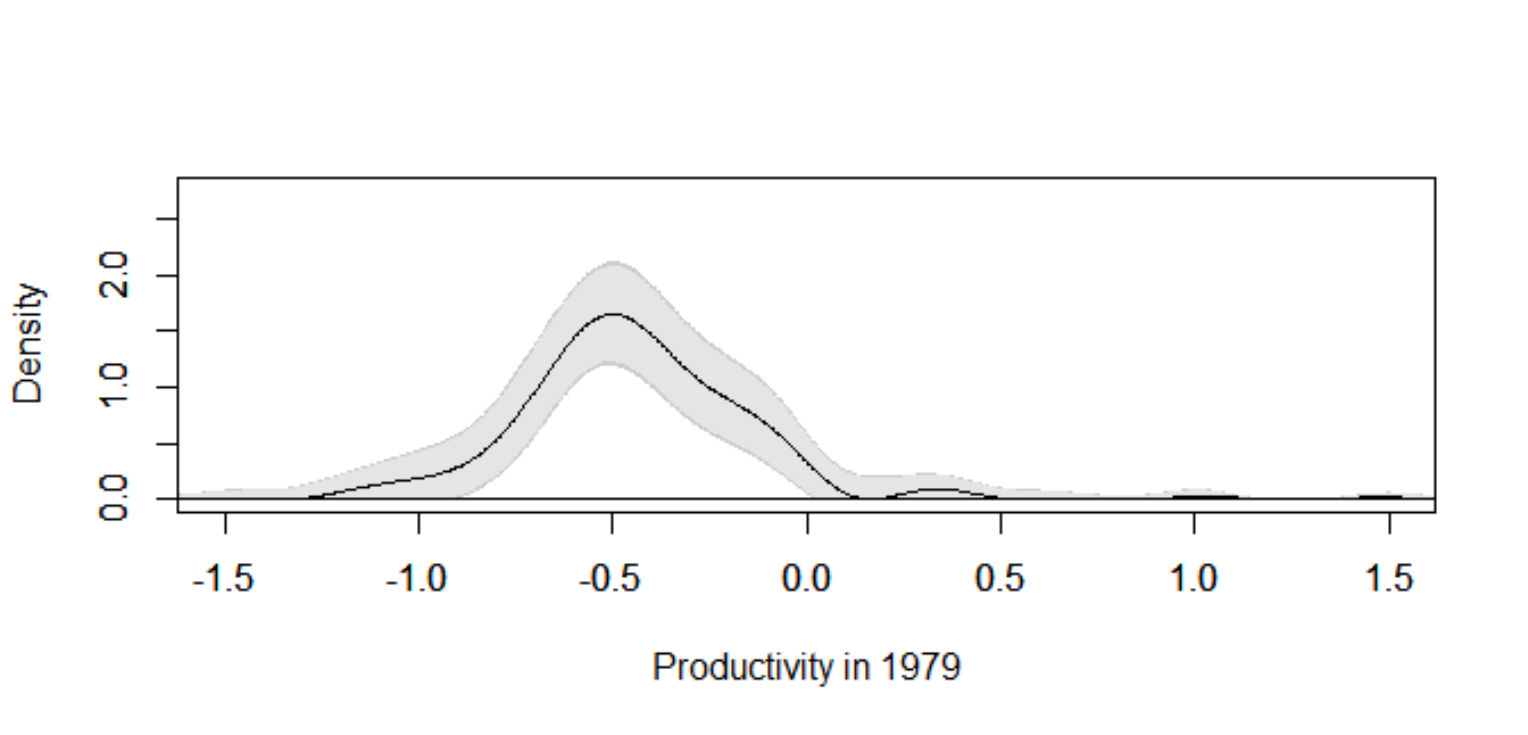}\\
		\includegraphics[width=0.67\textwidth]{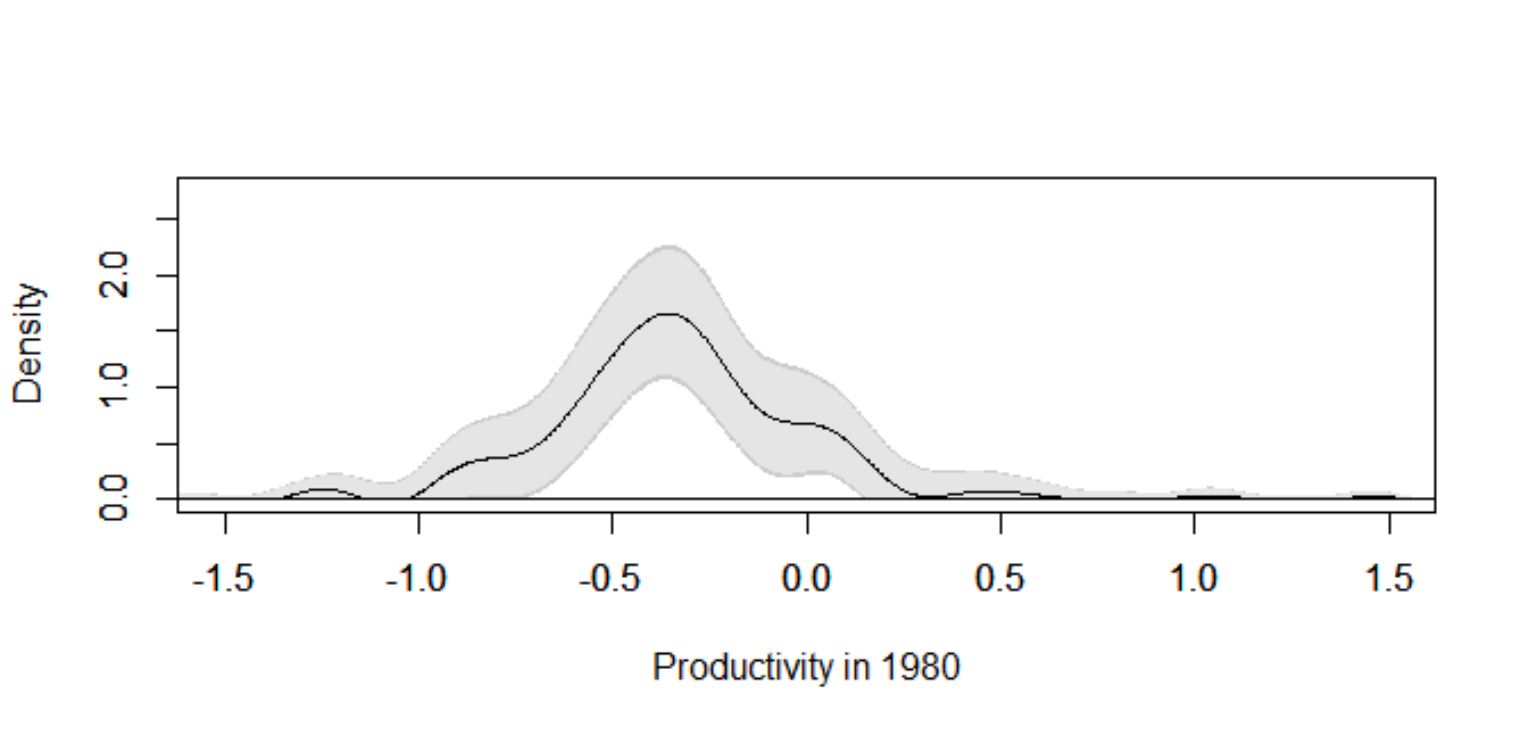}\\
		\includegraphics[width=0.67\textwidth]{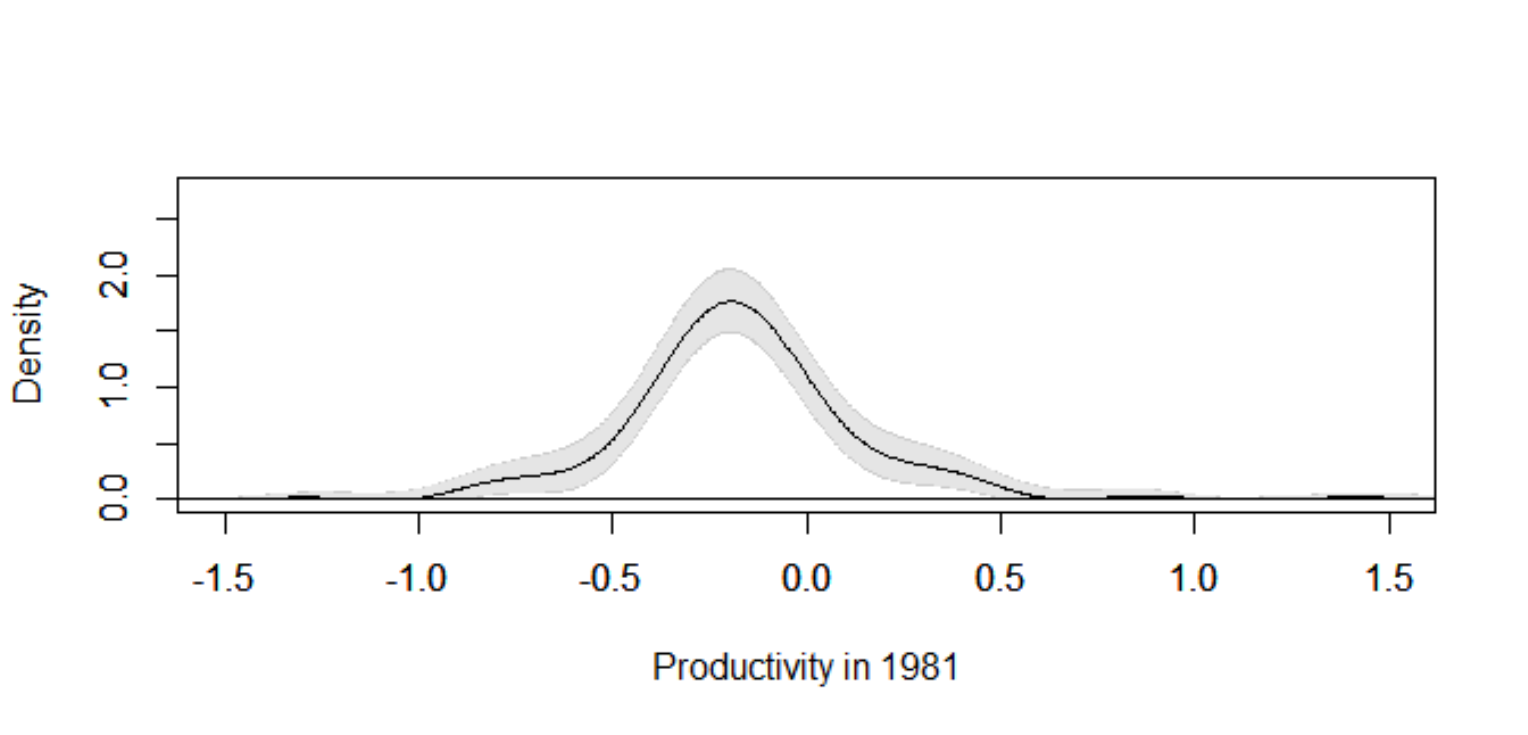}\\
		\includegraphics[width=0.67\textwidth]{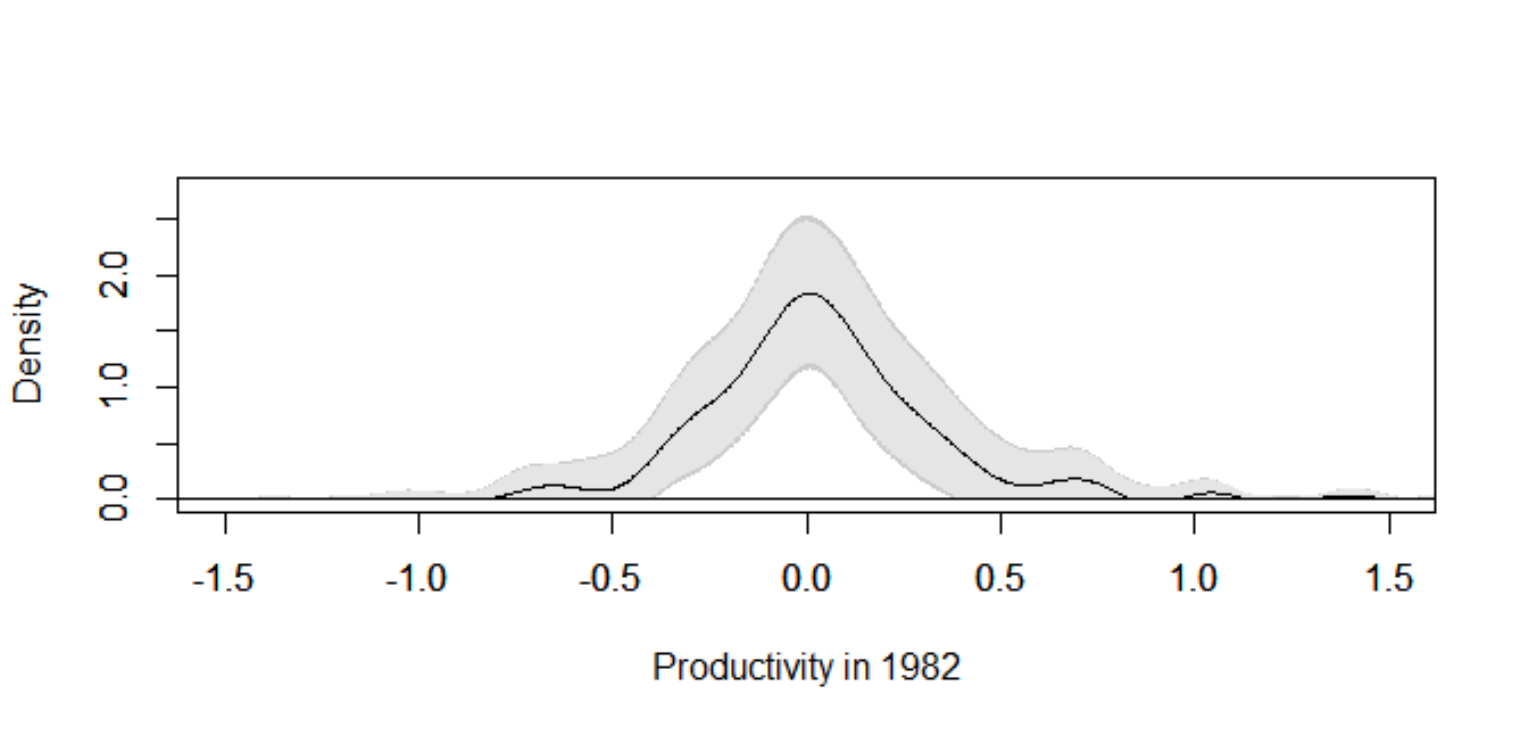}
	\caption{90\% and 95\% confidence bands for $f_{\omega_{j,t}}$ marked in light gray and dark gray, respectively. The black curve in the middle indicates the deconvolution kernel density estimate.}
	\label{fig:productivity1}
\end{figure}

\begin{figure}
	\centering
		\includegraphics[width=0.67\textwidth]{fig_productivity1982.pdf}\\
		\includegraphics[width=0.67\textwidth]{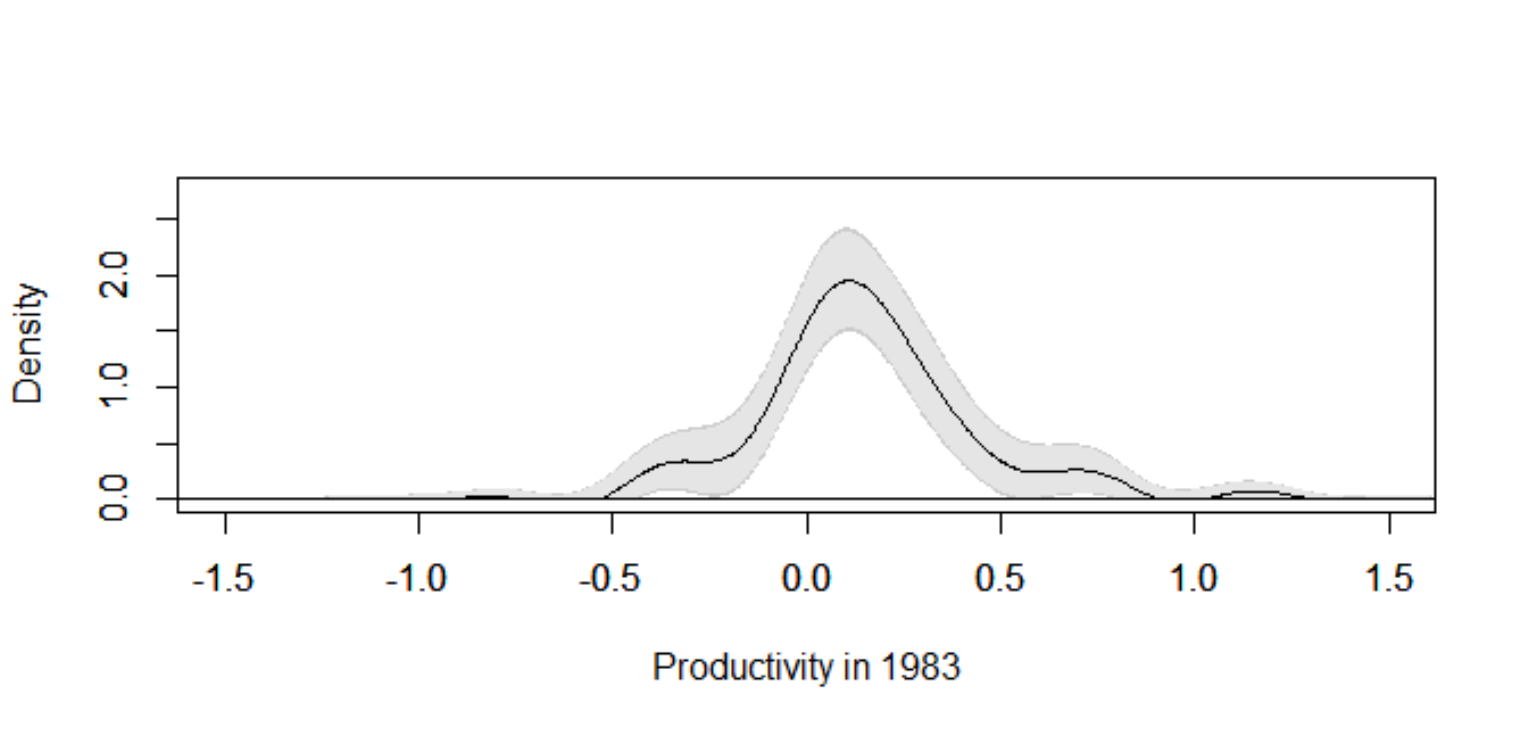}\\
		\includegraphics[width=0.67\textwidth]{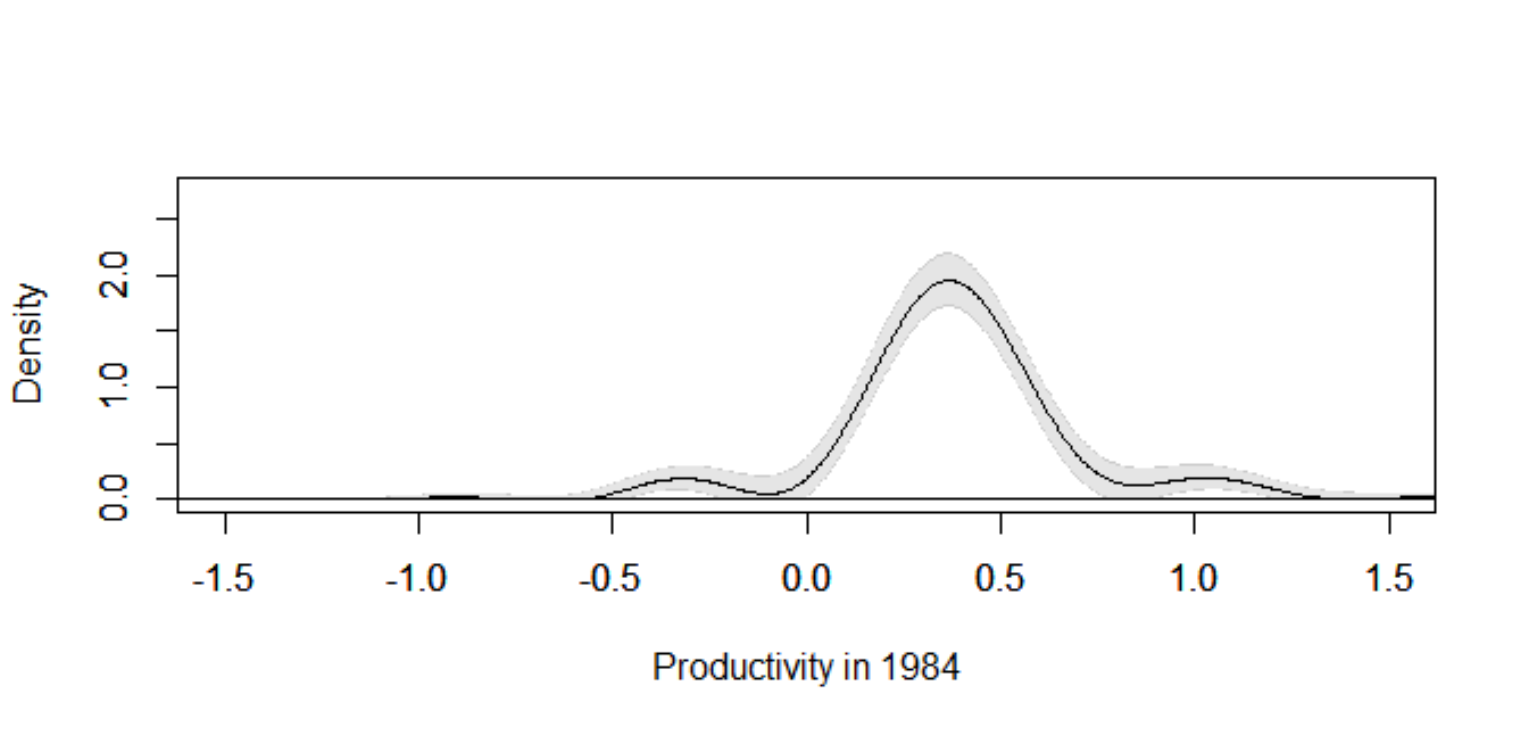}\\
		\includegraphics[width=0.67\textwidth]{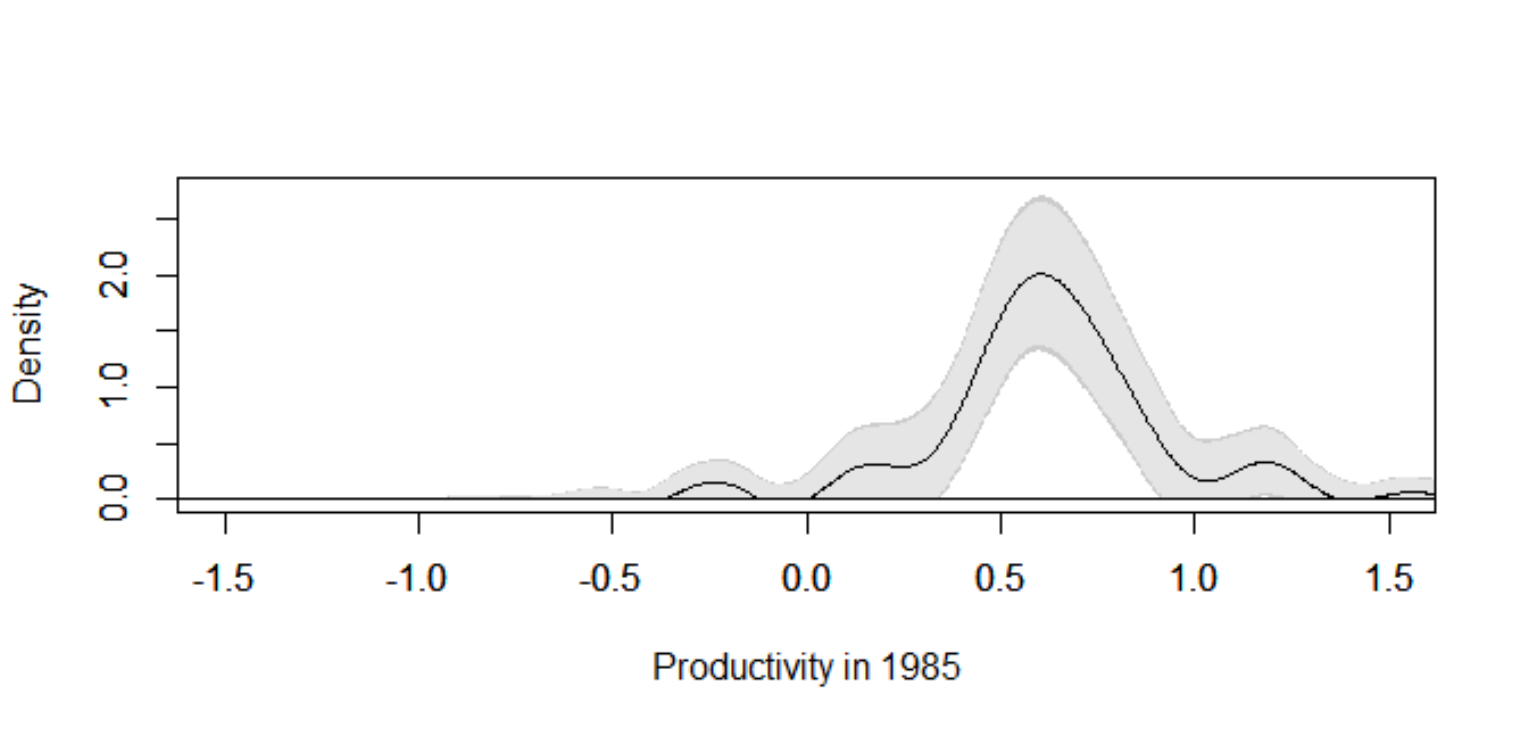}
	\caption{90\% and 95\% confidence bands for $f_{\omega_{j,t}}$ marked in light gray and dark gray, respectively. The black curve in the middle indicates the deconvolution kernel density estimate.}
	\label{fig:productivity2}
\end{figure}

\section{Extensions to super-smooth case}
\label{sec: supersmooth}

In this section, we consider extensions of the results on confidence bands to the case where the error density $f_{\varepsilon}$ is super-smooth.
While some notations were changed in Section \ref{sec: panel} to accommodate panel data models, we switch back in this section to the original notations used prior to Section \ref{sec: panel}.
We still keep Assumptions \ref{as: standard} and \ref{as: moment condition}, but require a different set of assumptions on the kernel function $K$, the bandwidth $h_{n}$, and the sample size $m$ for $f_{\varepsilon}$. It turns out that from a technical reason, we require $m/n \to \infty$ in the super-smooth case, and so Example \ref{ex: canonical} is formally not covered in the super-smooth case. Still, we believe that the extensions to the super-smooth case are of some interest. 

We modify Assumptions \ref{as: kernel}, \ref{as: ill-posedness}, and \ref{as: rate} as follows. First, for the kernel  and error characteristic functions, we assume the following conditions. 

\begin{assumption}
\label{as: kernel supersmooth}
Let $K: \R \to \R$ be a kernel function  such that its Fourier transform $\varphi_{K}$ is even (i.e., $\varphi_{K}(-t)=\varphi_{K}(t)$) and has support $[-1,1]$. Furthermore, there exist  constants  $C_{2} > 0$ and $\lambda \geq 0$ such that $\varphi_{K}(1-t) = C_{2} t^{\lambda} + o(t^{\lambda})$ as $t \downarrow 0$.
\end{assumption}

\begin{assumption}
\label{as: supersmooth}
The error characteristic function $\varphi_{\varepsilon}$ does not vanish on $\R$, and there exist constants $C_{3} > 0, \gamma > 1, \gamma_{0} \in \R, \nu > 0$ such that 
$\varphi_{\varepsilon}(t) = C_{3} (1+o(1) )|t|^{\gamma_{0}}e^{-\nu |t|^{\gamma}}$ as $|t| \to \infty$.
\end{assumption}

These assumptions are adapted from \cite{EsUh05}. Assumption \ref{as: supersmooth} covers cases where the error characteristic function decays exponentially fast as $|t| \to \infty$, thereby covering cases where the error density is super-smooth. 
However, Assumption \ref{as: supersmooth} is more restrictive than  standard super-smoothness conditions; e.g., it excludes the Cauchy error. This assumption is needed to derive a lower bound on $\sigma_{n}^{2}(x)$; see the following discussion. 

Assumptions \ref{as: kernel supersmooth} and \ref{as: supersmooth}, together with the assumption that  $\Ep[Y^{2}] < \infty$, ensure that the variance function $\sigma_{n}^{2}(x) = \Var (K_{n}((x-Y)/h_{n}))$ is expanded as
\begin{equation}
\sigma_{n}^{2}(x) =(1+o(1)) \frac{C_{2}^{2}}{2C_{3}^{2}\pi^{2}} (\nu \gamma)^{-2\lambda-2} (\Gamma (\lambda+1))^{2}  h_{n}^{2\gamma (1+\lambda) + 2\gamma_{0}} e^{2\nu h_{n}^{-\gamma}} \label{eq: variance supersmooth} 
\end{equation}
as $n \to \infty$; see the proof of Theorem 1.5 in \cite{EsUh05}. It is not difficult to verify from their proof that  $o(1)$ in (\ref{eq: variance supersmooth}) is uniform in $x \in I$ for any compact interval $I \subset \R$. 
It is worthwhile to point out that, in contrast to the ordinary smooth case, the lower bound on $\sigma_{n}^{2}(x)$ in (\ref{eq: variance supersmooth}) does not explicitly depend on $x$ nor $f_{Y}$. 
Further, Assumption \ref{as: supersmooth} implies that 
\begin{equation}
\inf_{|t| \leq h_{n}^{-1}} | \varphi_{\varepsilon}(t) | \geq C_{3}  (1-o(1)) h_{n}^{-\gamma_{0}} e^{-\nu h_{n}^{-\gamma}} \label{eq: lower bound on chf supersmooth}
\end{equation}
as $n \to \infty$.  It turns out that (\ref{eq: variance supersmooth}) and (\ref{eq: lower bound on chf supersmooth}) are the only differences  to take care of when proving the analogues of Theorems \ref{thm: Gaussian approximation} and \ref{thm: multiplier bootstrap} in the super-smooth case. Finally, we modify Assumption \ref{as: rate} as follows.

\begin{assumption}
\label{as: rate supersmooth}
(a)  $\frac{(\log h_{n}^{-1})^{2}}{nh_{n}^{4\gamma (1+\lambda)}} \to 0$. (b) $\frac{n \log h_{n}^{-1}}{mh_{n}^{2\gamma (1+\lambda) -2}}  \bigvee \frac{e^{2\nu h{n}^{-\gamma}} (\log h_{n}^{-1})^{2}}{mh_{n}^{4\gamma (1+\lambda)-2\gamma_{0}}} \to 0$.
\end{assumption}

The requirement that $\frac{n \log h_{n}^{-1}}{mh_{n}^{2\gamma (1+\lambda) -2}} \to 0$ implies that we at least need $m/n \to \infty$. This condition is used to ensure that the effect of estimating $\varphi_{\varepsilon}$ is negligible. To be precise, 
in our proof, a bound on $\| \hat{f}_{X} - \hat{f}_{X}^{*} \|_{\R}$ involves a term of order $m^{-1/2}h_{n}^{\gamma_{0}} e^{\nu h_{n}^{-\gamma}}$, which has to be of smaller order than $n^{-1/2}h_{n}^{\gamma (1+\lambda) + \gamma_{0}-1} e^{\nu h_{n}^{-\gamma}}(\log h_{n}^{-1})^{-1/2}$. 
Technically, this problem happens because the ratio of $1/\inf_{|t| \leq h_{n}^{-1}} | \varphi_{\varepsilon}(t)|$ over $\inf_{x \in I} \sigma_{n}(x)$ is larger in the super-smooth case than that in the ordinary smooth case; the ratio is $O(h_{n}^{-\gamma (1+\lambda)})$ in the super-smooth case, while it is $O(h_{n}^{-1/2})$ in the ordinary smooth case. It is not known at the current moment whether we could relax this condition on $m$ in the super-smooth case.

In any case, these assumptions guarantee that the conclusions of Theorems \ref{thm: Gaussian approximation} and \ref{thm: multiplier bootstrap}, except for the result on the width of the band,  hold true in the super-smooth case. 

\begin{theorem}
\label{thm: supersmooth}
Suppose that Assumptions \ref{as: standard}, \ref{as: moment condition}, and \ref{as: kernel supersmooth}--\ref{as: rate supersmooth} are satisfied. Let $I \subset \R$ be any compact interval, and suppose in addition that $\Ep[Y^{2}] < \infty$. 
Then the conclusions of Theorems \ref{thm: Gaussian approximation} and \ref{thm: multiplier bootstrap} , except for the result on the width of the band, hold true. 
\end{theorem}

\begin{remark}[Comparisons with \cite{EsGu08}]
\label{rem: van Es}
\cite{EsGu08} prove that, under the assumptions that $f_{\varepsilon}$ is known and satisfies Assumption \ref{as: supersmooth} (of the present paper) with $\gamma=2$, 
\[
\frac{\sqrt{n}}{h_{n}^{2 (1+\lambda) + \gamma_{0}-1}e^{\nu h_{n}^{-2}}} \| \hat{f}_{X}^{*} (\cdot) - \Ep[\hat{f}_{X}^{*}(\cdot)] \|_{[0,1]} \stackrel{d}{\to} \frac{C_{2}}{\sqrt{2}C_{3}\pi} (2\nu)^{-\lambda-1} \Gamma (\lambda+1) V,
\]
where $V$ follows the Rayleigh distribution, i.e., $V$ is a random variable having density $f_{V}(v) = v e^{-v^{2}/2}1_{[0,\infty)}(v)$ \citep[see][for the precise regularity conditions]{EsGu08}. Interestingly, the limit distribution differs from Gumbel distributions. 

Despite this non-standard feature, Theorem \ref{thm: supersmooth} shows that the multiplier bootstrap ``works'', i.e., the conditional distribution of $\| \hat{Z}_{n}^{\xi} \|_{I}$ can consistently estimate the distribution of $\| \hat{Z}_{n} \|_{I}$ in the sense that $\sup_{z \in \R} | \Pr \{ \| \hat{Z}_{n}^{\xi} \|_{I} \leq z \mid \mathcal{D}_{n} \} - \Pr \{ \| \hat{Z}_{n} \|_{I} \leq z \} | \stackrel{\Pr}{\to} 0$. Further, Theorem \ref{thm: supersmooth} extends the admissible range of $\gamma$ compared with the result of \cite{EsGu08}.  
\end{remark}

If $f_{X}$ belongs to a H\"{o}lder ball $\Sigma (\beta,B)$, then we have the following corollary. 

\begin{corollary}
\label{cor: supersmooth}
Assume all the conditions of Theorem \ref{thm: supersmooth}. Further, suppose that $f_{X} \in \Sigma (\beta,B)$ for some $\beta > 0, B > 0$, and that Condition (\ref{eq: higher order kernel}) is satisfied for the kernel function $K$ where $k$ is the integer such that $\beta < k \leq \beta+1$. Consider the multiplier bootstrap confidence band $\hat{\mathcal{C}}_{n}$ defined in (\ref{eq: multiplier bootstrap confidence band}). Then as $n \to \infty$, $\Pr \{ f_{X}(x) \in \hat{\mathcal{C}}_{n}(x) \ \forall x \in I \} \to 1-\tau$,
provided that $
h_{n}^{\beta+1-\gamma (1+\lambda) - \gamma_{0}} e^{-\nu h_{n}^{-\gamma}}\sqrt{n \log h_{n}^{-1}} \to 0$.
\end{corollary}

\section{Conclusion}
\label{sec: conclusion}
The previous literature on inference in deconvolution has focused on the case where the error distribution is known. 
In econometric applications, the assumption that the error distribution is known is unrealistic, and the present paper fills this important void. 
Specifically, we develop a method to construct uniform confidence bands in deconvolution when the error distribution is unknown and needs to be estimated with an auxiliary sample from the error distribution.
The auxiliary sample may directly come from validation data, such as administrative data, or can be constructed from panel data with a symmetric error distribution.

We first focus on the baseline setting where the error density is ordinary smooth. 
The construction is based upon the ``intermediate'' Gaussian approximation and the Gaussian multiplier bootstrap, instead of explicit limit distributions such as Gumbel distributions.
This approach allows us to prove the validity of the proposed multiplier bootstrap confidence band under mild regularity conditions. 
Simulation studies demonstrate that the multiplier bootstrap confidence bands perform well in the finite sample. 
We apply our method to the Outer Continental Shelf (OCS) Auction Data and draw confidence bands for the density of common values of mineral rights on oil and gas tracts.
We also discuss an application of our main result to additive fixed-effect panel data models.
As an empirical illustration of the panel analysis, we draw confidence bands for the density of the total factor productivity in the food manufacturing industry in Chile following the analysis of \cite{LePe03}.
Finally, we present extensions of the baseline theoretical results to the case of super-smooth error densities.

Throughout this paper, we suppose the availability of an auxiliary sample, an example of which is panel data with a symmetric error distribution similarly to \cite{HoMa96}.
The estimator of \cite{LiVu98}, on the other hand, relaxes the assumption of a symmetric error distribution.
An extension of our results on the method of inference to this case is left for future research.

\appendix 

\section{Proofs}

In what follows, the notation $\lesssim$ signifies that the left hand side is bounded by the right hand side up to some constant independent of $n$ and $x$.

\subsection{Proof of Theorem \ref{thm: Gaussian approximation}}
\label{sec: proofs}


We first state the following lemmas which will be used in the proof of Theorem \ref{thm: Gaussian approximation}. 
For a class of measurable functions $\mF$ on a measurable space $(S,\mathcal{S})$ and a probability measure $Q$ on $\mathcal{S}$, let $N(\mF,\| \cdot \|_{Q,2},\delta)$ denote the $\delta$-covering number for $\mF$ with respect to the $L^{2}(Q)$-seminorm  $\| \cdot \|_{Q,2}$; see Section 2.1 in \cite{vaWe96} for details.

\begin{lemma}
\label{lem: VC type}
Let $K$ be a kernel function on $\R$ such that $\varphi_{K}$ is  supported in $[-1,1]$, and suppose that  $\varphi_{\varepsilon}$ does not vanish on $\R$. Let $r_{n}=1/\inf_{|t| \leq h_{n}^{-1}} | \varphi_{\varepsilon}(t)|$.
Consider the class of functions $\mathcal{K}_{n} = \{ y \mapsto K_{n}((x-y)/h_{n}) : x \in \R \}$, where $K_{n}$ denotes the corresponding deconvolution kernel. Then there exist constants $A,v > 0$ independent of $n$ such that for all $n \geq 1$,
\[
\sup_{Q} N(\mathcal{K}_{n},\| \cdot \|_{Q,2},  r_{n}\delta ) \leq (A/\delta)^{v}, \ 0 < \forall \delta \leq 1,
\]
where $\sup_{Q}$ is taken over all Borel probability measures $Q$ on $\R$. 
\end{lemma}

In view of Lemma 1 in \cite{GiNi09} (or  Proposition 3.6.12 in \cite{GiNi16}), Lemma \ref{lem: VC type} follows as soon as we show that $K_{n}$ has \textit{quadratic variation} $\lesssim r_{n}^{2}$. 
Recall that a real-valued function $f$ on $\R$ is said to be of bounded $p$-variation for $p \in [1,\infty)$ if 
\[
V_{p}(f) := \sup \Bigg \{ \sum_{\ell=1}^{N} |f(x_{\ell}) - f(x_{\ell-1})|^{p} : -\infty < x_{0} < \cdots < x_{N} < \infty, N =1,2,\dots \Bigg \}
\]
is finite. A function of bounded $2$-variation is said to be of bounded quadratic variation.
Now, Lemma \ref{lem: VC type} follows from the next lemma.

\begin{lemma}
\label{lem: quadratic variation}
Assume the same conditions as in Lemma \ref{lem: VC type}. 
Then the deconvolution kernel $K_{n}$ is of bounded quadratic variation with $V_{2}(K_{n}) \lesssim r_{n}^{2}$. 
\end{lemma}

\begin{proof}
The basic idea of the proof is due to the proof of Lemma 1 in \cite{LoNi11}. In view of the continuous embedding of the homogeneous Besov space $\dot{B}_{2,1}^{1/2}(\R)$ into $BV_{2}(\R)$, the space of functions of bounded quadratic variation \citep[][Theorem 5]{BoLaSi06}, it is enough to show that $\| K_{n} \|_{\dot{B}_{2,1}^{1/2}} \lesssim r_{n}$, where 
\[
\| f \|_{\dot{B}_{2,1}^{1/2}} = \int_{\R} \frac{1}{|u|^{3/2}}  \left ( \int_{\R} | f(x+u) - f(x)|^{2} dx \right )^{1/2} du. 
\] 
Precisely speaking, for any real-valued function $f$ on $\R$ that vanishes at infinity (i.e., $\lim_{|x| \to \infty} f(x) = 0$), the following bound holds: $V_{2}(f)^{1/2} \lesssim \| f \|_{B_{2,1}^{1/2}}$ up to a constant independent of $f$. Observe that $K_{n}$ vanishes at infinity by the Riemann-Lebesgue lemma.
Let $\psi_{n} (t) = \varphi_{K}(t)/\varphi_{\varepsilon}(t/h_{n})$, and observe that, using Plancherel's theorem, 
\[
\int_{\R} | K_{n} (x+u) - K_{n}(x) |^{2} dx = \frac{1}{2\pi} \int_{\R} | e^{-itu} - 1 |^{2} | \psi_{n}(t) |^{2}dt = \frac{1}{\pi} \int_{\R} (1-\cos (tu)) | \psi_{n}(t) |^{2} dt.
\]
Using the inequality $1-\cos (tu) \leq \min \{ 2, (tu)^{2}/2 \}$, 
we conclude that 
\begin{align*}
\| K_{n} \|_{\dot{B}_{2,1}^{1/2}} &\lesssim \left ( \int_{\R} t^{2} | \psi_{n}(t) |^{2} dt \right )^{1/2} \int_{[-1,1]} \frac{1}{|u|^{1/2}} du  +  \left ( \int_{\R}| \psi_{n}(t) |^{2} dt \right )^{1/2} \int_{[-1,1]^{c}} \frac{1}{|u|^{3/2}} du \\
&\lesssim r_{n}.
\end{align*}
This completes the proof. 
\end{proof}

\begin{remark}
Lemma \ref{lem: VC type} generalizes a part of Lemma 5.3.5 in \cite{GiNi16} that focuses on the case where the bandwidth $h_{n}$ takes values in $\{ 2^{-k} : k=1,2,\dots \}$. 
The proof of Lemma \ref{lem: VC type} appears to be simpler than that of Lemma 5.3.5 in \cite{GiNi16} (but note that Lemma 5.3.5 in \cite{GiNi16} also covers wavelet kernels). 
\end{remark}

\begin{lemma}
\label{lem: variance}
Assumptions \ref{as: standard}--\ref{as: variance} imply that, for sufficiently large $n$,  $\inf_{x \in I} \sigma_{n}^{2} (x) \gtrsim h_{n}^{-2\alpha+1}$.
\end{lemma}

\begin{proof}
The proof is inspired by \cite{Fa91b}. The difficulty here is that $K_{n}$ has unbounded support (since its Fourier transform $\frac{\varphi_{K}}{\varphi_{\varepsilon}(\cdot/h_{n})}$ is compactly supported) and depends intrinsically on $n$. Observe first that $\| K_{n} \|_{\R} \lesssim h_{n}^{-\alpha}$. Second, integration by parts yields that 
\[
K_{n}(x) = \frac{1}{2\pi ix} \int_{\R} e^{-itx} \left \{ \frac{\varphi_{K}(t)}{\varphi_{\varepsilon}(t/h_{n})} \right \}' dt
= \frac{1}{2\pi ix} \int_{\R} e^{-itx} \left \{ \frac{\varphi_{K}'(t)}{\varphi_{\varepsilon}(t/h_{n})} -\frac{\varphi_{K}(t) \varphi_{\varepsilon}'(t/h_{n})}{h_{n}\varphi^{2}_{\varepsilon}(t/h_{n})}\right \} dt. 
\]
It is not difficult to verify that 
\[
\int_{\R} \left |  \frac{\varphi_{K}'(t)}{\varphi_{\varepsilon}(t/h_{n})} \right | dt \lesssim h_{n}^{-\alpha}.
\]
Splitting the integral into $|t/h_{n}| \leq 1$ and $|t/h_{n}| > 1$, we also see that 
\begin{align*}
\left \{ \int_{|t/h_{n}| \leq 1} + \int_{|t/h_{n}| > 1} \right \} \left | \frac{\varphi_{K}(t)\varphi_{\varepsilon}'(t/h_{n})}{h_{n}\varphi^{2}_{\varepsilon}(t/h_{n})} \right | dt 
\lesssim 1 + h_{n}^{-\alpha} \int_{\R} |t|^{\alpha-1} |\varphi_{K}(t)| dt,
\end{align*}
which is $\lesssim h_{n}^{-\alpha}$. 
This yields that $h_{n}^{2\alpha} K_{n}^{2}(x) \lesssim 1/x^{2}$, and so $h_{n}^{2\alpha} K_{n}^{2}(x) \lesssim \min \{ 1,1/x^{2} \}$.

Now, observe that
\begin{align*}
| \Ep[ K_{n} ((x-Y)/h_{n}) ] | &= h_{n} \left | \int_{\R} K(y) f_{X}(x-h_{n}y) dy \right |  \\
&\leq h_{n} \| f_{X} \|_{\R} \int_{\R} |K(y)| dy =O(h_{n}), \ \text{and} \\
\Ep[ K_{n}^{2}((x-Y)/h_{n}) ] &= h_{n} \int_{\R} K_{n}^{2}(y) f_{Y}(x-h_{n}y) dy.
\end{align*}
So, it is enough to prove that
\[
\inf_{x \in I} \int_{\R} K_{n}^{2}(y) f_{Y}(x-h_{n}y) dy \gtrsim (1-o(1)) h_{n}^{-2\alpha}.
\]
To this end, since 
\[
\int_{\R} K_{n}^{2}(y) dy = \frac{1}{2 \pi} \int_{\R} \frac{| \varphi_{K}(t) |^{2}}{| \varphi_{\varepsilon}(t/h_{n}) |^{2}} dt \gtrsim h_{n}^{-2\alpha},
\]
by Plancherel's theorem (we have used that $|\varphi_{\varepsilon}(t)| \leq C_{1}|t|^{-\alpha}$ for $|t| \geq 1$ to deduce the last inequality), it is enough  to prove that as $n \to \infty$,
\[
\sup_{x \in I} \left | h_{n}^{2\alpha} \int_{\R} K_{n}^{2}(y) \{ f_{Y}(x-h_{n}y) - f_{Y}(x) \} dy \right | \to 0.
\]
Since $f_{Y}$ is continuous and $I$ is compact, for any $\rho > 0$, there exists $\delta > 0$ such that $\sup_{x \in I} |f_{Y}(x-y) - f_{Y}(x)| \leq \rho$ whenever $|y| \leq \delta$. So,
\begin{align*}
&\sup_{x \in I} h_{n}^{2\alpha} \int_{\R} K_{n}^{2}(y) | f_{Y}(x-h_{n}y) - f_{Y}(x) | dy \\
&\quad \lesssim \rho \int_{|y| \leq \delta/h_{n}} \min \{ 1, 1/y^{2} \} dy +2\| f_{Y} \|_{\R} \int_{|y| > \delta/h_{n}} y^{-2} dy \lesssim \rho + o(1),
\end{align*}
which yields the desired conclusion. 
\end{proof}

\begin{proof}[Proof of Theorem \ref{thm: Gaussian approximation}]
We divide the proof into three steps. 

\textbf{Step 1.} (Gaussian approximation to $Z_{n}^{*}$). Recall the empirical process $Z_{n}^{*}(x), x \in I$ defined as
\begin{align*}
Z_{n}^{*}(x) &= \frac{\sqrt{n} h_{n} \{ \hat{f}_{X}^{*}(x) - \Ep [\hat{f}_{X}^{*}(x)] \}}{\sigma_{n}(x)} \\
&=\frac{1}{\sigma_{n}(x)\sqrt{n}} \sum_{j=1}^{n} \{ K_{n}((x-Y_{j})/h_{n}) - \Ep[K_{n}((x-Y)/h_{n})] \}, \ x \in I. 
\end{align*}
Consider the class of functions 
\[
\mG_{n} = \left \{ \frac{1}{\sigma_{n}(x)} \{ K_{n}((x-\cdot)/h_{n}) - \E[K_{n}((x-Y)/h_{n})] \}   : x \in I \right \},
\]
together with the empirical process indexed by $\mathcal{G}_{n}$ defined as
\[
\mG_{n} \ni g \mapsto \nu_{n}(g) = \frac{1}{\sqrt{n}} \sum_{j=1}^{n} g(Y_{j}).
\]
Observe that $\| Z_{n}^{*} \|_{I} =\| \nu_{n} \|_{\mG_{n}}$. We apply Corollary 2.2 in \cite{ChChKa14a} to $\| \nu_{n} \|_{\mG_{n}}$. First, since the set $\{ 1/\sigma_{n}(x) : x \in I \}$ is bounded with $\| 1/\sigma_{n} \|_{I} \lesssim h_{n}^{\alpha-1/2}$, 
in view of Lemma \ref{lem: VC type} of the present paper and Corollary A.1 in \cite{ChChKa14a}, there exist constants $A',v' > 0$ independent of $n$ such that 
\begin{equation}
\sup_{Q} N(\mG_{n}, \| \cdot \|_{Q,2}, \delta/\sqrt{h_{n}}) \leq (A'/\delta)^{v'}, \ 0 < \forall \delta \leq 1, \label{eq: covering number bound}
\end{equation}
which ensures the existence a tight Gaussian random variable $G_{n}$ in $\ell^{\infty}(\mG_{n})$ with mean zero and the same covariance function as $\nu_{n}$ \citep[cf.][Lemma 2.1]{ChChKa14a}. 
Since $\| K_{n}((x-\cdot)/h_{n})/\sigma_{n}(x) \|_{\R} \lesssim 1/\sqrt{h_{n}}$ and $\Var (K_{n}((x-Y)/h_{n})/\sigma_{n}(x)) = 1$, application of Corollary 2.2 in \cite{ChChKa14a} with $q=\infty, b \lesssim 1/\sqrt{h_{n}}, \sigma=1, A \lesssim 1, v \lesssim 1$, and $\gamma = 1/\log n$, yields that there exists a sequence  of random variables $W_{n}$ with $W_{n} \stackrel{d}{=} \| G_{n} \|_{\mG_{n}}$ (where the notation $\stackrel{d}{=}$ signifies equality in distribution) and such that 
\begin{equation}
\left | \| \nu_{n} \|_{\mG_{n}} - W_{n} \right | = O_{\Pr}\{  (\log n) /(nh_{n})^{1/6} \}, \label{eq: coupling}
\end{equation}
where the left hand side is equal to $| \| Z_{n}^{*} \|_{I} - W_{n} |$. 

Next, for $f_{n,x}(\cdot) = \{ K_{n}((x-\cdot)/h_{n})-\E[ K_{n}((x-Y)/h_{n})] \}/\sigma_{n}(x)$,  define
\[
Z_{n}^{G} (x) = G_{n} \left (f_{n,x}  \right ), \ x \in I,
\]
and observe that $Z_{n}^{G}$ is a tight Gaussian random variable in $\ell^{\infty}(I)$ with mean zero and the same covariance function as $Z_{n}^{*}$, and such that $\| Z_{n}^{G} \|_{I} = \| G_{n} \|_{\mG_{n}} \stackrel{d}{=} W_{n}$.
It is worth noting  that deducing from (\ref{eq: coupling}) a bound on 
\[
\sup_{z \in \R}\left | \Pr \{ \| Z^{*}_{n} \|_{I} \leq z \} - \Pr \{ \| Z_{n}^{G} \|_{I} \leq z \} \right |
\]
is  a non-trivial step, since the distribution of the approximating Gaussian  process $Z_{n}^{G}$ changes with $n$.  
To this end, we will use the anti-concentration inequality for the supremum of a Gaussian process, which yields that 
\begin{equation}
\sup_{z \in \R; \Delta > 0} \frac{\Pr \{ z - \Delta \leq \| Z_{n}^{G} \|_{I} \leq z + \Delta \}}{\Delta}  \lesssim \Ep [ \| Z_{n}^{G} \|_{I} ].   \label{eq: AC bound}
\end{equation}
See Corollary 2.1 in \cite{ChChKa14b} (see also Theorem 3 in \cite{ChChKa15}). 
To apply this inequality, we shall bound $\Ep [ \| Z_{n}^{G} \|_{I} ] = \Ep [ \| G_{n} \|_{\mG_{n}} ]$, but given the covering number bound (\ref{eq: covering number bound}) and $\Var (K_{n}((x-Y)/h_{n})/\sigma_{n}(x)) = 1$, Dudley's entropy integral bound \citep[cf.][Corollary 2.2.8]{vaWe96} yields that 
\[
\E[ \| G_{n} \|_{\mG_{n}} ] \lesssim \int_{0}^{1} \sqrt{1+\log (1/(\delta\sqrt{h}_{n}))} d\delta \lesssim \sqrt{\log h_{n}^{-1}}. 
\]

Now, combining (\ref{eq: coupling}) with the anti-concentration inequality (\ref{eq: AC bound}), we conclude that 
\[
\sup_{z \in \R} \left | \Pr \{ \| Z^{*}_{n} \|_{I} \leq z \} - \Pr \{ \| Z_{n}^{G} \|_{I} \leq z \} \right | \to 0,
\]
provided that $\frac{ (\log n)\sqrt{\log h_{n}^{-1}}}{(nh_{n})^{1/6}} \to 0$,
which is satisfied under our assumption.

\textbf{Step 2}. (Gaussian approximation to the intermediate process). Define the intermediate process 
\[
\tilde{Z}_{n}(x) =  \frac{\sqrt{n} h_{n} \{ \hat{f}_{X}(x) - \Ep[ \hat{f}_{X}^{*}(x) ] \}}{\sigma_{n}(x)}, \ x \in I, 
\]
where the difference from $\hat{Z}_{n}(x)$ is that $\hat{\sigma}_{n}(x)$ is replaced by $\sigma_{n}(x)$. In this step, we wish to prove that 
\begin{equation}
\sup_{z \in \R} | \Pr \{ \|\tilde{Z}_{n} \|_{I} \leq z \} - \Pr \{ \| Z_{n}^{G} \|_{I} \leq z \} | \to 0,  \label{eq: intermediate}
\end{equation}
where $Z_{n}^{G}$ is given in the previous step.

Since $\varphi_{\varepsilon}$ does not vanish on $\R$, we have that $\{ \varphi_{Y} \neq 0 \} = \{ \varphi_{X} \neq 0 \}$, so that
we have
\begin{align*}
\hat{f}_{X}^{*} (x) 
&= \frac{1}{2\pi} \left \{ \int_{\{ \varphi_{X} \neq 0 \}} e^{-itx} \varphi_{K}(th_{n}) \frac{\hat{\varphi}_{Y}(t)}{\varphi_{Y}(t)} \varphi_{X}(t) dt 
+ \int_{\{ \varphi_{X} = 0 \}} e^{-itx}  \frac{\varphi_{K}(th_{n})}{\varphi_{\varepsilon}(t)} \hat{\varphi}_{Y}(t) dt \right \}, \ \text{and} \\
\hat{f}_{X} (x) &=  
 \frac{1}{2\pi}  \Bigg \{ \int_{\{ \varphi_{X} \neq 0 \}} e^{-itx} \varphi_{K}(th_{n}) \frac{\hat{\varphi}_{Y}(t)}{\varphi_{Y}(t)} \frac{\varphi_{\varepsilon}(t)}{\hat{\varphi}_{\varepsilon}(t)}\varphi_{X}(t) dt + \int_{\{ \varphi_{X}= 0 \}} e^{-itx} \frac{\varphi_{K}(th_{n})}{\varphi_{\varepsilon}(t)} \hat{\varphi}_{Y}(t) \frac{\varphi_{\varepsilon}(t)}{\hat{\varphi}_{\varepsilon}(t)} dt \Bigg \}.
\end{align*}
So, letting
\[
\tilde{f}_{X}(x) = \frac{1}{2\pi} \int_{\R} e^{-itx} \varphi_{K}(th_{n}) \frac{\varphi_{\varepsilon}(t)}{\hat{\varphi}_{\varepsilon}(t)}\varphi_{X}(t) dt,
\]
we obtain the following decomposition:
\begin{align*}
&\hat{f}_{X}(x) - \hat{f}^{*}_{X} (x) \\
&= \left [ \{ \hat{f}_{X}(x) - \tilde{f}_{X}(x) \}  - \{ \hat{f}_{X}^{*}(x) - \Ep[\hat{f}_{X}^{*}(x)] \} \right ] + \{ \tilde{f}_{X}(x)  - \Ep[\hat{f}_{X}^{*}(x)] \} \\
&= \frac{1}{2\pi} \int_{\{ \varphi_{X} \neq 0 \}} e^{-itx} \varphi_{K}(th_{n}) \left \{ \frac{\hat{\varphi}_{Y}(t)}{\varphi_{Y}(t)} - 1 \right \} \left \{ \frac{\varphi_{\varepsilon}(t)}{\hat{\varphi}_{\varepsilon}(t)} - 1 \right \} \varphi_{X}(t) dt \\
&\quad + \frac{1}{2\pi} \int_{\{ \varphi_{X} = 0 \}} e^{-itx} \frac{\varphi_{K}(th_{n})}{\varphi_{\varepsilon}(t)} \hat{\varphi}_{Y}(t) \left \{ \frac{\varphi_{\varepsilon}(t)}{\hat{\varphi}_{\varepsilon}(t)} - 1 \right \} dt\\
&\quad + \frac{1}{2\pi}  \int_{\R} e^{-itx} \varphi_{K}(th_{n}) \left \{ \frac{\varphi_{\varepsilon}(t)}{\hat{\varphi}_{\varepsilon}(t)}-1 \right \} \varphi_{X}(t)  dt.
\end{align*}
Hence the Cauchy-Schwarz inequality yields that 
\begin{align*}
| \hat{f}_{X}(x) - \hat{f}_{X}^{*}(x) |^{2} 
&\lesssim \left \{ \int_{\{ \varphi_{X} \neq 0 \} \cap [-h_{n}^{-1},h_{n}^{-1}]} \left | \frac{\hat{\varphi}_{Y}(t)}{\varphi_{Y}(t)} - 1 \right |^{2} | \varphi_{X}(t) |^{2} dt \right \} \left \{ \int_{-h_{n}^{-1}}^{h_{n}^{-1}} \left | \frac{\varphi_{\varepsilon}(t)}{\hat{\varphi}_{\varepsilon}(t)} - 1 \right |^{2} dt \right \} \\
&\quad + h_{n}^{-2\alpha} \left \{ \int_{\{ \varphi_{X}=0 \} \cap [-h_{n}^{-1},h_{n}^{-1}]} | \hat{\varphi}_{Y}(t) |^{2} dt \right \} \left \{ \int_{-h_{n}^{-1}}^{h_{n}^{-1}} \left | \frac{\varphi_{\varepsilon}(t)}{\hat{\varphi}_{\varepsilon}(t)} - 1 \right |^{2} dt \right \} \\
&\quad + \int_{-h_{n}^{-1}}^{h_{n}^{-1}} \left | \frac{\varphi_{\varepsilon}(t)}{\hat{\varphi}_{\varepsilon}(t)}-1 \right |^{2} | \varphi_{X}(t) |  dt.
\end{align*}
We will bound the following four terms: 
\begin{align*}
&\int_{\{ \varphi_{X} \neq 0 \} \cap [-h_{n}^{-1},h_{n}^{-1}]} \left | \frac{\hat{\varphi}_{Y}(t)}{\varphi_{Y}(t)} - 1 \right |^{2} | \varphi_{X}(t) |^{2} dt, \quad \int_{\{ \varphi_{X} = 0 \} \cap [-h_{n}^{-1},h_{n}^{-1}]} | \hat{\varphi}_{Y}(t) |^{2} dt, \\
&\int_{-h_{n}^{-1}}^{h_{n}^{-1}} \left | \frac{\varphi_{\varepsilon}(t)}{\hat{\varphi}_{\varepsilon}(t)} - 1 \right |^{2} dt , \quad
\int_{-h_{n}^{-1}}^{h_{n}^{-1}} \left | \frac{\varphi_{\varepsilon}(t)}{\hat{\varphi}_{\varepsilon}(t)}-1 \right |^{2} | \varphi_{X}(t) |  dt.
\end{align*}

To bound the first term, pick and fix any $t \in \R$ such that $\varphi_{Y}(t) \neq 0$; let $\zeta_{j} = e^{it Y_{j}}/\varphi_{Y}(t)$ for $j=1,\dots,n$.
Then we have that
\[
\Ep \left [ \left | \frac{\hat{\varphi}_{Y}(t)}{\varphi_{Y}(t)} - 1 \right |^{2} \right ] = \Ep \left [ \left | \frac{1}{n} \sum_{j=1}^{n} \{ \zeta_{j} - \Ep [ \zeta_{j} ] \} \right |^{2} \right ] \leq \frac{1}{n | \varphi_{Y}(t) |^{2}}. 
\]
Therefore, 
\[
\Ep \left [ \int_{\{ \varphi_{X} \neq 0 \} \cap [-h_{n}^{-1},h_{n}^{-1}]} \left | \frac{\hat{\varphi}_{Y}(t)}{\varphi_{Y}(t)} - 1 \right |^{2} | \varphi_{X}(t) |^{2} dt \right ] \leq n^{-1} \int_{-h_{n}^{-1}}^{h_{n}^{-1}} \frac{1}{|\varphi_{\varepsilon}(t)|^{2}} dt \lesssim h_{n}^{-2\alpha}(nh_{n})^{-1}.
\]
On the other hand, since $\varphi_{Y}(t)=0$ whenever $\varphi_{X}(t)=0$, we have that
\[
\Ep \left [\int_{\{ \varphi_{X} = 0 \} \cap [-h_{n}^{-1},h_{n}^{-1}]} | \hat{\varphi}_{Y}(t) |^{2} dt \right ] \leq 2(nh_{n})^{-1}. 
\]
To bound the third and fourth terms, we first note that, from Lemma \ref{lem: rate ecf} ahead together with the fact that $\Ep[ | \varepsilon |^{p}] < \infty$ for some $p > 0$, 
\begin{equation}
\| \hat{\varphi}_{\varepsilon} - \varphi_{\varepsilon} \|_{[-h_{n}^{-1},h_{n}^{-1}]} = O_{\Pr}\{ m^{-1/2} \log h_{n}^{-1} \}, \label{eq: rate ecf}
\end{equation}
which is $o_{\Pr}(h_{n}^{\alpha})$ by Assumption \ref{as: rate} (b). Hence
\[
\inf_{|t| \leq h_{n}^{-1}} | \hat{\varphi}_{\varepsilon}(t) | \geq \inf_{|t| \leq h_{n}^{-1}} | \varphi_{\varepsilon}(t) | - o_{\Pr}(h_{n}^{\alpha}) \gtrsim (1-o_{\Pr}(1)) h_{n}^{\alpha},
\]
from which we have
\begin{align*}
\int_{-h_{n}^{-1}}^{h_{n}^{-1}} \left | \frac{\varphi_{\varepsilon}(t)}{\hat{\varphi}_{\varepsilon}(t)} - 1 \right |^{2} dt  &\leq O_{\Pr}(h_{n}^{-2\alpha}) \int_{-h_{n}^{-1}}^{h_{n}^{-1}} | \varphi_{\varepsilon}(t) - \hat{\varphi}_{\varepsilon}(t) |^{2} dt \\
&=  O_{\Pr} \left \{ h_{n}^{-2\alpha} \int_{-h_{n}^{-1}}^{h_{n}^{-1}} \Ep[ | \varphi_{\varepsilon}(t) - \hat{\varphi}_{\varepsilon}(t) |^{2} ]dt \right \} \\
&=  O_{\Pr}(h_{n}^{-2\alpha-1}m^{-1}), \ \text{and} \\
\int_{-h_{n}^{-1}}^{h_{n}^{-1}} \left | \frac{\varphi_{\varepsilon}(t)}{\hat{\varphi}_{\varepsilon}(t)}-1 \right |^{2} | \varphi_{X}(t) |  dt &=  O_{\Pr} \left \{ h_{n}^{-2\alpha} \int_{-h_{n}^{-1}}^{h_{n}^{-1}} \Ep[ | \varphi_{\varepsilon}(t) - \hat{\varphi}_{\varepsilon}(t) |^{2} ] | \varphi_{X}(t) | dt \right \} \\
& =  O_{\Pr}(h_{n}^{-2\alpha}m^{-1}),
\end{align*}
where we have used the fact that $| \varphi_{X} |$ is integrable on $\R$. 

Taking these together, we have that
\[
\| \hat{f}_{X} - \hat{f}_{X}^{*} \|_{\R} = o_{\Pr}\left [h_{n}^{-\alpha}\{ (nh_{n})\log h_{n}^{-1} \}^{-1/2} \right ]
\]
by Assumption \ref{as: rate} (b), from which we conclude that
\[
\| \tilde{Z}_{n}- Z_{n}^{*} \|_{I} = o_{\Pr}\{ (\log h_{n}^{-1})^{-1/2} \}.
\]
This shows that there exists a sequence of constants $\Delta_{n} \downarrow 0$ such that
\[
\Pr \{ \| \tilde{Z}_{n} - Z_{n}^{*} \|_{I}  > \Delta_{n} (\log h_{n}^{-1})^{-1/2} \} \leq \Delta_{n}
\]
 (which follows from the fact that convergence in probability is metrized by the Ky Fan metric; see Theorem 9.2.2 in \cite{Du02}), and so 
\begin{align*}
\Pr \{ \| \tilde{Z}_{n} \|_{I} \leq z \} &\leq \Pr \{ \| Z_{n}^{*} \|_{I} \leq z + \Delta_{n} (\log h_{n}^{-1})^{-1/2} \} + \Delta_{n} \\
&\leq  \Pr \{ \| Z_{n}^{G} \|_{I} \leq z + \Delta_{n} (\log h_{n}^{-1})^{-1/2} \} + o(1) \\
&\leq \Pr \{ \| Z_{n}^{G} \|_{I} \leq z \} + o(1)
\end{align*}
uniformly in $z \in \R$, where the second inequality follows from the previous step, and the last inequality follows from the anti-concentration inequality (\ref{eq: AC bound}). Likewise, we have 
$\Pr \{ \| \tilde{Z}_{n} \|_{I} \leq z \} \geq \Pr \{ \| Z_{n}^{G} \|_{I} \leq z \} - o(1)$
uniformly in $z \in \R$, so that we obtain the conclusion of this step.

\textbf{Step 3}. (Proof of the theorem). 
Observe that 
\begin{align*}
\| \hat{K}_{n} - K_{n} \|_{\R} &\leq \frac{1}{2\pi} \int_{\R}  \left | \frac{1}{\hat{\varphi}_{\varepsilon}(t/h_{n})} - \frac{1}{\varphi_{\varepsilon}(t/h_{n})} \right | | \varphi_{K} (t) | dt \\
&\leq O_{\Pr}(h_{n}^{-2\alpha})\int_{\R}  \left | \hat{\varphi}_{\varepsilon}(t/h_{n}) - \varphi_{\varepsilon}(t/h_{n}) \right | | \varphi_{K} (t) |dt, 
\end{align*}
which is $O_{\Pr}(h_{n}^{-2\alpha}m^{-1/2})$ since $
\int_{\R}  \Ep \left [ \left | \hat{\varphi}_{\varepsilon}(t/h_{n}) - \varphi_{\varepsilon}(t/h_{n}) \right | \right ] | \varphi_{K} (t) |dt \lesssim m^{-1/2}$.
Together with the fact that $\| K_{n} \|_{\R} = O(h_{n}^{-\alpha})$, we have 
\[
\| \hat{K}_{n}^{2} - K_{n}^{2} \|_{\R} \leq \| \hat{K}_{n} - K_{n} \|_{\R} \| \hat{K}_{n}+K_{n} \|_{\R} = O_{\Pr}(h_{n}^{-3\alpha} m^{-1/2}),
\]
which yields that
\[
\hat{\sigma}^{2}_{n}(x) = \frac{1}{n} \sum_{j=1}^{n} K_{n}^{2}((x-Y_{j})/h_{n}) - \left ( \frac{1}{n} \sum_{j=1}^{n} K_{n}((x-Y_{j})/h_{n}) \right )^{2}  + O_{\Pr}(h_{n}^{-3\alpha}m^{-1/2})
\]
uniformly in $x \in I$. It is not difficult to verify that 
\begin{align*}
&\frac{1}{n} \sum_{j=1}^{n} K_{n}^{2}((x-Y_{j})/h_{n}) = \Ep [ K_{n}^{2}((x-Y)/h_{n}) ] + O_{\Pr}(h_{n}^{-2\alpha} n^{-1/2}), \ \text{and} \\
&\frac{1}{n} \sum_{j=1}^{n} K_{n}((x-Y_{j})/h_{n}) = \Ep [ K_{n}((x-Y)/h_{n}) ] +O_{\Pr}(h_{n}^{-\alpha} n^{-1/2} )
\end{align*}
uniformly in $x \in I$. Indeed, in view of Lemma \ref{lem: VC type} of the present paper and Corollary A.1 in \cite{ChChKa14a}, these estimates follow from application of Theorem 2.14.1 in \cite{vaWe96}. Therefore, we have $
\hat{\sigma}_{n}^{2}(x) = \sigma_{n}^{2}(x) + O_{\Pr}\{ h_{n}^{-2\alpha} (h_{n}^{-\alpha}m^{-1/2} + n^{-1/2}) \}$
uniformly in $x \in I$, and so $
\hat{\sigma}_{n}^{2}(x)/\sigma_{n}^{2}(x) = 1 + O_{\Pr} \{ h_{n}^{-1} ( h_{n}^{-\alpha}m^{-1/2} + n^{-1/2} ) \}$
uniformly in $x \in I$ by Lemma \ref{lem: variance}, where 
\[
h_{n}^{-1} (h_{n}^{-\alpha}m^{-1/2} + n^{-1/2}) = o\{ (\log h_{n}^{-1})^{-1} \}
\]
 by Assumption \ref{as: rate}. This yields that $\| \sigma_{n}/\hat{\sigma}_{n} - 1 \|_{I} = o_{\Pr} \{ (\log h_{n}^{-1})^{-1} \}$.

By Steps 1 and 2 together with the fact that $\Ep[\| Z_{n}^{G} \|_{I}]=O(\sqrt{\log h_{n}^{-1}})$, we see that $\| \tilde{Z}_{n} \|_{I} = O_{\Pr}(\sqrt{\log h_{n}^{-1}})$. 
So,
\[
\| \hat{Z}_{n} - \tilde{Z}_{n} \|_{I} \leq \left \| \sigma_{n}/\hat{\sigma}_{n} - 1 \right \|_{I} \| \tilde{Z}_{n} \|_{I} = o_{\Pr} \{ (\log h_{n}^{-1})^{-1/2} \},
\]
and arguing as in the last part of the proof of  Step 2, we conclude that
\[
\sup_{z \in \R} | \Pr \{ \| \hat{Z}_{n} \|_{I} \leq z \} - \Pr \{ \| Z_{n}^{G} \|_{I} \leq z \} | \to 0.
\]
This completes the proof of Theorem \ref{thm: Gaussian approximation}. 
\end{proof}

\subsection{Proofs of Corollaries \ref{thm: uniform convergence rate} and  \ref{cor: uniform convergence rate}}

\begin{proof}[Proof of Corollary \ref{thm: uniform convergence rate}]
Step 2 in the proof of Theorem \ref{thm: Gaussian approximation} yields that 
\[
\| \hat{f}_{X} - \hat{f}_{X}^{*} \|_{\R} = o_{\Pr} \{ h_{n}^{-\alpha}(nh_{n})^{-1/2} \}
\]
(it is not difficult to verify that Assumption \ref{as: variance} was not used to derive this rate). 
Hence we have to show that $\| \hat{f}_{X}^{*} (\cdot) - \Ep[\hat{f}_{X}^{*}(\cdot)]  \|_{\R}= O_{\Pr} \{ h_{n}^{-\alpha}(nh_{n})^{-1/2}\sqrt{\log h_{n}^{-1}} \}$. 
To this end, we make use of Corollary 5.1 in \cite{ChChKa14a}. Invoke Lemma \ref{lem: VC type} and observe that $\| K_{n} \|_{\R} \lesssim h_{n}^{-\alpha}$ and $\sigma_{n}^{2} (x) \lesssim h_{n}^{-2\alpha+1}$. 
The latter bound follows from 
\[
\sigma_{n}^{2}(x) \leq \Ep[K_{n}^{2}((x-Y)/h_{n})] \leq h_{n} \| f_{Y} \|_{\R} \int_{\R} K_{n}^{2}(y) dy = \frac{h_{n} \| f_{Y} \|_{\R}}{2\pi} \int_{\R} \frac{| \varphi_{K}(t)|^{2}}{| \varphi_{\varepsilon}(t/h) |^{2}} dt \lesssim h_{n}^{-2\alpha+1}.
\]
Then application of Corollary 5.1 in \cite{ChChKa14a} to the function class $\{ f - \Ep[ f (Y) ] : f \in \mathcal{K}_{n} \}$ yields that 
\begin{align*}
&\Ep \left [  \left \| \sum_{j=1}^{n} \{ K_{n} ((\cdot - Y_{j})/h_{n}) - \Ep[K_{n}((\cdot -Y)/h_{n})] \} \right \|_{\R} \right ] \\
&\quad \lesssim h_{n}^{-\alpha}  \sqrt{n h_{n} \log h_{n}^{-1}} + h_{n}^{-\alpha} \log h_{n}^{-1} \\
&\quad \lesssim h_{n}^{-\alpha} \sqrt{nh_{n} \log h_{n}^{-1}}, 
\end{align*}
which in turn yields that 
\[
\Ep [ \| \hat{f}^{*}_{X}(\cdot) - \Ep[\hat{f}^{*}_{X}(\cdot)] \|_{\R}] \lesssim h_{n}^{-\alpha}(nh_{n})^{-1/2}  \sqrt{\log h_{n}^{-1}}.
\]
This completes the proof. 
\end{proof}

\begin{proof}[Proof of Corollary \ref{cor: uniform convergence rate}]
By assumption, 
$\int_{\R} x^{\ell} K(x) dx= 0$ for all $\ell =1,\dots, k$. 
Further, observe that 
\[
\Ep[  \hat{f}_{X}^{*}(x) ] = \frac{1}{2\pi} \int_{\R} e^{-itx} \varphi_{X}(t) \varphi_{K}(th_{n}) dt = [f_{X}*(h_{n}^{-1}K(\cdot/h))] (x) = \int_{\R} f_{X}(x-h_{n}y) K(y)dy.
\]
Hence, using the Taylor expansion, we have that
\begin{align*}
\| \Ep[ \hat{f}_{X}^{*}(\cdot) ] - f_{X}(\cdot) \|_{\R} &=\left \|  \int_{\R} \{ f_{X}(\cdot - h_{n}y) - f_{X}(\cdot) \} K(y) dy \right \|_{\R} \\
&\leq \frac{B h_{n}^{\beta}}{k!} \int_{\R} |y|^{\beta} |K(y)| dy.
\end{align*}
Combining the result of Corollary \ref{thm: uniform convergence rate}, we obtain the desired conclusion. 
\end{proof}

\subsection{Proofs of Theorem \ref{thm: multiplier bootstrap} and Corollary \ref{cor: multiplier bootstrap}}
\begin{proof}[Proof of Theorem \ref{thm: multiplier bootstrap}]
We divide the proof into three  steps. 

\textbf{Step 1}. Define 
\[
Z_{n}^{\xi}(x) = \frac{1}{\sigma_{n}(x)\sqrt{n}} \sum_{j=1}^{n} \xi_{j} \left \{ K_{n}((x-Y_{j})/h_{n}) - n^{-1} {\textstyle \sum}_{j'=1}^{n} K_{n}((x-Y_{j'})/h_{n}) \right \}
\]
for $x \in I$. We first prove that 
\[
\sup_{z \in \R} | \Pr \{ \| Z_{n}^{\xi} \|_{I} \leq z \mid \mathcal{D}_{n} \} - \Pr \{ \| Z_{n}^{G} \|_{I} \leq z \} | \stackrel{\Pr}{\to} 0.
\]
To this end, we make use of Theorem 2.2 in \cite{ChChKa16}. Recall the class of functions $\mG_{n}$ defined in the proof of Theorem \ref{thm: Gaussian approximation}, and let 
\[
\nu_{n}^{\xi}(g) = \frac{1}{\sqrt{n}} \sum_{j=1}^{n} \xi_{j} \left \{ g(Y_{j}) - n^{-1} {\textstyle \sum}_{j'=1}^{n} g(Y_{j'}) \right  \}, \ g \in \mG_{n}.
\]
Then application of Theorem 2.2 in \cite{ChChKa16} to $\| \nu_{n} \|_{\mG_{n}}$ with $B(f) \equiv 0, b \lesssim 1/\sqrt{h_{n}}, \sigma = 1, A \lesssim 1, v \lesssim 1, \gamma = 1/\log n$, and $q$ sufficiently large, yields that there exists a random variable $W_{n}^{\xi}$ of which the conditional distribution given $\mathcal{D}_{n}$ is the same as the distribution of $\| G_{n} \|_{\mG_{n}} (= \| Z_{n}^{G} \|_{I})$, i.e., $\Pr\{ W_{n}^{\xi} \leq z \mid \mathcal{D}_{n} \} = \Pr \{ \| Z_{n}^{G}\|_{I} \leq z \}$ for all $z \in \R$ almost surely,  and such that 
\[
| \| \nu_{n}^{\xi} \|_{\mG_{n}} - W_{n}^{\xi} | = o_{\Pr} \{ (\log h_{n}^{-1})^{-1/2} \},
\]
which shows that there exists a sequence of constants $\Delta_{n} \downarrow 0$ such that 
\[
\Pr \{ | \| \nu_{n}^{\xi} \|_{\mG_{n}} - W_{n}^{\xi} |  > \Delta_{n} (\log h_{n}^{-1})^{-1/2} \mid \mathcal{D}_{n} \} \stackrel{\Pr}{\to} 0
\]
by Markov's inequality. 
Since $\| \nu_{n}^{\xi} \|_{\mG_{n}} = \| Z_{n}^{\xi} \|_{I}$, we have that 
\begin{align*}
\Pr \{ \| Z_{n}^{\xi} \|_{I} \leq z \mid \mathcal{D}_{n} \} &\leq \Pr \{  W_{n}^{\xi} \leq z+\Delta_{n} (\log h_{n}^{-1})^{-1/2} \mid \mathcal{D}_{n}  \} + o_{\Pr}(1) \\
&=\Pr \{ \| Z_{n}^{G} \|_{I} \leq z+\Delta_{n} (\log h_{n}^{-1})^{-1/2} \} + o_{\Pr}(1)
\end{align*}
uniformly  in $z \in \R$, and the anti-concentration inequality (\ref{eq: AC bound}) yields that 
\[
\Pr \{ \| Z_{n}^{G} \|_{I} \leq z+\Delta_{n} (\log h_{n}^{-1})^{-1/2} \} \leq \Pr \{ \| Z_{n}^{G} \|_{I} \leq z \} + o(1)
\]
uniformly in $z \in \R$. Likewise, we have 
\[
\Pr \{ \| Z_{n}^{\xi} \|_{I} \leq z \mid \mathcal{D}_{n} \} \geq \Pr \{ \| Z_{n}^{G} \|_{I} \leq z \} - o_{\Pr}(1)
\]
uniformly in $z \in \R$. Therefore, we obtain the conclusion of this step.

\textbf{Step 2}. In view of the proof of Step 1, in order to prove the result (\ref{eq: bootstrap}), it is enough to prove that 
\[
\| \hat{Z}_{n}^{\xi} - Z_{n}^{\xi} \|_{I} = o_{\Pr}\{ (\log h_{n}^{-1})^{-1/2} \}.
\]
Define 
\[
\tilde{Z}_{n}^{\xi} = \frac{1}{\sigma_{n}(x)\sqrt{n}} \sum_{j=1}^{n} \xi_{j} \left \{ \hat{K}_{n}((x-Y_{j})/h_{n}) - n^{-1} {\textstyle \sum}_{j'=1}^{n} \hat{K}_{n}((x-Y_{j'})/h_{n}) \right \}
\]
for $x \in I$. We first prove that $\| \tilde{Z}_{n}^{\xi} - Z_{n}^{\xi} \|_{I} = o_{\Pr} \{ (\log h_{n}^{-1})^{-1/2} \}$. Step 2 in the proof of Theorem \ref{thm: Gaussian approximation} shows that 
\[
\left \| \frac{1}{\sigma_{n}(\cdot)\sqrt{n}} \sum_{j=1}^{n} \{ \hat{K}_{n}((\cdot - Y_{j})/h_{n}) - K_{n}((\cdot - Y_{j})/h_{n}) \} \right \|_{I}
\]
is $o_{\Pr}\{ (\log h_{n}^{-1})^{-1/2} \}$ (which is more that what we need), and so it remains to prove that 
\[
\left \| \frac{1}{\sigma_{n}(\cdot )\sqrt{n}} \sum_{j=1}^{n} \xi_{j}\{ \hat{K}_{n}((\cdot-Y_{j})/h_{n}) - K_{n}((\cdot - Y_{j})/h_{n}) \} \right \|_{I}
\]
is $o_{\Pr}\{ (\log h_{n}^{-1})^{-1/2} \}$. Since $\| 1/\sigma_{n} \|_{I} \lesssim h_{n}^{\alpha-1/2}$, it is enough to prove that
\[
\left \| \sum_{j=1}^{n} \xi_{j}\{ \hat{K}_{n}((\cdot-Y_{j})/h_{n}) - K_{n}((\cdot - Y_{j})/h_{n}) \} \right \|_{I} 
\]
is $o_{\Pr}\{  h_{n}^{-\alpha}(nh_{n})^{1/2} (\log h_{n}^{-1})^{-1/2} \}$. Observe that 
\begin{align*}
&\left | \sum_{j=1}^{n} \xi_{j} \{ \hat{K}_{n}((x-Y_{j})/h_{n}) -  K_{n}((x- Y_{j})/h_{n}) \} \right | \\
&\lesssim \int_{\R} \left | \sum_{j=1}^{n} \xi_{j} e^{itY_{j}/h_{n}} \right | \left |  \frac{1}{\hat{\varphi}_{\varepsilon} (t/h_{n})} - \frac{1}{\varphi_{\varepsilon}(t/h_{n})} \right | | \varphi_{K}(t) | dt \\
&\lesssim \left \{ \int_{\R} \left | \sum_{j=1}^{n} \xi_{j} e^{itY_{j}/h_{n}} \right |^{2} | \varphi_{K}(t) | dt \right \}^{1/2}  \left \{ \int_{\R} \left |  \frac{1}{\hat{\varphi}_{\varepsilon} (t/h_{n})} - \frac{1}{\varphi_{\varepsilon}(t/h_{n})} \right |^{2} | \varphi_{K}(t) | dt \right \}^{1/2} \\
&= O_{\Pr} (n^{1/2} h_{n}^{-2\alpha} m^{-1/2}),
\end{align*}
which is $o_{\Pr}\{ h_{n}^{-\alpha}(nh_{n})^{1/2}  (\log h_{n}^{-1})^{-1/2} \}$ by Assumption \ref{as: rate} (b). Therefore, we have $\| \tilde{Z}_{n}^{\xi} - Z_{n}^{\xi} \|_{I} = o_{\Pr} \{ (\log h_{n}^{-1})^{-1/2} \}$.

By Step 1, the fact that $\Ep [ \| Z_{n}^{G} \|_{I}] = O(\sqrt{\log h_{n}^{-1}})$ (see the proof of Theorem \ref{thm: Gaussian approximation}), and the previous result, we see that $\| \tilde{Z}_{n}^{\xi} \|_{I} = O_{\Pr}(\sqrt{\log h_{n}^{-1}})$.
Now, because  $\| \sigma_{n}/\hat{\sigma}_{n}- 1 \|_{I} = o_{\Pr} \{ (\log h_{n}^{-1})^{-1} \}$ from Step 3 in the proof of Theorem \ref{thm: Gaussian approximation}, we conclude that 
\[
\| \hat{Z}_{n}^{\xi} - \tilde{Z}_{n}^{\xi} \|_{I} \leq \| \sigma_{n}/\hat{\sigma}_{n}- 1 \|_{I} \| \tilde{Z}_{n}^{\xi} \|_{I} = o_{\Pr}\{ (\log h_{n}^{-1})^{-1/2} \},
\]
which leads to (\ref{eq: bootstrap}). 

\textbf{Step 3}. In this step, we shall verify the last two assertions of the theorem. The result (\ref{eq: bootstrap}) implies that there exists a sequence of constants $\Delta_{n} \downarrow 0$ such that with probability greater than $1-\Delta_{n}$, 
\begin{equation}
\sup_{z \in \R} | \Pr \{ \| \hat{Z}_{n}^{\xi} \|_{I} \leq z \mid \mathcal{D}_{n} \} - \Pr \{ \| Z_{n}^{G} \|_{I} \leq z \} | \leq \Delta_{n}. \label{eq: bootstrap 2}
\end{equation}
Taking $\Delta_{n} \downarrow 0$ more slowly if necessary, we also have 
\[
\sup_{z \in \R} | \Pr \{ \| \hat{Z}_{n} \|_{I} \leq z \} - \Pr \{ \| Z_{n}^{G} \|_{I} \leq z \} | \leq \Delta_{n}.
\]
Let $\mathcal{E}_{n}$ denote the event on which (\ref{eq: bootstrap 2}) holds, and let $c_{n}^{G}(u)$ denote the $u$-quantile of $\| Z_{n}^{G} \|_{I}$ for $u \in (0,1)$. Then on the event $\mathcal{E}_{n}$, 
\[
\Pr \{ \| \hat{Z}_{n}^{\xi} \|_{I} \leq c_{n}^{G}(1-\tau + \Delta_{n}) \mid \mathcal{D}_{n} \} \geq \Pr \{ \| Z_{n}^{G} \|_{I} \leq c_{n}^{G}(1-\tau + \Delta_{n}) \} - \Delta_{n} = 1-\tau,
\]
where the last equality holds since $\| Z_{n}^{G} \|_{I}$ has a continuous distribution function (recall the anti-concentration inequality (\ref{eq: AC bound})). This yields that the inequality 
\[
\hat{c}_{n}(1-\tau) \leq c_{n}^{G}(1-\tau + \Delta_{n})
\]
holds on $\mathcal{E}_{n}$, and so 
\begin{align*}
\Pr \{ \| \hat{Z}_{n} \|_{I} \leq \hat{c}_{n}(1-\tau) \} &\leq \Pr \{ \| \hat{Z}_{n} \|_{I} \leq c_{n}^{G}(1-\tau + \Delta_{n}) \} + \Pr (\mathcal{E}_{n}^{c}) \\
&\leq \Pr \{ \| Z_{n}^{G} \|_{I} \leq c_{n}^{G}(1-\tau + \Delta_{n}) \} + 2 \Delta_{n} \\
&= 1-\tau + 3 \Delta_{n}.
\end{align*}
Likewise, we have 
\[
\Pr \{ \| \hat{Z}_{n} \|_{I} \leq \hat{c}_{n}(1-\tau) \} \geq 1-\tau - 3\Delta_{n}.
\] 
This leads to the result (\ref{eq: validity of MB}). 

Finally, the Borell-Sudakov-Tsirelson inequality \citep[][Lemma A.2.2]{vaWe96} yields that 
\[
c_{n}^{G}(1-\tau+\Delta_{n}) \leq \Ep[ \| Z_{n}^{G} \|_{I}] + \sqrt{2\log (1/(\tau-\Delta_{n}))} \lesssim \sqrt{\log h_{n}^{-1}},
\]
which implies that $\hat{c}_{n}(1-\tau) = O_{\Pr}(\sqrt{\log h_{n}^{-1}})$. Furthermore,
\[
\sup_{x \in I} \hat{\sigma}_{n}(x) \leq \sup_{x \in I} \sigma_{n}(x) \cdot \sup_{x \in I} \frac{\hat{\sigma}_{n}(x)}{\sigma_{n}(x)} = O_{\Pr}(h_{n}^{-\alpha+1/2}).
\]
Therefore, the supremum width of the band $\hat{\mathcal{C}}_{n}$ is 
\[
2\sup_{x \in I} \frac{\hat{\sigma}_{n}(x)}{\sqrt{n}h_{n}} \hat{c}_{n}(1-\tau) = O_{\Pr}\left \{ h_{n}^{-\alpha}(nh_{n})^{-1/2}\sqrt{\log h_{n}^{-1}} \right \}.
\]
This completes the proof. 
\end{proof}

\begin{proof}[Proof of Corollary \ref{cor: multiplier bootstrap}]
Recall the stochastic process  $\tilde{Z}_{n}(x), x \in I$ defined in the proof of Theorem \ref{thm: Gaussian approximation}. Observe that $\| f_{X} (\cdot) - \Ep [ \hat{f}_{X}^{*}(\cdot) ] \|_{\R} = O(h_{n}^{\beta})$. Condition (\ref{eq: undersmoothing}) then yields that 
\begin{align*}
\frac{\sqrt{n}h_{n} (\hat{f}_{X}(x) - f_{X}(x))}{\hat{\sigma}_{n}(x)} &= \left [  1 + o_{\Pr} \{ (\log h_{n}^{-1})^{-1} \} \right ] \frac{\sqrt{n}h_{n} (\hat{f}_{X}(x) - f_{X}(x))}{\sigma_{n}(x)} \\
&= \left [  1 + o_{\Pr} \{ (\log h_{n}^{-1})^{-1} \} \right ] \tilde{Z}_{n} (x) + o \{ (\log h_{n}^{-1})^{-1/2} \} \\
&= \tilde{Z}_{n} (x) + o_{\Pr} \{ (\log h_{n}^{-1})^{-1/2} \}
\end{align*}
uniformly in $x \in I$, where we have used the facts that $\| \sigma_{n}/\hat{\sigma}_{n} - 1 \|_{I} = o_{\Pr} \{ (\log h_{n}^{-1})^{-1} \}$ and $\| \tilde{Z}_{n} \|_{I} = O_{\Pr} (\sqrt{\log h_{n}^{-1}})$ (these estimates are derived in the proof of Theorem \ref{thm: Gaussian approximation}). Using the anti-concentration inequality (\ref{eq: AC bound}) together with the result of Step 2 in the proof of Theorem \ref{thm: Gaussian approximation}, we have that
\[
\sup_{z \in \R} \left | \Pr \left \{  \left \| \sqrt{n}h_{n}(\hat{f}_{X}-f_{X})/\hat{\sigma}_{n} \right \|_{I} \leq z \right \} - \Pr \{ \| Z_{n}^{G} \|_{I} \leq z \} \right | \to 0.
\]
Now, arguing as in Step 3 in the proof of Theorem \ref{thm: multiplier bootstrap}, we conclude that 
\[
 \Pr \left \{  \left \| \sqrt{n}h_{n}(\hat{f}_{X}-f_{X})/\hat{\sigma}_{n} \right \|_{I} \leq \hat{c}_{n}(1-\tau) \right \} \to 1-\tau,
\]
which yields the desired result. 
\end{proof}

\subsection{Proof of Theorem \ref{thm: validity of MB panel}}

 For the notational convenience, in this proof, we assume $d=1$, i.e,, $W_{j,t}$ are univariate; the proof for the general case is completely analogous. 
Let $W_{j}^{+} = (W_{j,1}+W_{j,2})/2, W_{j}^{-} = (W_{j,1}-W_{j,2})/2$, and observe that 
\[
\hat{Y}_{j}^{\dagger} = Y_{j}^{\dagger} - (\hat{\theta}-\theta_{0})W_{j}^{+}, \ \hat{\eta}_{j} = \eta_{j} - (\hat{\theta}-\theta_{0})W_{j}^{-}.
\]
First, we shall show that 
\begin{equation}
\left \| \hat{f}_{U}  (\cdot) - \frac{1}{nh_{n}} \sum_{j=1}^{n} K_{n}((\cdot-Y_{j}^{\dagger})/h_{n}) \right \|_{\R} = o_{\Pr} \{ h_{n}^{-\alpha} (nh_{n})^{-1/2} (\log h_{n}^{-1})^{-1/2} \}, \label{eq: approx panel}
\end{equation}
where 
\[
K_{n}(u)= \frac{1}{2\pi} \int_{\R} e^{-itu} \frac{\varphi_{K}(t)}{\varphi_{\varepsilon}(t/h_{n})} dt.
\]
By Taylor's theorem, we have that $|e^{ity} - e^{itx} - (it) e^{itx} (y-x) | \leq t^{2}(y-x)^{2}/2$
for any $x,y \in \R$, which yields that
\[
\left | \hat{\varphi}_{Y^{\dagger}}(t) - \hat{\varphi}_{Y^{\dagger}}^{*}(t) +  \frac{(it)(\hat{\theta}-\theta_{0})}{n} \sum_{j=1}^{n} e^{itY_{j}^{\dagger}}W_{j}^{+}  \right | \leq \frac{t^{2} (\hat{\theta}-\theta_{0})^{2}}{2n} \sum_{j=1}^{n} (W_{j}^{+})^{2}.
\]
The right hand side is $O_{\Pr} (n^{-1}h_{n}^{-2})$ uniformly in $|t| \leq h_{n}^{-1}$. Observe that 
\[
\frac{1}{n} \sum_{j=1}^{n} e^{itY_{j}^{\dagger}}W_{j}^{+} = \frac{1}{n} \sum_{j=1}^{n} \{ e^{itY_{j}^{\dagger}}W_{j}^{+}  - \Ep[ e^{itY^{\dagger}_{1}} W_{1}^{+}] \} + \underbrace{\Ep[ e^{itY^{\dagger}_{1}} W_{1}^{+}]}_{=\varphi_{\varepsilon}(t) \Ep[ e^{itU_{1}} W_{1}^{+}]}.
\]
Following the proof of Theorem 4.1 in \cite{NeRe09} (cf. Lemma \ref{lem: rate ecf} ahead), under our assumption, we can show that 
\[
\sup_{|t| \leq h_{n}^{-1}} \left | \frac{1}{n} \sum_{j=1}^{n} \{ e^{itY_{j}^{\dagger}}W_{j}^{+}  - \Ep[ e^{itY^{\dagger}_{1}} W_{1}^{+}] \} \right | = O_{\Pr}\{ n^{-1/2} \log h_{n}^{-1} \}.
\]
Taking these together, we have that
\[
\sup_{|t| \leq h_{n}^{-1}} \left |\frac{ \hat{\varphi}_{Y^{\dagger}}(t) - \hat{\varphi}_{Y^{\dagger}}^{*}(t)}{\varphi_{\varepsilon}(t)} +(it)(\hat{\theta} - \theta_{0})  \Ep[ e^{itU_{1}} W_{1}^{+}]\right |  = O_{\Pr}( n^{-1}h_{n}^{-\alpha-2}). 
\]
Likewise, we have that
\[
\sup_{|t| \leq h_{n}^{-1}} \left |\frac{ \hat{\varphi}_{\varepsilon}(t) - \hat{\varphi}_{\varepsilon}^{*}(t)}{\varphi_{\varepsilon}(t)} + (it)(\hat{\theta} - \theta_{0}) \Ep[ W_{1}^{-} ] \right |  = O_{\Pr}(n^{-1}h_{n}^{-\alpha-2}), 
\]
which in particular ensures that $\sup_{|t| \leq h_{n}^{-1}} | \hat{\varphi}_{\varepsilon}(t)/\varphi_{\varepsilon}(t)-1| = o_{\Pr}(1)$ since $\sup_{|t| \leq h_{n}}  | \hat{\varphi}_{\varepsilon}^{*}(t)/\varphi_{\varepsilon}(t)-1| = o_{\Pr}(1)$ (cf. Step 2 in the proof of Theorem \ref{thm: Gaussian approximation}). Hence
\[
\sup_{u \in \R} \Bigg | \hat{f}_{U}(u) - \underbrace{\frac{1}{2\pi} \int_{\R} e^{-itu} \hat{\varphi}^{*}_{Y^{\dagger}} (t) \frac{\varphi_{K}(th_{n})}{\hat{\varphi}_{\varepsilon}(t)} dt}_{=\frac{1}{nh_{n}} \sum_{j=1}^{n} \hat{K}_{n}((x-Y^{\dagger}_{j})/h_{n})} \Bigg | \leq O_{\Pr}(n^{-1/2}) \int_{-h_{n}^{-1}}^{h_{n}^{-1}}  | t \Ep[ e^{itU_{1}} W_{1}^{+} ] | dt +  O_{\Pr}(n^{-1}h_{n}^{-\alpha-3}).
\]
By assumption, the first term on the right hand side is $o_{\Pr} \{ h_{n}^{-\alpha} (nh_{n})^{-1/2} (\log h_{n}^{-1})^{-1/2} \}$, and so is the second term since $\frac{\log h_{n}^{-1}}{nh_{n}^{5}}\to 0$. 
Therefore, we have that 
\[
\left \| \hat{f}_{U} (\cdot) - \frac{1}{nh_{n}} \sum_{j=1}^{n} \hat{K}_{n}((\cdot-Y^{\dagger}_{j})/h_{n}) \right \|_{\R} = o_{\Pr} \{ h_{n}^{-\alpha} (nh_{n})^{-1/2} (\log h_{n}^{-1})^{-1/2} \}.
\]
Furthermore, observe that 
\begin{align}
&\sup_{|t| \leq h_{n}^{-1}} \left |\frac{ \hat{\varphi}_{\varepsilon}(t) - \hat{\varphi}_{\varepsilon}^{*}(t)}{\hat{\varphi}_{\varepsilon}(t)} \right | 
\leq O_{\Pr} (1) \sup_{|t| \leq h_{n}^{-1}} \left |\frac{ \hat{\varphi}_{\varepsilon}(t) - \hat{\varphi}_{\varepsilon}^{*}(t)}{\varphi_{\varepsilon}(t)} \right | \notag \\
&\quad =
 O_{\Pr} (n^{-1/2}h_{n}^{-1}  + n^{-1} h_{n}^{-\alpha-2} ) = O_{\Pr} (n^{-1/2} h_{n}^{-1}), \label{eq: ecf estimated}
\end{align}
where the last equality follows since $nh_{n}^{2\alpha+2} \to \infty$, 
and observe that
\[
\frac{\varphi_{\varepsilon}(t)}{\hat{\varphi}_{\varepsilon}(t)} - 1 = \frac{\hat{\varphi}_{\varepsilon}^{*}(t)-\hat{\varphi}_{\varepsilon}(t)}{\hat{\varphi}_{\varepsilon}(t)} + \frac{1}{\hat{\varphi}_{\varepsilon}(t)} \{ \varphi_{\varepsilon}(t) - \hat{\varphi}_{\varepsilon}^{*}(t)\}.
\]
The first term on the right hand side is $O_{\Pr}(n^{-1/2}h_{n}^{-1})$ uniformly in $|t| \leq h_{n}^{-1}$, and $\| 1/\hat{\varphi}_{\varepsilon} \|_{[-h_{n}^{-1},h_{n}^{-1}]} \leq O_{\Pr}(1) \| 1/\varphi_{\varepsilon} \|_{[-h_{n}^{-1},h_{n}^{-1}]} = O_{\Pr}(h_{n}^{-\alpha})$. 
Combining these bounds and arguing as in Step 2 in the proof of Theorem \ref{thm: Gaussian approximation}, we obtain the result (\ref{eq: approx panel}). 
Note that the condition $\alpha > 1/2$ is used to ensure that $n^{-1/2}h_{n}^{-1} = o \{ h_{n}^{-\alpha} (nh_{n})^{-1/2} (\log h_{n}^{-1})^{-1/2} \}$. 

Second, let $\sigma_{n}^{2}(u)  = \Var (K_{n}((u-Y_{1}^{\dagger})/h_{n}))$ for $u \in I$, and we shall show that
\begin{equation}
\| \hat{\sigma}^{2}_{n}/\sigma^{2}_{n} - 1 \|_{I} = o_{\Pr}\{ (\log h_{n}^{-1})^{-1} \}.
\label{eq: variance panel}
\end{equation}
By Lemma \ref{lem: variance}, we have that $\inf_{u \in I} \sigma_{n}^{2}(u) \gtrsim h_{n}^{-2\alpha+1}$. From (\ref{eq: ecf estimated}), it is not difficult to verify that $\| \hat{K}_{n} - K_{n} \|_{\R} = O_{\Pr}(n^{-1/2} h_{n}^{-2\alpha} + n^{-1/2} h_{n}^{-\alpha-1})$, 
so that $\| \hat{K}_{n}^{2} - K_{n}^{2} \|_{\R} \leq \| \hat{K}_{n} + K_{n} \|_{\R} \| \hat{K}_{n} - K_{n} \|_{\R} = O_{\Pr}(n^{-1/2} h_{n}^{-3\alpha} + n^{-1/2} h_{n}^{-2\alpha-1})$, which is $o_{\Pr}\{h_{n}^{-2\alpha+1} (\log h_{n}^{-1})^{-1} \}$ under our assumption, so that 
\[
\frac{1}{n} \sum_{j=1}^{n} \hat{K}_{n}^{2}((u-\hat{Y}_{j}^{\dagger})/h_{n}) = \frac{1}{n} \sum_{j=1}^{n} K_{n}^{2}((u-\hat{Y}_{j}^{\dagger})/h_{n}) + o_{\Pr} \{h_{n}^{-2\alpha+1} (\log h_{n}^{-1})^{-1} \}
\]
uniformly in $u \in I$. We want to replace $\hat{Y}_{j}^{\dagger}$ by $Y_{j}^{\dagger}$ on the right hand side. Observe that 
$\| (K_{n}^{2})' \|_{\R} = \| 2 K_{n}' K_{n} \|_{\R} \lesssim h_{n}^{-2\alpha}$, so that 
\begin{align*}
&\left \| \frac{1}{n} \sum_{j=1}^{n} \{ K_{n}^{2}((\cdot-\hat{Y}_{j}^{\dagger})/h_{n}) - K_{n}^{2}((\cdot-Y_{j}^{\dagger})/h_{n}) \} \right \|_{I} \lesssim h_{n}^{-2\alpha-1} | \hat{\theta} - \theta_{0} | \frac{1}{n}\sum_{j=1}^{n} | W_{j}^{+} | \\
&\quad = O_{\Pr} (n^{-1/2} h_{n}^{-2\alpha-1}) = o_{\Pr} \{h_{n}^{-2\alpha+1} (\log h_{n}^{-1})^{-1} \}. 
\end{align*}
Applying a similar analysis to the term  $n^{-1}\sum_{j=1}^{n} \hat{K}_{n}((x-\hat{Y}_{j}^{\dagger})/h_{n})$, we conclude that 
\[
\hat{\sigma}_{n}^{2}(u) = \underbrace{\frac{1}{n} \sum_{j=1}^{n} K_{n}^{2}((u-Y_{j}^{\dagger})/h_{n})  - \left ( \frac{1}{n} \sum_{j=1}^{n}K_{n}((u-Y_{j}^{\dagger})/h_{n})  \right )^{2}}_{=(\star)} + o_{\Pr} \{h_{n}^{-2\alpha+1} (\log h_{n}^{-1})^{-1} \}
\]
uniformly in $u \in I$. Finally, Step 3 in the proof of Theorem \ref{thm: Gaussian approximation} shows that $(\star) = o_{\Pr} \{ (\log h_{n}^{-1})^{-1} \}$ uniformly in $u \in I$, so that $\| \hat{\sigma}_{n}^{2}/\sigma_{n}^{2} - 1 \|_{I} = o_{\Pr} \{ (\log h_{n}^{-1})^{-1} \}$.

Now, from the proof of Theorem \ref{thm: Gaussian approximation}, together with that $\| \Ep[ h_{n}^{-1} K_{n}((\cdot-Y_{1}^{\dagger})/h_{n})] - f_{U}(\cdot) \|_{I} \lesssim h_{n}^{\beta} = o \{h_{n}^{-2\alpha+1} (\log h_{n}^{-1})^{-1} \}$ by our choice of the bandwidth, we conclude that there exists a tight Gaussian random variable $Z_{n}^{G}$ in $\ell^{\infty}(I)$ with mean zero and covariance function $\Cov (Z_{n}^{G}(u),Z_{n}^{G}(v)) = \Cov \{ K_{n}((u-Y_{1}^{\dagger})/h_{n}),K_{n}((v-Y_{1}^{\dagger})/h_{n}) \}/\{ \sigma_{n}(u)\sigma_{n}(v) \}$ for $u,v \in I$, and such that as $n \to \infty$, 
\[
\sup_{z \in \R} \left | \Pr \left \{ \left \| \sqrt{n} h_{n} (\hat{f}_{U} - f_{U})/\hat{\sigma}_{n} \right \|_{I} \leq z \right \} - \Pr \{ \| Z_{n}^{G} \|_{I} \leq z \} \right | \to 0.
\]
In view of the proof of Theorem \ref{thm: multiplier bootstrap}, the desired result follows as soon as we verify that 
\[
\sup_{z \in \R} \left | \Pr \{ \| \hat{Z}_{n}^{\xi} \|_{I} \leq z \mid \mathcal{D}_{n} \} - \Pr \{ \| Z_{n}^{G} \|_{I} \leq z \} \right | \stackrel{\Pr}{\to} 0.
\]
From the proof of Theorem \ref{thm: multiplier bootstrap} and the result (\ref{eq: variance panel}), what we need to verify is that 
\[
\left \| \sum_{j=1}^{n} \xi_{j} \{ \hat{K}_{n}((\cdot-\hat{Y}_{j}^{\dagger})/h_{n}) - K_{n}((\cdot - Y_{j}^{\dagger})/h_{n}) \} \right \|_{I} = o_{\Pr} \{ h_{n}^{-\alpha}(nh_{n})^{1/2} (\log h_{n}^{-1})^{-1/2} \}.
\]
Observe that 
\begin{multline*}
| \hat{K}_{n}((u-\hat{Y}_{j}^{\dagger})/h_{n}) -\hat{K}_{n}((u-Y_{j}^{\dagger})/h_{n}) - (\hat{\theta} - \theta_{0})\hat{K}_{n}'((u-Y_{j}^{\dagger})/h_{n}) W_{j}^{+}/h_{n} | \\
\leq \frac{(\hat{\theta}- \theta_{0})^{2}}{2h_{n}^{2}} \| \hat{K}_{n}'' \|_{\R} (W_{j}^{+})^{2}.
\end{multline*}
Since $\| \hat{K}_{n}' \|_{\R} \lesssim \| 1/\hat{\varphi}_{\varepsilon} \|_{[-h_{n}^{-1},h_{n}^{-1}]} = O_{\Pr}(h_{n}^{-\alpha})$, we have that 
\begin{align}
&\left \| \sum_{j=1}^{n} \xi_{j} \{ \hat{K}_{n}((\cdot-\hat{Y}_{j}^{\dagger})/h_{n}) - \hat{K}_{n}((\cdot - Y_{j}^{\dagger})/h_{n}) \} \right \|_{I} \notag \\
&\leq O_{\Pr}(n^{-1/2} h_{n}^{-1}) \left \| \sum_{j=1}^{n} \xi_{j}\hat{K}_{n}((\cdot-Y_{j}^{\dagger})/h_{n}) W_{j}^{+} \right \|_{I} + O_{\Pr}(h_{n}^{-\alpha-2}), 
\label{eq: MB mid panel}
\end{align}
where $h_{n}^{-\alpha-2}  = o \{ h_{n}^{-\alpha}(nh_{n})^{1/2} (\log h_{n}^{-1})^{-1/2} \}$ under our assumption. Observe that 
\begin{align*}
&\left \| \sum_{j=1}^{n} \xi_{j}\hat{K}_{n}((\cdot-Y_{j}^{\dagger})/h_{n}) W_{j}^{+} \right \|_{I} \lesssim \int_{\R} \left | \sum_{j=1}^{n} \xi_{j} e^{itY_{j}^{\dagger}/h_{n}} W_{j}^{+} \right | \left | \frac{\varphi_{K}(t)}{\hat{\varphi}_{\varepsilon}(t/h_{n})} \right | dt \\
&\quad \lesssim O_{\Pr}(h_{n}^{-\alpha}) \int_{-1}^{1}\left | \sum_{j=1}^{n} \xi_{j} e^{itY_{j}^{\dagger}/h_{n}} W_{j}^{+} \right |  dt = O_{\Pr}(n^{1/2} h_{n}^{-\alpha}). 
\end{align*}
Hence the first term on the right hand side of (\ref{eq: MB mid panel}) is $O_{\Pr} (h_{n}^{-\alpha-1}) = o \{ h_{n}^{-\alpha}(nh_{n})^{1/2} (\log h_{n}^{-1})^{-1/2} \}$. Finally, we shall show that 
\[
\left \| \sum_{j=1}^{n} \xi_{j} \{ \hat{K}_{n}((\cdot-Y_{j}^{\dagger})/h_{n}) - K_{n}((\cdot - Y_{j}^{\dagger})/h_{n}) \} \right \|_{I} = o_{\Pr} \{ h_{n}^{-\alpha}(nh_{n})^{1/2} (\log h_{n}^{-1})^{-1/2} \},
\]
but this follows from mimicking Step 2 in the proof of Theorem \ref{thm: multiplier bootstrap} using the bound (\ref{eq: ecf estimated}). This completes the proof. 
\qed

\subsection{Proofs for Section \ref{sec: supersmooth}}

We first point out that the expansion (\ref{eq: variance supersmooth}) holds uniformly in $x \in I$ under our assumption. This follows from the proof of Theorem 1.5 in \cite{EsUh05} and the observation that 
\begin{equation}
\Var (\cos ((x-Y)/h_{n})) \to \frac{1}{2} 
\label{eq: variance cosine}
\end{equation}
as $n \to \infty$ uniformly in $x \in I$ (in fact in $x \in \R$). To see that (\ref{eq: variance cosine}) holds uniformly in $x \in \R$, observe that 
\begin{align*}
 \cos^{2}((x-Y)/h_{n})  &= \frac{\cos (2(x-Y)/h_{n}) + 1}{2}\\
& = \frac{1}{2} + \frac{1}{2} \left \{ \cos (2x/h_{n}) \cos (2Y/h_{n}) + \sin (2x/h_{n}) \sin (2Y/h_{n}) \right \}.
\end{align*}
The Riemann-Lebesgue lemma yields that both $\Ep[ \cos (2Y/h_{n})]$ and $\Ep[ \sin (2Y/h_{n})]$ converge to $0$, and since the cosine and sine functions are bounded by $1$, we have that $\Ep[ \cos^{2}((x-Y)/h_{n})  ] \to 1/2$ uniformly in $x \in \R$. Likewise, we have that $\Ep [\cos ((x-Y)/h_{n})] \to 0$ uniformly in $x \in \R$.

Now, the proof of Theorem \ref{thm: supersmooth} is almost identical to the proofs of Theorems \ref{thm: Gaussian approximation} and \ref{thm: multiplier bootstrap} in the ordinary smooth case. The only changes that have to be taken into account are (\ref{eq: variance supersmooth}) and (\ref{eq: lower bound on chf supersmooth}), which imply that $\| K_{n}((x-\cdot)/h_{n})/\sigma_{n}(x) \|_{\R} \lesssim h_{n}^{-\gamma (1+\lambda)}$, for example. 
To avoid repetitions, we omit the details for brevity. In view of the proof of Corollary \ref{cor: multiplier bootstrap}, Corollary \ref{cor: supersmooth}  directly follows from Theorem \ref{thm: supersmooth}. 
\qed

%

%

\section{Uniform convergence rates of the empirical characteristic function}
\label{sec: appendix}

In this appendix, we establish rates of convergence of the empirical characteristic function on expanding sets. The proof of the following lemma is due essentially to \citet[][Theorem 4.1]{NeRe09}. 

Let $F$ be a distribution function on $\R$ with characteristic function $\varphi (t) = \int_{\R} e^{itx} dF(x)$, and let $X_{1},\dots,X_{n}$ be an independent sample from $F$. 
Let $F_{n}(x) = n^{-1}\sum_{j=1}^{n} 1_{(-\infty,x]}(X_{j})$ be the empirical distribution function, and let $\varphi_{n}(t) = \int_{\R} e^{itx} dF_{n}(x) = n^{-1} \sum_{j=1}^{n}e^{itX_{j}}$ be the empirical characteristic function. 

\begin{lemma}
\label{lem: rate ecf}
Suppose that $\int_{\R} |x|^{p} dF(x) < \infty$ for some $p > 0$. Then for any $\delta > 0$ and any $T_{n} \to \infty$, we have
\[
\| \varphi_{n} - \varphi \|_{[-T_{n},T_{n}]}  = O_{\Pr}\{ n^{-1/2} (\log T_{n})^{1/2+\delta}  \}.
\]
\end{lemma}

\begin{proof}
Let $w (t) = (\log (e+|t|))^{-1/2-\delta}$. According to Theorem 4.1 in \cite{NeRe09}, it follows that 
\[
C := \sup_{n \geq 1} \Ep [ \| \sqrt{n} (\varphi_{n} - \varphi) w \|_{\R} ] < \infty.
\]
Now, because 
\[
\| \sqrt{n} (\varphi_{n} - \varphi) w \|_{\R} \geq \sqrt{n} \| \varphi_{n} - \varphi \|_{[-T_{n},T_{n}]} \inf_{|t| \leq T_{n}} w(t),
\]
we conclude that 
\[
\Ep [ \| \varphi_{n} - \varphi \|_{[-T_{n},T_{n}]} ] \leq \frac{C}{\sqrt{n} \inf_{|t| \leq T_{n}} w(t)} = O\{ n^{-1/2} (\log T_{n})^{1/2+\delta}  \},
\]
which leads to the desired result by Markov's inequality. 
\end{proof}

It is worthwhile to point out that the restriction to the set $|t| \leq T_{n}$ in Lemma \ref{lem: rate ecf} is essential.
In fact, although the class of functions $\{ x \mapsto e^{itx} : t \in \R \}$ is uniformly bounded, it is in general not Glivenko-Cantelli (nor Donsker, of course). 
See \cite{FeMu77} for details. 

\section{Auction Data}

The source data for our empirical application can be obtained from the Center for the Study of Auctions, Procurements and Competition Policy hosted by Penn State University.
We pre-process bid values in this source data and obtain firms' values $(Y^{(1)},Y^{(2)})$ based on an equilibrium restriction (Bayesian Nash equilibrium) for the first-price sealed-bid auction mechanism -- we use the same procedure as the one used in \cite{LiPeVu00}. See also \cite{GuPeVu00}.
While the original sample consists of 217 tracts with two firms in each tract, we obtain 169 tracts with 2 firms in each tract as a result of trimming.
Figure \ref{fig:application_data} depicts a simple kernel density estimate of the values $(Y^{(1)},Y^{(2)})$ in the logarithm of US dollars per acre.
This figure essentially reproduces Figure 3 of \cite{LiPeVu00}.
Note that the value distribution is bimodal.

\begin{figure}
	\centering
		\includegraphics[width=0.75\textwidth]{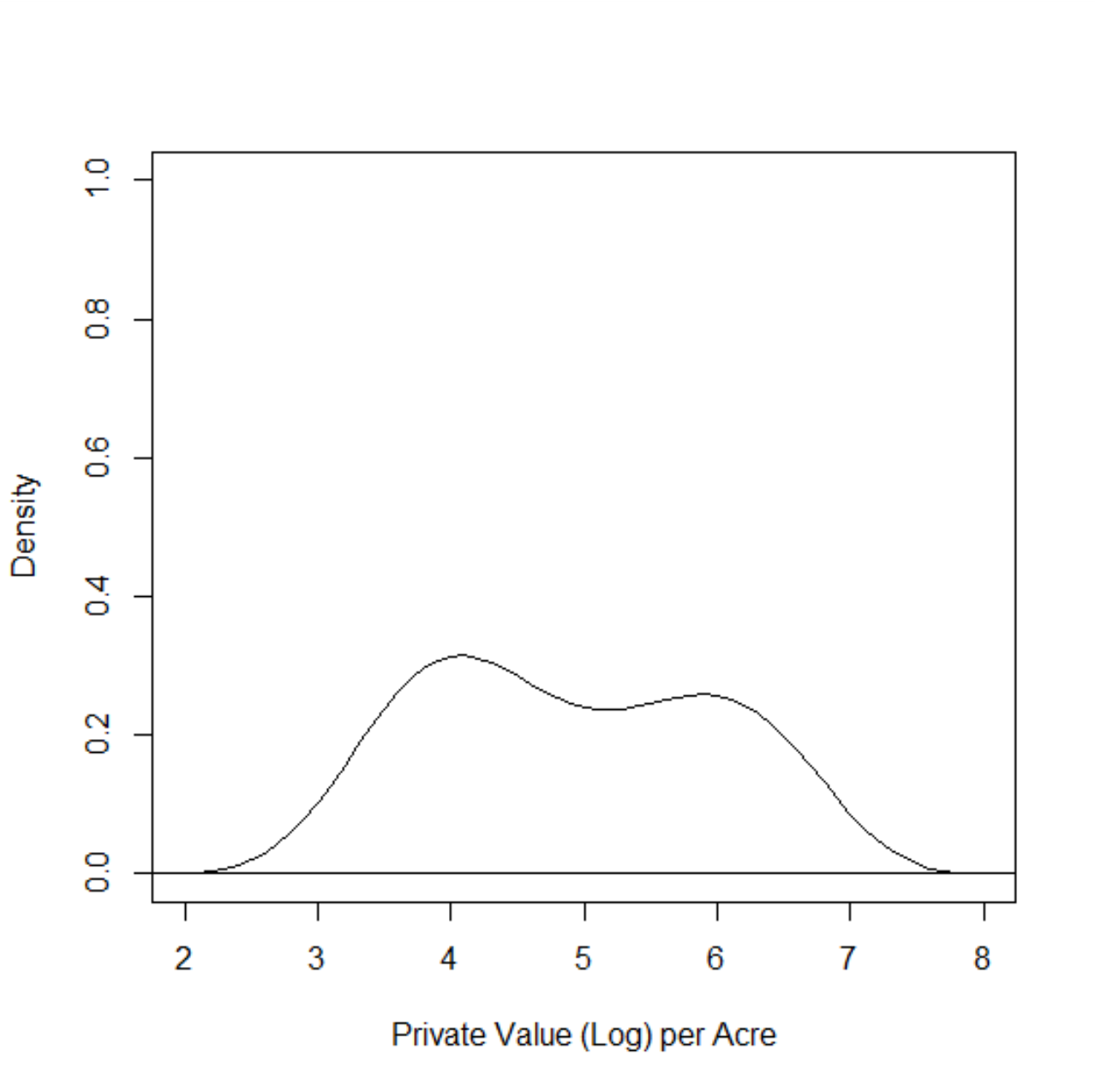}
	\caption{{\small A kernel density estimate of the values $(Y^{(1)},Y^{(2)})$ in the logarithm of US dollars per acre.}}
	\label{fig:application_data}
\end{figure}


\begin{thebibliography}{99}
\bibitem[Ackerberg et al.(2006)]{AcCaFr06}
Ackerberg, D.A., Caves, K., and Frazer, G. (2006). Structural identification of production functions. Unpublished manuscript.
\bibitem[Adusumilli et al.(2016)]{AdOtWh16}
Adusumilli, K., Otsu, T., and Whang Y.-J. (2016). Inference on distribution functions under measurement error. Unpublished manuscript.
\bibitem[Armstrong and Koles\'{a}r(2017)]{ArKo14}
Armstrong, T. and Kols\'{a}r, M. (2017). A simple adjustment for bandwidth snooping. \textit{Rev. Econom. Stud.}, forthcoming. 
\bibitem[Babii(2016)]{Ba16}
Babii, A. (2016). Honest confidence sets in nonparametric IV regression and other ill-posed models. arXiv:1611.03015.
\bibitem[Bickel and Rosenblatt(1973)]{BiRo73}
Bickel, P. and Rosenblatt, M. (1973). On some global measures of the deviations of density function estimates. \textit{Ann. Statist.} \textbf{1} 1071-1095. Correction (1975) \textbf{3} 1370.
\bibitem[Bissantz et al.(2007)]{BiDuHoMu07}
Bissantz, N., D\"{u}mbgen, L., Holzmann, H., and Munk, A. (2007). Non-parametric confidence bands in deconvolution density estimation. \textit{J. R. Stat. Soc. Ser. B. Stat. Methodol.} \textbf{69} 483-506.
\bibitem[Bissantz and Holzmann(2008)]{BiHo08}
Bissantz, N. and Holzmann, H. (2008). Statistical inference for inverse problems. \textit{Inverse Problems} 24:034009.
\bibitem[Bohnomme and Robin(2010)]{BoRo10}
Bohnomme, S. and Robin, J.-M. (2010). Generalized nonparametric deconvolution with an application to earnings dynamics. \textit{Rev. Econom. Stud.} \textbf{77} 491-533.
\bibitem[Bonhomme and Sauder(2011)]{BoSa11}
Bonhomme, S. and Sauder, U. (2011). Recovering distributions in difference-in-differences models: a comparison of selective and comprehensive schooling. \textit{Rev. Econ. Stat.} \textbf{93} 479-494.
\bibitem[Bourdaud et al.(2006)]{BoLaSi06}
Bourdaud, G., Lanza de Cristoforis, M., and Sickel, W. (2006). Superposition operators and functions of bounded $p$-variation. \textit{Rev. Mat. Iberoamericana} \textbf{22} 455-487.
\bibitem[Calonico et al.(2017)]{CaCaFa15}
Calonico, S., Cattaneo, M.D., and Farrell, M.H. (2017). On the effect of bias estimation on coverage accuracy in nonparametric inference. \textit{J. Amer. Stat. Assoc.}, forthcoming. 
\bibitem[Carroll and Hall(1988)]{CaHa88}
Carroll, R.J. and Hall, P. (1988). Optimal rates of convergence for deconvolving a density. \textit{J. Amer. Statist. Assoc.} \textbf{83} 1184-1186.
\bibitem[Carroll et al.(2006)]{CaRuStCr06}
Carroll, R.J., Ruppert, D., Stefanski, L.A., and Crainiceanu, C.M. (2006). \textit{Measurement Error in Nonlinear Models: A Modern Perspective (2nd Edition)}. Chapman \& Hall/CRC.
\bibitem[Chen and Christensen(2015)]{ChCh15}
Chen, X. and Christensen, T. (2015). Optimal sup-norm rates, adaptivity and inference in nonparametric instrumental variables estimation. arXiv:1508.03365.
\bibitem[Chen et al.(2011)]{ChHoNe11}
Chen X, Hong, H, and Nekipelov, D. (2011). Nonlinear models of measurement errors. \textit{J. Econ. Lit.} \textbf{49} 901-937.
\bibitem[Chernozhukov et al.(2014a)]{ChChKa14a}
Chernozhukov, V., Chetverikov, D., and Kato, K. (2014a). Gaussian approximation of suprema of empirical processes. \textit{Ann. Statist.} \textbf{42} 1564-1597.
\bibitem[Chernozhukov et al.(2014b)]{ChChKa14b}
Chernozhukov, V., Chetverikov, D., and Kato, K. (2014b). Anti-concentration and honest, adaptive confidence bands. \textit{Ann. Statist.} \textbf{42} 1787-1818.
\bibitem[Chernozhukov et al.(2015)]{ChChKa15}
Chernozhukov, V., Chetverikov, D., and Kato, K. (2015). Comparison and anti-concentration bounds for maxima of Gaussian random vectors. \textit{Probab. Theory Related Fields} \textbf{162} 47-70. 
\bibitem[Chernozhukov et al.(2016)]{ChChKa16}
Chernozhukov, V., Chetverikov, D., and Kato, K. (2016). Empirical and multiplier bootstraps for suprema of empirical processes of increasing complexity, and related Gaussian couplings. \textit{Stochastic Process. Appl.}, to appear. arXiv:1502:00352.
\bibitem[Claeskens and Van Keilegom(2003)]{Clva03}
Claeskens, G. and Van Keilegom, I. (2003). Bootstrap confidence bands for regression curves and their derivatives. \textit{Ann. Statist.} \textbf{31} 1852-1884.
\bibitem[Comte and Lacour(2011)]{CoLa11}
Comte, F. and Lacour, C. (2011). Data-driven density estimation in the presence of additive noise with unknown distribution. \textit{J. R. Stat. Soc. Ser. B. Stat. Methodol.} \textbf{73} 601-627.
\bibitem[Dattner et al.(2016)]{DaReTr16}
Dattner, I., Rei\ss, M., and Trabs, M. (2016). Adaptive quantile estimation in deconvolution with unknown error distribution. \textit{Bernoulli} \textbf{22} 143–192. 
\bibitem[Delaigle and Gijbels(2004)]{DeGi04}
Delaigle, A. and Gijbels, I. (2004). Practical bandwidth selection in deconvolution kernel density estimation. \textit{Comput. Statist. Data Anal.} \textbf{45} 249-267.
\bibitem[Delaigle and Hall(2016)]{DeHa16}
Delaigle, A. and Hall, P. (2016). Methodology for nonparametric deconvolution when the error distribution is unknown. \textit{J. R. Stat. Soc. Ser. B. Stat. Methodol.} \textbf{78} 231-252.
\bibitem[Delaigle et al.(2015)]{DeHaJa15}
Delaigle, A., Hall, P., and Jamshidi, F. (2015). Confidence bands in nonparametric errors-in-variables regression. \textit{J. R. Stat. Soc. Ser. B. Stat. Methodol.} \textbf{77} 149-169.
\bibitem[Delaigle et al.(2008)]{DeHaMe08}
Delaigle, A., Hall, P., and Meister, A. (2008). On deconvolution with repeated measurements. \textit{Ann. Statist.} \textbf{36} 665-685. 
\bibitem[Diggle and Hall(1993)]{DiHa93}
Diggle, P.J. and Hall, P. (1993). A Fourier approach to nonparametric deconvolution of a density estimate. \textit{J. Roy. Stat. Soc. Ser. B. Stat. Methodol.} \textbf{55} 523-531.
\bibitem[Dudley(2002)]{Du02}
Dudley, R.M. (2002). \textit{Real Analysis and Probability}. Cambridge University Press.
\bibitem[Efromovich(1997)]{Ef97}
Efromovich, S. (1997). Density estimation for the case of supersmooth measurement error. \textit{J. Amer. Stat. Assoc.} \textbf{92} 526-535.
\bibitem[van Es and Gugushvili(2008)]{EsGu08}
van Es, B. and Gugushvili, S. (2008). Weak convergence of the supremum distance for supersmooth kernel deconvolution. \textit{Statist. Probab. Lett.} \textbf{78} 2932-2938. 
\bibitem[van Es and Uh(2005)]{EsUh05}
van Es, B. and Uh, H.-W. (2005). Asymptotic normality of kernel-type deconvolution estimators. \textit{Scand. J. Statist.} \textbf{32} 467-483. 
\bibitem[Eubank and Speckman(1993)]{EuSp93}
Eubank, R.L. and Speckman, P.L. (1993). Confidence bands in nonparametric regression. \textit{J. Amer. Stat. Assoc.} \textbf{88} 1287-1301.
\bibitem[Fan(1991a)]{Fa91a}
Fan, J. (1991a). On the optimal rates of convergence for nonparametric deconvolution problems. \textit{Ann. Statist.} \textbf{19} 1257-1272.
\bibitem[Fan(1991b)]{Fa91b}
Fan, J. (1991b). Asymptotic normality for deconvolution kernel density estimators. \textit{Sankhya A} \textbf{53} 97-110.
\bibitem[Feurerverger and Mureika(1977)]{FeMu77}
Feurerverger, A. and Mureika, R.A. (1977). Empirical characteristic function and its applications. \textit{Ann. Statist.} \textbf{5} 88-97. 
\bibitem[Folland(1999)]{Fo99}
Folland, G.B. (1999). \textit{Real Analysis (2nd Edition)}. Wiley. 
\bibitem[Fuller(1987)]{Fu87}
Fuller, W.A. (1987). \textit{Measurement Error Models}. Wiley.
\bibitem[Gin\'{e} and Nickl(2009)]{GiNi09}
Gin\'{e}, E. and Nickl, R. (2009). Uniform limit theorems for wavelet density estimators. \textit{Ann. Probab.} \textbf{37} 1605-1646.
\bibitem[Gin\'{e} and Nickl(2016)]{GiNi16}
Gin\'{e}, E. and Nickl, R. (2016). \textit{Mathematical Foundations of Infinite-Dimensional Statistical Models}. Cambridge University Press.
\bibitem[Guerre et al.(2000)]{GuPeVu00}
Guerre, E., Perrigne, I., and Vuong, Q. (2000). Optimal nonparametric estimation of first-price auctions. \textit{Econometrica} \textbf{68} 525-574.
\bibitem[Hall(1991)]{Ha91}
Hall, P. (1991). On convergence rates of suprema. \textit{Probab. Theory Related Fields} \textbf{89} 447-455.
\bibitem[Hall and Horowitz(2013)]{HaHo13}
Hall, P. and Horowitz, J.L. (2013). A simple bootstrap method for constructing nonparametric
confidence bands for functions. \textit{Ann. Statist.} \textbf{41} 1892-1921.
\bibitem[Hendricks et al.(1987)]{HePoBo87}
Hendricks, K., Porter, R.H., and Boudreau, B. (1987). Information, returns, and bidding behavior in OCS auctions: 1954-1969. \textit{J. Indust. Econom.} \textbf{35} 517-542.
\bibitem[Horowitz(2009)]{Ho09}
Horowitz, J.L. (2009). \textit{Semiparamtric and Nonparametric Methods in Econometrics}. Springer. 
\bibitem[Horowitz and Lee(2012)]{HoLe12}
Horowitz, J. L. and Lee, S. (2012). Uniform confidence bands for functions estimated nonparametrically with instrumental variables. \textit{J. Econometrics} \textbf{168} 175-188.
\bibitem[Horowitz and Markatou(1996)]{HoMa96}
Horowitz, J.L. and Markatou, M. (1996). Semiparametric estimation of regression models for panel data. \textit{Rev. Econom. Stud.} \textbf{63} 145-168.
\bibitem[Johannes(2009)]{Jo09}
Johannes, J. (2009). Deconvolution with unknown error distribution. \textit{Ann. Statist.} \textbf{37} 2301-2323.
\bibitem[Koml\'{o}s et al.(1975)]{KoMaTu75}
Koml\'{o}s, J., Major, P., and Tusn\'{a}dy, G. (1975). An approximation for partial sums of independent rv's and the sample df I. {\em Z. Warhsch. Verw. Gabiete} \textbf{32} 111-131.
\bibitem[Krasnokutskaya(2011)]{Kr11} 
Krasnokutskaya, E. (2011).  Identification and estimation of auction models with unobserved heterogeneity. \textit{‎Rev. Econ. Stud.} \textbf{78} 293-327.
\bibitem[Leadbetter et al.(1983)]{LeLiRo83}
Leadbetter, M.R., Lindgren, G., and Rootz\'{e}n, H. (1983). \textit{Extremes and Related Properties of Random Sequences and Processes}. Springer. 
\bibitem[Levinsohn and Petrin(2003)]{LePe03}
Levinsohn, J. and Petrin, A. (2003) Estimating production functions using inputs to control for unobservables. \textit{Rev. Econ. Stud.} \textbf{70} 317-341.
\bibitem[Li et al.(2000)]{LiPeVu00}
Li, T., Perrigne, I., and Vuong, Q. (2000). Conditionally independent private information in OCS wildcat auctions. \textit{J. Econometrics} \textbf{98} 129-161. 
\bibitem[Li and Vuong(1998)]{LiVu98}
Li, T. and Vuong, Q. (1998). Nonparametric estimation of the measurement error model using multiple indicators. \textit{J. Multivariate Anal.} \textbf{65} 139-165.
\bibitem[Lounici and Nickl(2011)]{LoNi11}
Lounici, K. and Nickl, R. (2011). Global uniform risk bounds for wavelet deconvolution estimators. {\em Ann. Statist.} \textbf{39} 201-231.
\bibitem[McMurry and Politis(2004)]{McPo04}
McMurry, T.L. and Politis, D.N. (2004). 
Nonparametric regression with infinite order flat-top kernels. \textit{J. Nonparametric Statist.} \textbf{16} 549-562.
\bibitem[Meister(2009)]{Me09}
Meister, A. (2009). \textit{Deconvolution Problems in Nonparametric Statistics}. Springer. 
\bibitem[Neumann(1997)]{Ne97}
Neumann, M.H. (1997). On the effect of estimating the error density in nonparametric deconvolution. \textit{J. Nonparametric Statist.} \textbf{7} 307-330.
\bibitem[Neumann(2007)]{Ne07}
Neumann, M.H. (2007). Deconvolution from panel data with unknown error distribution. \textit{J. Multivariate Anal.} \textbf{98} 1955-1968.
\bibitem[Neumann and Rei\ss(2009)]{NeRe09}
Neumann, M.H. and Rei\ss, M. (2009). Nonparametric estimation for L\'{e}vy processes from low-frequency observations. \textit{Bernoulli} \textbf{15} 223-248.
\bibitem[Olley and Pakes(1996)]{OlPa96}
Olley, G.S. and Pakes, A. (1996) The dynamics of productivity in the telecommunications equipment industry. \textit{Econometrica} \textbf{64} 1263-1297.
\bibitem[Schennach(2013)]{Sc13}
Schennach, S.M. (2013). Convolution without independence. Cemmap working paper.
\bibitem[Schennach(2015)]{Sc15}
Schennach, S.M. (2015). A bias bound approach to nonparametric inference. Cemmap working paper  CWP71/15. 
\bibitem[Schennach(2016)]{Sc16}
Schennach, S.M. (2016). Recent advances in the measurement error literature. In: \textit{Annual Review of Economics}, Vol. 8, pp. 341-377. 
\bibitem[Schmidt-Hieber et al.(2013)]{ScMuDu13}
Schmidt-Hieber, J., Munk, A., and D\"{u}mbgen, L. (2013). Multiscale methods for shape constraints in deconvolution: confidence statements for qualitative features. {\em Ann. Statist.} \textbf{41} 1299-1328. 
\bibitem[Smirnov(1950)]{Sm50}
Smirnov, N.V. (1950). On the construction of confidence regions for the density of distribution of random variables. \textit{Doklady Akad. Nauk SSSR} \textbf{74} 189-191 (Russian).
\bibitem[Stefanski and Carroll(1990)]{StCa90}
Stefanski, L. and Carroll, R.J. (1990). Deconvoluting kernel density estimators. \textit{Statistics} \textbf{21} 169-184.
\bibitem[van der Vaart and Wellner(1996)]{vaWe96}
van der Vaart, A.W. and Wellner, J.A. (1996). {\em Weak Convergence and Empirical Processes: With Applications to Statistics}. Springer.
\bibitem[Wooldridge(2009)]{Wo09}
Wooldridge, J.M. (2009) On estimating firm-level production functions using proxy variables to control for unobservables. \textit{Economics Lett.} \textbf{104} 112-114.
\bibitem[Xia(1998)]{Xi98}
Xia, Y. (1998). Bias-corrected confidence bands in nonparametric regression. \textit{J. R. Stat. Soc. Ser. B Stat. Methodol.} \textbf{60} 797-811.
\end{thebibliography}
\end{document}